\pdfoutput=1
\RequirePackage{ifpdf}
\ifpdf 
\documentclass[pdftex]{sigma}
\else
\documentclass{sigma}
\fi

\makeatletter
\newcommand*{\rom}[1]{\expandafter\@slowromancap\romannumeral #1@}
\makeatother

\usepackage[bf,labelsep=period]{caption}
\usepackage{cases}

\numberwithin{equation}{section}

\newtheorem{thm}{Theorem}[section]
\newtheorem*{Theorem*}{Theorem}

\newtheorem{lmm}[thm]{Lemma}
\newtheorem{prp}[thm]{Proposition}
\newtheorem{mresult}[thm]{Main Result}
 {\theoremstyle{definition}

\newtheorem{rem}[thm]{Remark}

\newtheorem{defn}[thm]{Definition} }

\begin{document}

\newcommand{\arXivNumber}{2312.10759}

\renewcommand{\PaperNumber}{012}

\FirstPageHeading

\ShortArticleName{Counting Curves with Tangencies}

\ArticleName{Counting Curves with Tangencies}

\Author{Indranil BISWAS~$^{\rm a}$, Apratim CHOUDHURY~$^{\rm b}$, Ritwik MUKHERJEE~$^{\rm c}$\newline and Anantadulal PAUL~$^{\rm d}$}

\AuthorNameForHeading{I.~Biswas, A.~Choudhury, R.~Mukherjee and A.~Paul}

\Address{$^{\rm a)}$~Department of Mathematics, Shiv Nadar University,\\
\hphantom{$^{\rm a)}$}~NH91, Tehsil Dadri, Greater Noida, Uttar Pradesh 201314, India}
\EmailD{\href{indranil.biswas@snu.edu.in}{indranil.biswas@snu.edu.in}, \href{indranil29@gmail.com}{indranil29@gmail.com}}

\Address{$^{\rm b)}$~Institut f{\"u}r Mathematik, Humboldt-Universit{\"a}t zu Berlin,\\
\hphantom{$^{\rm b)}$}~Unter den Linden 6, Berlin 10099, Germany}
\EmailD{\href{mailto:apratim.choudhury@hu-berlin.de}{apratim.choudhury@hu-berlin.de}}

\Address{$^{\rm c)}$~School of Mathematical Sciences, National Institute of Science Education and Research,\\
\hphantom{$^{\rm c)}$}~Bhubaneswar, An OCC of Homi Bhabha National Institute, Khurda 752050, Odisha, India}
\EmailD{\href{mailto:ritwikm@niser.ac.in}{ritwikm@niser.ac.in}}

\Address{$^{\rm d)}$~International Center for Theoretical Sciences, Survey No.~151, Hesaraghatta,\\
\hphantom{$^{\rm d)}$}~Uttarahalli Hobli, Sivakote, Bangalore 560089, India}
\EmailD{\href{mailto:anantadulal.paul@icts.res.in}{anantadulal.paul@icts.res.in}, \href{mailto:paulanantadulal@gmail.com}{paulanantadulal@gmail.com}}

\ArticleDates{Received May 04, 2024, in final form February 07, 2025; Published online February 23, 2025}

\Abstract{Interpreting tangency as a limit of two transverse intersections, we obtain a~concrete formula to enumerate smooth degree $d$ plane curves tangent to a given line at multiple points with arbitrary order of tangency. Extending that idea, we then enumerate curves with one node with multiple tangencies to a given line of any order. Subsequently, we enumerate curves with one cusp, that are tangent to first order to a given line at multiple points. We~also present a new way to enumerate curves with one node; it is interpreted as a degeneration of a curve tangent to a given line. That method is extended to enumerate curves with two nodes, and also curves with one tacnode are enumerated. In the final part of the paper, it is shown how this idea can be applied in the setting of stable maps and perform a concrete computation to enumerate rational curves with first-order tangency. A~large number of low degree cases have been worked out explicitly.}


\Classification{14N35; 14J45; 53D45}

\section{Introduction}
A prototypical question in enumerative geometry is as follows: what is the characteristic number of
curves in a linear system that have certain
prescribed singularities and are tangent to a given divisor of various orders at multiple points?
The curves are, of course, required to meet further insertion conditions so that the ultimate answer is a finite
number. For example,
in $\mathbb{P}^2$ there are exactly $2$ conics passing through $4$ generic points that are tangent to a given line, and there
are~36 nodal cubics in $\mathbb{P}^2$ through $7$ generic points tangent to a given line.
These are some special cases of the famous Caporaso--Harris formula~\cite{CH}, which addresses the following question:
How many degree $d$ curves are there in $\mathbb{P}^2$ that pass through
$\mathsf{j}$ generic points, having $\delta$ nodes that are
tangent to a given line at $r$ distinct points,
with the orders of tangency being $k_1, k_2,\dots, k_r$, where
\[ \mathsf{j} := \frac{d(d+3)}{2} -(\delta +(k_1+k_2+ \dots + k_r)). \]
We have defined here the order of tangency to be $k$ if the order of contact is $k+1$ (i.e., the intersection multiplicity of the curve and the line is $k+1$, hence the transversal intersection is of contact order $1$).

The purpose of this paper is to give a new way to think about the question of tangency.
The main idea can be summarized in one picture:
This interpretation of tangency allows us to
effortlessly enumerate smooth curves with multiple tangencies of any order.
With a little more effort, this method can also be applied to enumerate
$1$-nodal curves with tangencies (henceforth, curves with $\delta$ nodes are referred
to as $\delta$-nodal curves). We then go on to show that the method
can also be applied to enumerate $1$-cuspidal curves (i.e., curves with one cusp)
with first-order tangencies at multiple points.
We note that the Caporaso--Harris formula counts curves with only nodal singularities.

Our idea can also be applied to
study the question of enumerating stable maps tangent to a given divisor.
The difficulty of applying this idea in the context of stable maps
has been discussed by Gathmann in his paper \cite[p.~41]{Gath_blow_up}.
Subsequently, the idea has also been discussed in the more recent paper by Dusa McDuff and Kyler Siegel
\cite[pp.~1179--1180]{MK_published}.
The discussion in \cite{MK_published} illustrates that for applying this idea in
the context of stable maps to enumerate curves with tangencies, we need to compute the characteristic number
of curves with an $m$-fold point; \cite{AiM_m_fold_pt}~precisely does the latter. Hence, using the
results of \cite{AiM_m_fold_pt} and by interpreting tangency as a~limit of points lying on a line (as
illustrated by Figure~\ref{kk_pic}), it is possible to enumerate stable maps with tangencies.
In Section \ref{count_st_mp_tang}, a very concrete computation based on this idea is worked out.\looseness=-1

\begin{figure}[h]\centering
\includegraphics[scale = .85]{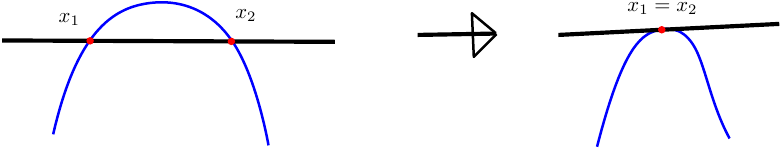}

\caption{Tangency as a limit of transverse intersection.} \label{kk_pic}
\end{figure}

Finally, in Section \ref{Rel_GW_WDVV} this idea is pursued again, but with a difference. We make the points in the
\textit{domain} come together, not in the \textit{image}. This directly gives us rational curves with tangencies.
The idea is implemented using
the equality
of divisors in $\overline{M}_{0,4}$.
In this way, we are able to count rational curves with first-order tangencies
without making use of the psi classes in any way;
only a knowledge of primary Gromov--Witten invariants
is required.

\section{Main results}
\label{mr_s}
We now state the main results of this paper.
Let $k_1, k_2, \dots, k_n$ be positive integers
and $d$ be an integer greater than or equal to
$k+n$,
where $k:= k_1 +\dots + k_n$.
Define
$
\mathsf{N}_d(\mathsf{T}_{k_1} \dots \mathsf{T}_{k_n})
$
to be the number of degree $d$ curves in $\mathbb{P}^2$, passing through the requisite number
of generic points, and that are tangent to a given line at $n$ distinct points with orders $k_1, k_2, \dots, k_n$. Note that aside from these $n$ points of contact,
the curve also intersects the line transversally at~${d-(k+n)}$ further points.
Define
\[
\mathsf{N}_d(\mathsf{A}_1 \mathsf{T}_{k_1} \dots \mathsf{T}_{k_n}) \qquad \textnormal{and} \qquad
\mathsf{N}_d(\mathsf{A}_2 \mathsf{T}_{k_1} \dots \mathsf{T}_{k_n})
\]
to be the number of $1$-nodal (respectively $1$-cuspidal)
degree $d$ curves in $\mathbb{P}^2$, passing through the requisite number
of generic points and that are tangent to a given line at $n$ distinct points with orders $k_1, k_2,\dots, k_n$.
The first main result of this paper is as follows.
\begin{mresult}
\label{m_rslt1}
By interpreting tangency as a limit of transverse intersections,
we are able to compute
\[\mathsf{N}_d(\mathsf{T}_{k_1} \dots \mathsf{T}_{k_n}),
\qquad \mathsf{N}_d(\mathsf{A}_1 \mathsf{T}_{k_1} \dots \mathsf{T}_{k_n})
\qquad \textnormal{and} \qquad
\mathsf{N}_d(\mathsf{A}_2 \underbrace{\mathsf{T}_{1} \dots \mathsf{T}_{1}}_{n\text{-times}}),\]
provided $d>n+k-1$, $d>n+k$ and $d>2n+2$, respectively.
\end{mresult}
Denote by $\mathsf{N}_d(\mathsf{A}_1)$ (respectively, $\mathsf{N}_d(\mathsf{A}_3)$)
the number of $1$-nodal (respectively, $1$-tacnodal)
plane degree $d$ curves, passing through the right number of generic points.
Also, denote $\mathsf{N}_d(\mathsf{A}_1 \mathsf{A}_1)$ by the number
of $2$-nodal
plane degree $d$ curves, passing through the right number of generic points.
Our next result is as follows.

\begin{mresult}
\label{m_rslt2}
By interpreting a nodal point as a degeneration of a first-order tangency,
we obtain a new method to compute
$\mathsf{N}_d(\mathsf{A}_1) $ and $
\mathsf{N}_d(\mathsf{A}_1\mathsf{A}_1)$,
provided $d>1$ and $d>3$, respectively.
Furthermore, by interpreting a tacnode as a limit of two nodal points,
we obtain a new method to compute $\mathsf{N}_d(\mathsf{A}_3)$, provided $d>3$.
\end{mresult}
To see what we mean by a node arising as a degeneration
of a tangency point, consider the polynomial $f_t(x,y):= ty + y^2 - x^2$.
We note that the zero set of this curve is tangent to the $x$-axis at the origin.
The origin is also a smooth point of the curve when $t\neq 0$. However,
in the limit when $t$
goes to zero, we see that this becomes a nodal curve.

Our final result is on counting stable maps (modulo automorphisms of the domain).
Let~$\mathsf{N}_d^{\mathsf{T}_1}$ denote the number of
rational degree $d$ stable maps (modulo automorphisms) passing through $3d-2$
generic points and
that are tangent to a given line (to first order).
Our third result is as follows.\looseness=-1

\begin{mresult}\label{m_rslt3}
By interpreting tangency as a limit of transverse intersections, we obtain two
new ways to compute the number $\mathsf{N}_d^{\mathsf{T}_1}$. The first method
consists of
moving the two points on the images of the stable maps, while
the second method consists of moving the two points
in the domain of the stable map.
\end{mresult}

The bound we impose on $d$ for Main Results \ref{m_rslt1} and \ref{m_rslt2}
is not optimal. It is a sufficient condition for our formula
to be valid, it is not necessary. The bound is imposed to prove
transversality, which involves constructing curves satisfying certain conditions.
If $d$ is sufficiently large, then one can easily construct such curves.
In Section~\ref{verification}, we do several low degree checks.
We observe in that section that even when $d$ is smaller than what is
required by our result, the values agree with the expected values.

We have written a \textsc{Mathematica} program
to implement the formulas in Sections \ref{smooth_tangencies}, \ref{nodal_tang}, \ref{Cuspidal_tang},
\ref{singularity_in_our_format} and~\ref{count_st_mp_tang}.
Also a \textsc{Python} program is written to implement the Caporaso--Harris formula.\footnote{The programs
are available at \url{https://www.sites.google.com/site/ritwik371/home}.
The reader is invited to use the program and verify the assertions.}

\section{Comparison with other methods}

Let us compare our method with the Caporaso--Harris formula.
The central idea of the Capora\-so--Harris formula is a degeneration argument,
which can briefly be described as follows.
They argue that the
number of degree $d$ plane curves with prescribed tangencies
and passing through~$\mathsf{j}$ generic points is equal to
the number of plane degree $d$ curves with the same prescribed tangencies
passing through~$\mathsf{j}-1$ generic points and also
intersecting the line transversally at one more fixed point,
plus an excess contribution.
The excess contribution comes from
\begin{itemize}\itemsep=0pt
\item Degree $d$ curves with the same prescribed tangencies, but where one of the free tangency
points gets replaced by a fixed tangency point.
\item Curves of lower degree with possibly higher order tangencies.
\end{itemize}
The total excess contribution is a sum over all these terms.

This will be explained more precisely by working out the following
question: How many degree $d$ curves
are there in $\mathbb{P}^2$, that pass through $\delta_d-1$ generic points and are
tangent to a~given line, where $\delta_d := \frac{d(d+3)}{2}$.
After we explain how Caporaso and Harris solve this problem, we then
explain how we solve this question in this paper.
This one example clearly illustrates
the difference between the two methods.

We make a few definitions.
First of all, define
$\mathsf{N}_d(\mathsf{T}_1, m)$ to be the number of
degree $d$ curves, passing through $\delta_d-1-m$ generic points
that are tangent to first order to a given line at an \emph{unspecified} point
and intersecting the line transversally at $m$ \emph{fixed} points.
Notice that requiring the curve to intersect the line at a
\emph{fixed} point is what imposes a genuine constraint.
Requiring the curve to intersect the line transversally at some \emph{unspecified}
point would not impose any condition (because any curve does intersect a line somewhere).
Our goal is to compute $\mathsf{N}_d(\mathsf{T}_1, 0)$
(which we often abbreviate as $\mathsf{N}_d(\mathsf{T}_1)$).

Next, we define
$\mathsf{N}_d\bigl(\mathsf{T}_1^{\textnormal{pt}}, m\bigr)$ to be the
number of
degree $d$ curves, passing through $\delta_d-2-m$ generic points
that are tangent to first order to a given line at a
fixed point and that also intersects the line at $m$ fixed points.
As before, we often abbreviate
$\mathsf{N}_d\bigl(\mathsf{T}_1^{\textnormal{pt}}, 0\bigr)$ as
$\mathsf{N}_d\bigl(\mathsf{T}_1^{\textnormal{pt}}\bigr)$.
Also, define $\mathsf{N}_d(\mathsf{S})$ to be the number of
degree $d$ curves passing through $\delta_d$ generic
points. Of course, this number is equal to one, but never the less
denote it by the symbol $\mathsf{N}_d(\mathsf{S})$, which make the geometric
ideas behind the subsequent formulas more transparent.
A special case of the Caporaso--Harris formula \cite[p.~348, Theorem 1.1]{CH}
is as follows:
\begin{gather}
\mathsf{N}_d(\mathsf{T}_1, m)  =
2 \mathsf{N}_d\bigl(\mathsf{T}_1^{\textnormal{pt}}, m\bigr) +
\mathsf{N}_d(\mathsf{T}_1, m+1) \qquad \forall  d \geq m+2 \qquad \textnormal{and} \label{CH1}\\
\mathsf{N}_d\bigl(\mathsf{T}_1^{\textnormal{pt}}, m\bigr) =
\mathsf{N}_d\bigl(\mathsf{T}_1^{\textnormal{pt}}, m+1\bigr)  \qquad \forall  d \geq m+3.
\label{CH2}
\end{gather}
Furthermore,
\begin{align}
\mathsf{N}_d\bigl(\mathsf{T}_1^{\textnormal{pt}}, d-2\bigr) & = \mathsf{N}_{d-1}(\mathsf{S})
\qquad \forall  d \geq 2. \label{CH3}
\end{align}
Using the fact that $\mathsf{N}_d(\mathsf{S})=1$, equations \eqref{CH1}, \eqref{CH2}
and \eqref{CH3}
recursively give us the value of~$\mathsf{N}_d(\mathsf{T}_1)$ for all $d\geq 2$.

This recursive formula will be analysed.
The structure of the formula is as follows: one gets
a recursive formula for $\mathsf{N}_d(\mathsf{T}_1)$ in terms of
$\mathsf{N}_{d-1}(\mathsf{T}_1)$. Furthermore, the underlying geometric principle is as follows:
the point constraints that are imposed on the curves (to get a finite number) are one by one moved to
the line and the corresponding numerical invariants are compared. As we keep moving the points
on the line, the curve is forced to break into a~reducible curve (one of whose components is the line)
and a~curve of lower degree. Unwinding this geometric phenomenon in terms of numbers ultimately
gives us equations \eqref{CH1}, \eqref{CH2} and~\eqref{CH3}.

Next, we explain how Caporaso--Harris enumerate curves with one node.
We make a few definitions first. Define
$\mathsf{N}_d(\mathsf{A}_1, m)$ to be the number of
degree $d$ curves with one node, passing through $\delta_d-1-m$ generic points and
$m$ generic points on the line.
Our goal is to compute~$\mathsf{N}_d(\mathsf{A}_1, 0)$
(which as before, is abbreviated as $\mathsf{N}_d(\mathsf{A}_1)$).
With these notations, a special case of the Caporaso--Harris
formula \cite[p.~348, Theorem 1.1]{CH}
is as follows:
\begin{gather}
\mathsf{N}_d(\mathsf{A}_1) =
\mathsf{N}_d(\mathsf{A}_1, m) \qquad \forall  d \geq m+1, \label{CH1_Node}\\
\mathsf{N}_d(\mathsf{A}_1, d-1) = \mathsf{N}_d(\mathsf{A}_1, d) + (d-1)\mathsf{N}_{d-1}(\mathsf{S})
\qquad \textnormal{and} \label{CH2_Node} \\
\mathsf{N}_d(\mathsf{A}_1, d) = \mathsf{N}_{d-1}(\mathsf{A}_1) + d \mathsf{N}_{d-1}(\mathsf{S}) + 2 \mathsf{N}_{d-1}(\mathsf{T}_1).
\label{CH3_Node}
\end{gather}
Equations \eqref{CH1_Node}, \eqref{CH2_Node}
and \eqref{CH3_Node} give us that
\begin{align}
\mathsf{N}_d(\mathsf{A}_1) & =  (2d-1) + \mathsf{N}_{d-1}(\mathsf{A}_1) + 2 \mathsf{N}_{d-1}(\mathsf{T}_1). \label{na1_CH}
\end{align}
Since $\mathsf{N}_{d-1}(\mathsf{T}_1)$ can be computed,
equation \eqref{na1_CH} enables us to compute $\mathsf{N}_{d}(\mathsf{A}_1)$.

Let us again try to understand the underlying geometric reason behind the formulas.
Equation \eqref{CH1_Node} says that the number of $1$-nodal
degree $d$ curves
passing through $\delta_d-1$ generic points
is the same as the number of $1$-nodal degree $d$ curves
through $\delta_d-1-(d-1)$ generic points
and~${d-1}$ points on a line (apply the equation for $m=d-1$, for which it is valid).
Equation~\eqref{CH2_Node} on the other hand says that
once we move another point on the line, the two numbers are no longer the same.
There is an excess contribution from a curve of
degree $d-1$ (because the curve can now break).
The two numbers $\mathsf{N}_d(\mathsf{A}_1, d-1)$ and
$\mathsf{N}_d(\mathsf{A}_1, d)$ are related via equation~\eqref{CH2_Node}.
Finally, equation~\eqref{CH3_Node} gives us a formula for
$\mathsf{N}_d(\mathsf{A}_1, d)$ in terms of counts of
curves of lower degree.

More generally, consider the computation of the number
$\mathsf{N}_d\bigl(\mathsf{A}_1^{\delta}\mathsf{T}_{k_1} \dots \mathsf{T}_{k_n}\bigr)$,
the number of $\delta$-nodal
degree $d$ curves tangent to a given line at $n$ points, with order of tangency being~${k_1, k_2, \dots, k_n}$ (and also passing through the right number of generic points).
The Caporaso--Harris formula give a recursive formula for this number in terms of
$\mathsf{N}_{d^{\prime}}\bigl(\mathsf{A}_1^{\delta^{\prime}}\mathsf{T}_{l_1} \dots \mathsf{T}_{l_m}\bigr)$,
where $d^{\prime} \leq d$ and $\delta^{\prime} \leq \delta$.
The geometric principle based on which
this is obtained is the same: the point constraints are moved to the line and as a result the curve
breaks, giving contributions to the intersection from curves of lower degree.

We now briefly explain our method (details are of course worked out in the subsequent sections).
Consider the question of computing
$\mathsf{N}_d(\mathsf{T}_1)$. We view this as an intersection number in the ambient space
$\mathcal{D}_1 \times \mathcal{D}_d \times \mathbb{P}^2 \times \mathbb{P}^2$,
where $\mathcal{D}_1$ and $\mathcal{D}_d$ refer to the space of lines
and degree~$d$ curves in $\mathbb{P}^2$ respectively.
On this ambient space, we define the subspace of degree $d$ curves and a line and two
marked points, such that the two marked points lie on the line and the curve.
We find an expression for the homology class represented by the closure of
this space. Now we impose the
condition that the two marked points come together (this is where we use the collision lemma).
That gives us the space of curves tangent to a given line. Finally, we intersect it with the
point constraints that are necessary to get a finite number. That gives us the number~$\mathsf{N}_d(\mathsf{T}_1)$.

Next, consider the question of computing
$\mathsf{N}_d(\mathsf{A}_1)$. The underlying idea is as follows. We view this
as an intersection number on
$\mathcal{D}_1 \times \mathcal{D}_d \times \mathbb{P}^2$.
On this ambient space, we define the subspace of degree $d$ curves and a line and one
marked point, such that the curve is tangent to the line at the marked point.
On the closure of this space, impose the condition that the directional
derivative of the polynomial defining the curve in the normal direction to
the line vanishes. That gives us the subspace of line, a curve and a
marked point, such that the marked point is a nodal point of the curve
and it lies on the line. In order to compute
$\mathsf{N}_d(\mathsf{A}_1)$, we simply make the curve pass through $\delta_d-1$
generic points, and make the line pass through one point (the second point through
which the line will pass is the nodal point of the curve).

We now make a few remarks. First and foremost, the underlying geometric mechanism
that governs our formula is completely different from what Caporaso--Harris are doing.
Our formula for $\mathsf{N}_d(\mathsf{T}_1)$
is governed by the principle that when two transverse points coincide, they
become a point of tangency.
A generalization of that principle is used in the computation of
$\mathsf{N}_d(\mathsf{T}_{k_1} \dots \mathsf{T}_{k_n})$.
The Caporaso--Harris formula is based on the principle that when
the point constraints of an enumerative problem are repeatedly
moved to a divisor (a line in this case), the
curve eventually has to break.

The second point is as follows: our formula is not a recursion on $d$.
We directly get a formula in terms of $d$. Our formula does involve recursion on other
quantities, such as the number of points of tangency and the orders of tangency,
but not the degree of the curve.
It is therefore completely straightforward to see that
our formula produces a polynomial in $d$ (this is
explained clearly at the end of Section \ref{smooth_curves_enum}).
It is however far from obvious that a recursion on $d$ produce a~polynomial in $d$.
To see why that is so, write down the
full Caporaso--Harris formula \cite[p.~348, Theorem 1.1]{CH}
\begin{align}
N^{d, \delta}(\alpha, \beta)={}& \sum_{k:\beta_k>0} k N^{d, \delta}(\alpha+e_k, \beta-e_k) \nonumber \\
              & + \sum I^{\beta-\beta^{\prime}} \binom{\alpha}{\alpha^{\prime}}
                \binom{\beta^{\prime}}{\beta} N^{d-1, \delta^{\prime}}
                (\alpha^{\prime}, \beta^{\prime}),
                \label{CH_rewrite}
\end{align}
where the second sum is taken over all $\alpha^{\prime}$, $\beta^{\prime}$
and $\delta^{\prime}$, such that
\[\alpha^{\prime} \leq \alpha, \qquad
\beta \leq \beta^{\prime}, \qquad
\delta^{\prime} \leq \delta \qquad \textnormal{and}
\qquad \delta-\delta^{\prime}+ |\beta-\beta^{\prime}| = d-1.\]
The reader can refer to \cite{CH} for the
relevant notation; for the purpose of this discussion only the structure of the
formula is important.
From the structure of the formula, we can see that it is far from clear that
the characteristic numbers
$N^{d, \delta}(\alpha, \beta)$ are polynomials in $d$ is because
they are expressed recursively in terms of characteristic numbers of lower degree.
Seeing the formula~\eqref{CH_rewrite}, we can not rule out the possibility that
the characteristic numbers
behave as factorial or exponential functions of $d$.
In light of
G{\"o}ttsche like conjectures (which broadly speaking says that the solutions to enumerative problems
involving degree $d$ curves is given by universal polynomials), this is perhaps an
interesting point.

The final difference we would like to mention is that we enumerate curves with cuspidal
singularity that are tangent to a line at multiple points. As of now, Caporaso--Harris
has only been worked out for nodal curves with tangencies.
However, in our paper \cite{IB_RM_AC_ND_arxiv}, we have shown how to extend the
idea of Caporaso--Harris to enumerate curves with one cusp.

Finally, we mention another approach to enumerate curves with tangencies,
which is taken by the fourth author in
\cite{Anant-Thesis,PAUL2024103418}.
Again consider the computation of $\mathsf{N}_d(\mathsf{T}_1)$.
Instead of viewing tangency as a limit of two transverse intersections,
only one point is considered. The geometric
condition imposed here is that the
directional derivative of the function, defining the
curve, along the direction of the line is zero. This can be suitably interpreted as the
Euler class of a~bundle.
Using this approach, the author has been able to compute the
characteristic number of singular curves, tangent to first order to a given line at one point.

\section{A few remarks on intersection theory}

In this section, we summarize a few facts about intersection theory.
The ambient space $\mathsf{M}$
inside which intersection theory is done will be a product
of projective spaces. Whenever we talk about an open set,
we mean open set with the analytic topology of the projective space;
we are not talking about the Zariski topology.

Let $\alpha$ and $\beta$ be two homology classes
in $\mathsf{M}$, i.e., $\alpha, \beta \in H_*(\mathsf{M}, \mathbb{R})$.
We define $\alpha\cdot \beta$ the \emph{topological intersection} of
$\alpha$ and $\beta$ to be the unique homology class whose Poincar\'{e} dual is the
cup product of the Poincar\'{e} duals of $\alpha$ and $\beta$.

We now follow a standard abuse of notation.
Given a homology class $\alpha$, we denote its Poincar\'{e} dual (a cohomology class)
by the same letter $\alpha$.
Hence, from the point of view of cohomology,
topological intersection is simply the cup product.

The homology classes that we encounter are going to be obtained
by taking the closure of certain algebraic varieties. More precisely,
our situation is as follows: we have a set $\mathsf{S}$, which is
a smooth complex submanifold of the ambient space $\mathsf{M}$, but it is
not closed. The closure of $\mathsf{S}$ (inside $\mathsf{M}$)
need not be smooth. Nevertheless, $\overline{\mathsf{S}}$ does define a homology class.
The reason for this is as follows. The singularities of the closure are of
complex codimension one and hence, of real codimension two. Hence, by Stokes theorem
for analytic varieties
\cite[p.~33]{GH3}
integration over $\overline{\mathsf{S}}$ makes sense.
Hence, any analytic subvariety of a compact complex manifold $\mathsf{M}$, always
defines a homology class in $H_*(\mathsf{M}, \mathbb{R})$ as explained in
\cite[pp.~33 and 61]{GH3}. We often refer to~$\mathsf{S}$ as the
\emph{open part} of the cycle $\overline{\mathsf{S}}$.

We often be dealing with zero sets of sections of certain complex vector bundles.
It is a~standard fact that when the section is transverse to zero, the cycle represented
by the zero set is Poincar\'{e} dual to the top \emph{Chern class} of the vector bundle.
The top Chern class is also referred to as the \emph{Euler class} of the bundle; we
usually use the latter terminology.

\section{The collision lemma}
\label{coll_lemm}

This section contains the central lemma that we will be using throughout the paper, which is
be referred to as the collision lemma. Before stating the lemma, we need to introduce a few notations.

Let $\mathcal{D}_1$ be the space of lines in $\mathbb{P}^2$; it is the dual projective space $\check{\mathbb{P}}^2$. Define
$
M := \mathcal{D}_1 \times X_{1} \times X_{2}$,
where $X_i$ is a copy of $\mathbb{P}^2$ for $i = 1, 2$.
The pullback to $M$ of the hyperplane classes in
$X_{1}$, $X_{2}$ and $\mathcal{D}_1$ are denoted by $a_1$, $a_2$ and $y_1$, respectively. Define
\begin{align*}
X := \{ (H, q_1, q_2) \in M  \mid  q_1, q_2 \in H\}
\qquad \text{and}
\qquad
Y := \{ (H, q_1, q_2) \in X  \mid  q_1= q_2\}.
\end{align*}
\begin{lmm}[{collision lemma}]\label{CL_ver2}
The cohomology class $(a_1+a_2-y_1)$ restricted to $X$ is equal to
the Poincar\'{e} dual of $Y$ in $X$.
\end{lmm}

\begin{rem}
 Let $\Delta_{1, 2}^{\mathsf{L}}$
be the following divisor class on $M$
\begin{align}
\Delta_{1, 2}^{\mathsf{L}}& := a_1 + a_2-y_1. \label{Line_Diag_Defn}
\end{align}
The geometric content of the collision lemma
is that the intersection of $X$ with the class $\Delta_{1, 2}^{\mathsf{L}}$ is equivalent to the class
obtained by making the two points come together in $X$. \end{rem}

\begin{proof} First of all, we note that the cohomology of $X$
is generated by $a_1$, $a_2$ and $y_1$; this follows from the Leray--Hirsch theorem.
Since $Y$ is a codimension one submanifold of $X$, we conclude that the
Poincar\'{e} dual of $Y$ in $X$ is given by
\begin{align}
\textnormal{PD}_{X}[Y] = (A a_1 + B a_2 + C y_1)|_{X}, \label{k1}
\end{align}
for some numbers $A$, $B$ and $C$ that are to be determined.
To prove the lemma, it suffices to show that $A = 1$, $ B = 1$ and $C = -1$.
This, we show by computing three different intersection numbers.

First of all, we note that the Poincar\'{e} dual of $X$ in $M$ is given by
\begin{align}
\textnormal{PD}_{M}[X]& = (y_1+a_1)\cdot(y_1+a_2). \label{k2}
\end{align}
To see why this is so, first define $Z_1$ as
$
Z_1 := \{ (H, q_1) \in \mathcal{D}_1 \times X_1  \mid  q_1 \in H\}$.
To prove \eqref{k2}, it suffices to show that $[Z_1] = (y_1 + a_1)$.
To justify this, note that
$[Z_1] = n y_1 + m a_1$ for some numbers $n$ and $m$. Next, we note that
$n = [Z_1]\cdot y_1 a_1^2$ and $m = [Z_1]\cdot y_1^2 a_1$.
Geometrically, $[Z_1]\cdot y_1^2 a_1$ is the number of lines passing through
two points \big(which corresponds to the factor~$y_1^2$\big) and a marked point on the line
that intersects another generic line (which corresponds to the factor $a_1$). This is
clearly $1$. Similarly, $[Z_1]\cdot y_1 a_1^2$ is geometrically the number of lines
through two points (the first point corresponds to the factor $a_1^2$ and the second point
corresponds to the factor $y_1$). This number is also equal to $1$. Hence,
$[Z_1]= y_1 + a_1$, thereby proving \eqref{k2}.

Next, we note that $Y$ can also be described as a codimension three submanifold of $M$
in the following way
\begin{align*}
Y:= \{ (H, q_1, q_2) \in M\mid q_1 \in H,\,  q_1 = q_2\}.
\end{align*}
Hence, the Poincar\'{e} dual of $Y$ in $M$ is given by
$
\textnormal{PD}_{M}[Y] = (y_1+a_1)\cdot [\Delta_{12}]$,
where $\Delta_{12}$ is the subspace of points $(H_1, q_1, q_2)$ in $M$, such that
$q_1 = q_2$.
Using the standard fact that the class of the diagonal
$[\Delta_{12}]$ is equal to $\big(a_1^2 + a_1 a_2 + a_2^2\big)$,
we conclude that
\begin{align}
\textnormal{PD}_{M}[Y]& = (y_1+a_1)\cdot \big(a_1^2 + a_1 a_2 + a_2^2\big). \label{k3}
\end{align}
Let $\mu$ be a class in $M$ of degree three. By equations \eqref{k1}, \eqref{k2} and \eqref{k3},
we conclude that
\begin{align*}
(y_1+a_1)\cdot (a_1^2 + a_1 a_2 + a_2^2) \cdot \mu & = (A a_1 + B a_2 + C y_1)\cdot (y_1+a_1)\cdot(y_1+a_2) \cdot \mu.
\end{align*}
By suitably choosing $\mu$, we can determine $A$, $B$ and $C$.
Choosing $\mu := a_1 a_2^2$, $a_2a_1^2$ and $a_2 y_1^2$, we have
$A+C = 0$, $B+C = 0$ and $A = 1$, respectively. This precisely implies that
$A = B = 1$ and $C = -1$.
\end{proof}

\section{Counting smooth curves with multiple tangencies}\label{smooth_tangencies}
\label{smooth_curves_enum}
In this section, we use the collision lemma to derive our Main Result~\ref{m_rslt1}.
Before we get into the details, a brief outline of our method is described.

Let $\mathcal{D}_d$
denote the space of degree $d$ curves in $\mathbb{P}^2$. This is a complex projective space of dimension
\begin{equation}\label{eb}
\delta_d :=  \frac{d(d+3)}{2}.
\end{equation}
We now try to answer the following question:
{\it How many degree $d$ curves are there in $\mathbb{P}^2$ passing through
$\delta_d-1$ generic points and that are tangent to a given line?}

We approach this problem in the following way: First, consider the space of curves with two marked points $x_1$ and $x_2$, such that
both the points lie on the curve and the line. Now
impose the condition that $x_1$ becomes equal to $x_2$; this is precisely where
we use the collision lemma.
Once we impose the condition that the points $x_1$ and $x_2$
have become equal,
the curve becomes tangent to this line.
We now implement this idea precisely.
Let $X_1$ denote a copy of the projective plane. Define
the incidence variety
\begin{equation}\label{f1}
\mathrm{I}_d := \{ (H_d, x_1) \in  \mathcal{D}_d \times X_1 \mid  x_1 \in H_d\}.
\end{equation}
Elements of the incidence variety consists of
degree $d$-curve $H_d$ and a marked point $x_1$ that lies on this curve.
Let $a_1$ and $y_d$ denote the divisor classes on $\mathcal{D}_d \times X_1$ obtained by pulling back the
hyperplane classes on $X_1$ and $\mathcal{D}_d$, respectively. Then we have
\begin{equation}\label{incidence}
[\mathrm{I}_d]  =  y_d + d a_1.
\end{equation}
Indeed, this follows immediately from the fact that the restrictions of \eqref{incidence} to both
$\{{\rm point}\}\times \mathbb{CP}^2$ and $\mathcal{D}_d \times \{{\rm point}\}$ are valid.

Next, we study subspaces of $\mathcal{D}_1 \times \mathcal{D}_d \times X_1$.
First, define
$\mathrm{I}_{\mathsf{L}}$ (respectively, $\mathrm{I}_{\mathcal{C}}$)
to be the subspace of $\mathcal{D}_1 \times \mathcal{D}_d \times X_1$ consisting of all
$(H_1, H_d, x_1)  \in \mathcal{D}_1 \times \mathcal{D}_d \times X_1$
such that $x_1 \in H_1$ (respectively, $x_1 \in H_d$).
Next, define
\begin{equation}
\label{et}
\mathsf{T}_0
\end{equation}
to be the subspace
of $\mathcal{D}_1 \times \mathcal{D}_d \times X_1$ consisting of all
$(H_1, H_d, x_1)  \in \mathcal{D}_1 \times \mathcal{D}_d \times X_1$ such that~$H_1$ and~$H_d$ intersect transversally at $x_1$.
Note that the closure $\overline{\mathsf{T}}_0$ consists of
$(H_1, H_d, x_1)  \in \mathcal{D}_1 \times \mathcal{D}_d \times X_1$
such that $x_1 \in H_1\bigcap H_d$.

For making the notation easier to read, we denote the homology class
represented by the closure by the notation $[\mathsf{T}_0]$
\big(as opposed to more cumbersome $\big[\overline{\mathsf{T}}_0\big]$\big).

\begin{lmm}\label{Theo_tzr}
The class $[\mathsf{T}_0] \in H_{2(\delta_d +2)}(\mathcal{D}_1 \times \mathcal{D}_d \times X_1, {\mathbb R})$
represented by the cycle $\overline{\mathsf{T}}_0 \subset \mathcal{D}_1 \times \mathcal{D}_d \times X_1$
$($see \eqref{eb}, \eqref{et} and \eqref{eb}$)$ is the following:
\begin{equation}\label{T0_cycle_expression_base}
[\mathsf{T}_0]  =  [\mathrm{I}_{\mathsf{L}}] \cdot [\mathrm{I}_{\mathcal{C}}],
\end{equation}
where $[\mathrm{I}_{\mathsf{L}}]$ and $[\mathrm{I}_{\mathcal{C}}]$ are the homology classes of
$\mathrm{I}_{\mathsf{L}}$ and $\mathrm{I}_{\mathcal{C}}$, respectively.
\end{lmm}

\begin{proof}
Consider the divisor $\mathrm{I}_1  \subset \mathcal{D}_1\times X_1$ in \eqref{f1}.
Let
$
f \colon  \mathrm{I}_1\times \mathcal{D}_d  \hookrightarrow  \mathcal{D}_1 \times \mathcal{D}_d \times X_1$
be the natural map defined by $((H_1, x_1), H_d) \longmapsto (H_1, H_d, x_1)$, where
$(H_1, x_1) \in \mathrm{I}_1$ and $H_d \in \mathcal{D}_d$. Let
\begin{equation}\label{f3}
p \colon\ \mathcal{D}_1 \times \mathcal{D}_d \times X_1 \longrightarrow
\mathcal{D}_d \times X_1
\end{equation}
be the natural projection. Note that
$
\overline{\mathsf{T}}_0 \subset \mathrm{I}_1\times \mathcal{D}_d
$
(see \eqref{et}) and $\overline{\mathsf{T}}_0$ is a divisor.

To prove the lemma, it suffices to show the following:
\begin{equation}\label{f4}
f^*p^* {\mathcal O}_{\mathcal{D}_d \times X_1}(\mathrm{I}_d) =
{\mathcal O}_{\mathrm{I}_1\times \mathcal{D}_d}\bigl(\overline{\mathsf{T}}_0\bigr).
\end{equation}

We now describe three subvarieties $S_1$, $S_2$ and $S_3$ of $\mathrm{I}_1\times \mathcal{D}_d$. Fix
a point $z_0 \in \mathrm{I}_1$ and define%
\begin{equation}\label{s1}
S_1 := \{z_0\}\times\mathcal{D}_d \subset \mathrm{I}_1\times \mathcal{D}_d.
\end{equation}
Fix curves $(H_1, H_d) \in \mathcal{D}_1 \times \mathcal{D}_d$ in $X_1$, and we have
\begin{equation}\label{s2}
S_2 := \{(H_1, H_d, x_1) \in \mathcal{D}_1 \times \mathcal{D}_d \times X_1
  \mid  x_1 \in H_1\} \subset \mathrm{I}_1\times \mathcal{D}_d.
\end{equation}
So $S_2$ is identified with the line $H_1$. Fix a point $x_0 \in X_1$ and also an
element $H_d \in \mathcal{D}_d$. Now define
\begin{equation}\label{s3}
S_3 := \{(H, H_d, x_0) \in \mathcal{D}_1 \times \mathcal{D}_d \times X_1  \mid
x_0 \in H\} \subset \mathrm{I}_1\times \mathcal{D}_d.
\end{equation}
Consequently, $S_3$ is identified with the pencil of lines in $X_1$ containing $x_0$.

Given two cohomology classes ${\mathcal L}_1, {\mathcal L}_2  \in H^2(\mathrm{I}_1\times \mathcal{D}_d,
{\mathbb R})$, to show that
${\mathcal L}_1 = {\mathcal L}_2$, it is enough~to prove that
\smash{$
{\mathcal L}_1\big\vert_{S_j} = {\mathcal L}_2\big\vert_{S_j}
$}
for $j = 1, 2, 3$ (see \eqref{s1}, \eqref{s2}, \eqref{s3}).

In view of the above criterion, it is now straightforward to prove \eqref{f4} using \eqref{incidence}.
\end{proof}

Note that using \eqref{incidence}, we can rewrite equation \eqref{T0_cycle_expression_base} as
\begin{equation}\label{T0_cycle_expression_base_ag}
[\mathsf{T}_0]  =  (y_1 + a_1)\cdot(y_d + d a_1),
\end{equation}
where $\cdot$ denotes topological intersection. Generalizing the notion of
$\mathsf{T}_0$, given any nonnegative integer $k$, we define $\mathsf{T}_k$
to be the subspace of $\mathcal{D}_1 \times \mathcal{D}_d\times X_1$
consisting of all points $(H_1, H_d, x_1)$ such that
\begin{itemize}\itemsep=0pt
\item The points $x_1$ is a smooth point of the curve $H_d$.
\item The line $H_1$
intersects the curve $H_d$ at the points $x_1$
with the order of tangency precisely equal to $k$.
\end{itemize}
Notice that as per its definition, $\mathsf{T}_1$ is not a subset of $\mathsf{T}_0$, but a subset of the
closure $\overline{\mathsf{T}}_0$.

Next, generalizing $\mathcal{D}_1 \times \mathcal{D}_d \times X_1$, define
$
\mathsf{M}_n :=  \mathcal{D}_1 \times \mathcal{D}_d \times (X_1\times \cdots\times X_n)$,
where $X_j$, $1 \leq j \leq n$, is a copy of $\mathbb{P}^2$.
Let $\mathsf{S}_1, \mathsf{S}_2, \dots, \mathsf{S}_n$ be subvarieties of
$\mathcal{D}_1 \times \mathcal{D}_d \times \mathbb{P}^2$.
Denote $\mathsf{S}_1 \mathsf{S}_2 \dots \mathsf{S}_n$ by $\mathsf{S}$. Then,
$\mathsf{S} \subset \mathsf{M}_n$
consists of
$(H_1, H_d, x_1,\dots, x_n)$ such that
\begin{itemize}\itemsep=0pt
\item $(H_1, H_d, x_i) \in \mathsf{S}_i$ for all $i=1$ to $n$.
\item The points $x_1,\dots, x_n$ are all distinct.
\end{itemize}
As an example, consider the set $\mathsf{T}_1 \mathsf{T}_2$. This comprises of the set of
curves along with two \textit{distinct} marked points on a line, where the curve is tangent to first order to
the line at the first marked point and is tangent to second order to the line at the
second marked point.
Similarly, $\overline{\mathsf{T}}_1 \overline{\mathsf{T}}_2$~denotes a slightly bigger space, where the curve is
at least as degenerate as being tangent to the line to first order at the first marked
point and is at least
as degenerate as being tangent to the line to second order at the second marked point.
The curve could be tangent to second order at the first marked point. The curve could even have a nodal
point at the first marked point, since these both lie in the closure $\mathsf{T}_1$.
However, the two marked points have to be distinct.
In particular, $\overline{\mathsf{T}}_1 \overline{\mathsf{T}}_2$
is not the set-theoretic intersection of
$\overline{\mathsf{T}}_1$ and $\overline{\mathsf{T}}_2$, since the latter
includes the locus where the two marked points are equal.
Finally, in the space
$\overline{\mathsf{T}_1 \mathsf{T}}_2$, the two marked points need not be distinct; this denotes the
closure of the space $\mathsf{T}_1 \mathsf{T}_2$ and it includes the locus where the two marked points coincide.

We denote the homology class defined by the closure of $\mathsf{S}$
by the notation $\big[\mathsf{S}\big]$ as opposed to the more cumbersome $[\overline{\mathsf{S}}]$; this makes
some of the computations and statements easier to read.

Finally, let
\begin{equation}\label{ep}
\pi \colon\ \mathsf{M}_{n+1} \longrightarrow \mathsf{M}_n
\end{equation}
be the projection that sends
any $(H_1, H_d, x_1, \dots, x_{n}, x_{n+1})$ to $(H_1, H_d, x_1, \dots, x_{n})$. Let
\begin{equation}\label{ep1}
\pi_{n+1} \colon\  \mathsf{M}_{n+1} \longrightarrow \mathsf{M}_1
\end{equation}
be the projection that sends any $(H_1, H_d, x_1, \dots, x_{n}, x_{n+1})$ to
$(H_1, H_d, x_{n+1})$. For any $1 \leq i \leq n$, let
$
\Delta_{i, {n+1}}  \subset  \mathsf{M}_{n+1}
$
be the locus of all $(H_1, H_d, x_1, \dots, x_{n}, x_{n+1})$ such that $x_i = x_{n+1}$.
We are now ready to state the main results to enumerate smooth curves with tangencies.

\begin{thm}\label{theorem_for_many_Tks}
Let $n$ be a positive integer and $k_1, k_2,\dots, k_{n}$ nonnegative integers
with
$
k :=  k_1+k_2+\dots +k_n$.
Then the following equality of elements of $H_{2(\delta_d+2- k)}(\mathsf{M}_{n+1}, {\mathbb R})$ holds:
\begin{equation}\label{many_Tks}
\pi^{\ast}[\mathsf{T}_{k_1}\dots \mathsf{T}_{k_n}]\cdot
\pi_{n+1}^{\ast}[\mathsf{T}_0] =
[\mathsf{T}_{k_1}\dots \mathsf{T}_{k_n} \mathsf{T}_{0}]
 + \sum_{i=1}^{n}(k_i+1) \pi^{\ast}[\mathsf{T}_{k_1}\dots \mathsf{T}_{k_n}] \cdot [\Delta_{i, {n+1}}]
\end{equation}
$($see \eqref{ep}, \eqref{ep1}$)$ provided $d > k+n$.
\end{thm}

Note that implicit in Theorem \ref{theorem_for_many_Tks} is the assertion that
the closure of $\mathsf{T}_{k_1}\dots \mathsf{T}_{k_n}$ in
$\mathsf{M}_n$ and the closure of
$\mathsf{T}_{k_1}\dots \mathsf{T}_{k_n} \mathsf{T}_{0}$ in $\mathsf{M}_{n+1}$
actually define homology classes. We will justify that assertion as well.
In order to prove Theorem \ref{theorem_for_many_Tks}, we first prove an intermediate statement
which is interesting in its own right.
\begin{prp}
\label{prp_smth_mfld}
Let $n$ be a positive integer and $k_1, k_2,\dots, k_{n}$ nonnegative integers
with
$ k :=  k_1+k_2+\dots +k_n$.
Then $\overline{\mathsf{T}}_{k_1} \overline{\mathsf{T}}_{k_2}\dots \overline{\mathsf{T}}_{k_n}$
is a smooth submanifold of $\mathsf{M}_{n}$, provided
$d \geq k+n$.
\end{prp}

\begin{rem}
 We are \emph{not} claiming that
\smash{$\overline{\mathsf{T}_{k_1}\mathsf{T}_{k_2}\dots \mathsf{T}}_{k_n}$}
is a smooth manifold. Notice the difference between
\smash{$\overline{\mathsf{T}}_{k_1} \overline{\mathsf{T}}_{k_2}\dots \overline{\mathsf{T}}_{k_n}$}
and \smash{$\overline{\mathsf{T}_{k_1}\mathsf{T}_{k_2}\dots \mathsf{T}}_{k_n}$}; in the former space
all the marked points are distinct, while in the latter space, that is not necessarily true.
\end{rem}

\begin{proof}
We start by proving the proposition for $n = 1$
and $k_1 = 0$, i.e., we show that $\overline{\mathsf{T}}_0$ is a~smooth
submanifold of $\mathsf{M}_1$.

Let $\mathcal{F}_d$ be the space of
polynomials in two variables of degree at most $d$.
This is a vector space of dimension $\frac{d(d+3)}{2} + 1$, because
an element of $\mathcal{F}_d$ can be viewed as
\[
 f(x, y) := f_{00} + f_{10}x + f_{01} y + \frac{f_{20}}{2} x^2 + f_{11} xy + \frac{f_{02}}{2} y^2 +
\dots + \frac{f_{0d}}{d!} y^d,
\]
and hence $f$ can be identified with the vector
$(f_{00}, f_{10}, f_{01},\dots, f_{0d})$.
Let $\mathcal{F}_d^{+}$ denote the space of nonzero polynomials of degree $d$.
We note that the projectivization of $\mathcal{F}_d^{+}$ is $\mathcal{D}_d$.

Consider the projection map
\[ \pi_{+}\colon\ \mathcal{F}_1^{+} \times \mathcal{F}_d^{+} \times X_1 \longrightarrow
\mathcal{D}_1 \times \mathcal{D}_d \times X_1.  \]
We note that the projection map is a submersion. Hence, it suffices to show that
\smash{$\widetilde{\overline{\mathsf{T}}}_0:=
\pi_{+}^{-1}\bigl(\overline{\mathsf{T}}_0\bigr)$} is a smooth submanifold of
$\mathcal{F}_1^{+} \times \mathcal{F}_d^{+} \times X_1$. We prove that now.

Let \smash{$(g, f, p_1) \in \widetilde{\overline{\mathsf{T}}}_0$}.
By choosing a chart around the point $p_1$, we can identify an open neighbourhood of
$p_1$ (in $X_1$) by $\mathbb{C}^2$. Hence, in order to show that
\smash{$\widetilde{\overline{\mathsf{T}}}_0$} is a smooth
submanifold of~${\mathcal{F}_1^{+} \times \mathcal{F}_d^{+} \times X_1}$, it suffices to show that
$(0,0)$ is a regular value of the map
\[\varphi \colon\ \mathcal{F}_1^{+} \times \mathcal{F}_d^{+} \times \mathbb{C}^2 \longrightarrow
\mathbb{C}^2,
\qquad \text{given by}
\quad \varphi(g, f, (x, y)) := (g(x, y), f(x, y)).\]
Let us prove that now. Assume that $\varphi(g,f,(a,b))=0$.
We need to show that the differential of~$\varphi$,
evaluated at $(g, f, (a, b))$ is surjective.
In order to prove that, consider the two curves
\begin{align*}
\gamma_1, \gamma_2&\colon\ (-\varepsilon, \varepsilon)\longrightarrow
\mathcal{F}_1^{+} \times \mathcal{F}_d^{+} \times \mathbb{C}^2,
\qquad \textnormal{given by} \\
\gamma_1(t)&:= (g + t\eta_1, f, (a, b)) \qquad \textnormal{and}
\qquad \gamma_2(t):= (g, f+t\eta_2, (a, b)),
\end{align*}
where
$\eta_1$ and $\eta_2$ are as yet, unspecified elements of
$\mathcal{F}_1^{+}$ and $\mathcal{F}_d^{+}$.
We now note that
\[\{{\rm d}\varphi|_{(g,f(a,b))}\}(\gamma_1^{\prime}(0)) = (\eta_1(a,b), 0)
\qquad \textnormal{and} \qquad \{{\rm d}\varphi|_{(g,f,(a,b))}\}(\gamma_2^{\prime}(0)) = (0, \eta_2(a,b)).\]
Hence, to prove that the differential is surjective, we simply need to produce
$\eta_1 \in \mathcal{F}_1^{+}$ and $\eta_2 \in \mathcal{F}_d^{+}$
such that $\eta_1(a,b) \neq 0$ and $\eta_2(a,b) \neq 0$.
That is easily achieved: we simply define both of them to be the constant functions
taking the value $1$.

Before proceeding further, we make a couple of simplifications that make the
subsequent proofs easier.
We showed that if $\varphi(g, f, (a, b)) = 0$, then the
differential of $\varphi$ is surjective. We claim that by making a suitable change of
coordinates, we can always assume that~$(a,b)$ is the origin and the line is the $x$-axis.
To see why this is so,
assume that the line is given by~${g_{00} + g_{10}x + g_{01}y = 0}$.
Assuming that $g_{01}\neq 0$, define the new coordinates $X$
and $Y$ by $X := x-a$ and $Y := g_{00} + g_{10}x + g_{01}y$.
If $g_{01} =0$, then define $X:= y-b$ and $Y:= g_{00} + g_{10}x + g_{01}y$.
Define
\[F(X,Y):= f(x(X,Y), y(X,Y)).\]
This is the polynomial $f$
written in the new coordinates $X$ and $Y$. Define
\smash{$F_{ij} := \frac{\partial^{i+j}F(X,Y)}{\partial X^i \partial Y^j}\big|_{(0,0)}$}.
In these new coordinates, the point under consideration is the origin
and the line is the $X$-axis. Furthermore, the coefficients of the polynomial
are given by $\frac{F_{ij}}{i!j!}$.
Hence, what we have shown is that
\smash{$\widetilde{\overline{\mathsf{T}}}_0$} is a fibre bundle over the incidence variety
$J$ (where $J$ is defined to be the subset of~${\mathcal{F}_1^{+}\times X_1}$ where the point lies on the line)
and the fibre over $(g, p) \in J$ can be identified with~\smash{$\big(\overline{\mathsf{T}}_0\big)_{\mathsf{Aff}}$}, where
$\big(\overline{\mathsf{T}}_0\big)_{\mathsf{Aff}} := \{ f \in \mathcal{F}_d^+\colon f_{00}=0 \}$.
The map $f$ going to $F$ is a trivialization of this fibre bundle.
Henceforth, we set the line to be the $x$-axis and the point to be the origin; this makes the
calculations simpler.

We now show that $\overline{\mathsf{T}}_{k_1}$ is a smooth
submanifold of $\mathsf{M}_1$ for all $k_1$. We use induction on $k_1$.
Assume that we have proved the assertion till $k_1-1$. Hence
$\overline{\mathsf{T}}_{k_1-1}$ is a smooth
submanifold of~$\mathsf{M}_1$. Hence,
\smash{$\widetilde{\overline{\mathsf{T}}}_{k_1-1}:=
\pi_{+}^{-1}\big(\overline{\mathsf{T}}_{k_1-1}\big)$} is a smooth submanifold of
$\mathcal{F}_1^{+} \times \mathcal{F}_d^{+} \times X_1$.

We now show that \smash{$\widetilde{\overline{\mathsf{T}}}_{k_1}$} is a smooth submanifold
of \smash{$\widetilde{\overline{\mathsf{T}}}_{k_1-1}$}.
Define \smash{$\big(\overline{\mathsf{T}}_{k_1-1}\big)_{\mathsf{Aff}}$} as follows:
\[ \bigl(\overline{\mathsf{T}}_{k_1-1}\bigr)_{\mathsf{Aff}} :=
\bigl\{ f \in \mathcal{F}_d^{+}\mid f_{00}, f_{10} \dots, f_{k_1-1, 0} =0\bigr\}. \]
We note that
$\widetilde{\overline{\mathsf{T}}}_{k_1-1}$
is a fibre bundle over $\mathcal{F}_1^+\times X_1$
whose fibres can be identified with
\smash{$\bigl(\overline{\mathsf{T}}_{k_1-1}\bigr)_{\mathsf{Aff}}$}.
In order to show that
\smash{$\widetilde{\overline{\mathsf{T}}}_{k_1}$} is a smooth submanifold of
$\mathcal{F}_1^+\times \mathcal{F}_d^+\times X_1$, it suffices to show that
zero is a regular value of the map
\smash{$\varphi\colon \bigl(\overline{\mathsf{T}}_{k_1-1}\bigr)_{\mathsf{Aff}} \longrightarrow
\mathbb{C}$}, given by
$\varphi(f):= f_{k_1, 0}$.
Suppose
$f \in \bigl(\overline{\mathsf{T}}_{k_1-1}\bigr)_{\mathsf{Aff}}$.
Let \smash{$\gamma\colon(-\varepsilon, \varepsilon)\longrightarrow
\bigl(\overline{\mathsf{T}}_{k_1-1}\bigr)_{\mathsf{Aff}}$}
be a curve, given by
$\gamma(t):= f+t\eta$,
where $\eta$ is as yet an unspecified polynomial. We now note that
$\{d \varphi\}(\gamma^{\prime}(0)) = \eta_{k_1, 0}$. Now choose $\eta$ as follows
$\eta(x, y) := x^{k_1}$.
Since
$f \in \bigl(\overline{\mathsf{T}}_{k_1-1}\bigr)_{\mathsf{Aff}}$,
$f + t \eta$ also belongs to $\bigl(\overline{\mathsf{T}}_{k_1-1}\bigr)_{\mathsf{Aff}}$
for all $t$ nonzero but small.
Furthermore, $\eta_{k_1, 0} \neq 0$. This proves the claim.

Next, for multiple points, we now use induction on
$n$. If $k_n \geq 1$, define
\[\bigl(\overline{\mathsf{T}}_{k_1} \overline{\mathsf{T}}_{k_2}
\dots \overline{\mathsf{T}}_{k_{n-1}} \overline{\mathsf{T}}_{k_{n}-1}\bigr)_{\mathsf{Aff}}
\]
to be the following subset of $\mathcal{F}_d^{+} \times \mathbb{C}^{n-1}$: it is
the collection of all $(f, \mathbf{a}_1, \dots, \mathbf{a}_{n-1})$, such that
\begin{itemize}\itemsep=0pt
\item The numbers
$\mathbf{a}_1, \mathbf{a}_2, \dots \mathbf{a}_{n-1}$ are all distinct
from each other and different from zero.
\item All the derivatives of $f$ with respect to $x$
at $\mathbf{a}_i$ up to order $k_i$ are zero, for $i=1$ to $n-1$.
\item All the derivatives of $f$ with respect to $x$
at $(0, 0)$ up to order $k_n-1$ are zero.
\end{itemize}
It is also convenient to define
$\bigl(\overline{\mathsf{T}}_{k_1} \overline{\mathsf{T}}_{k_2}
\dots \overline{\mathsf{T}}_{k_{n-1}} \overline{\mathsf{T}}_{-1}\bigr)_{\mathsf{Aff}}$
as the following subset of $\mathcal{F}_d^{+} \times \mathbb{C}^{n-1}$: it is
the collection of all $(f, \mathbf{a}_1, \dots, \mathbf{a}_{n-1})$, such that
\begin{itemize}\itemsep=0pt
\item The numbers
$\mathbf{a}_1, \mathbf{a}_2, \dots , \mathbf{a}_{n-1}$ are all distinct
from each other and different from zero.
\item All the derivatives of $f$ with respect to $x$
at $\mathbf{a}_i$ up to order $k_i$ are zero, for $i=1$ to $n-1$.
\end{itemize}
Arguing as before,
it suffices to show that for all $k_n \geq 0$, zero is a regular value of the map
\[\varphi \colon\ \bigl(\overline{\mathsf{T}}_{k_1} \overline{\mathsf{T}}_{k_2}
\dots \overline{\mathsf{T}}_{k_{n-1}} \overline{\mathsf{T}}_{k_{n}-1}\bigr)_{\mathsf{Aff}}
 \longrightarrow \mathbb{C} \qquad \text{given by}
\quad \varphi(f, x_1, \dots, x_{n-1}) := f_{k_n, 0}. \]
Let us prove that now. Suppose
\[(f,\mathbf{a}_1, \dots, \mathbf{a}_{n-1})
\in \bigl(\overline{\mathsf{T}}_{k_1} \overline{\mathsf{T}}_{k_2}
\dots \overline{\mathsf{T}}_{k_{n-1}} \overline{\mathsf{T}}_{k_{n}-1}\bigr)_{\mathsf{Aff}}.\]
Let $\gamma\colon(-\varepsilon, \varepsilon)\longrightarrow
\bigl(\overline{\mathsf{T}}_{k_1} \overline{\mathsf{T}}_{k_2}
\dots \overline{\mathsf{T}}_{k_{n-1}} \overline{\mathsf{T}}_{k_{n}-1}\bigr)_{\mathsf{Aff}}$
be a curve, given by
$\gamma(t) := (f+t\eta, \mathbf{a}_1, \dots, \mathbf{a}_{n-1})$
where $\eta$ is as yet an unspecified polynomial. We now note that
$\{{\rm d} \varphi\}(\gamma^{\prime}(0)) = \eta_{k_{n}, 0}$. Now choose~$\eta$ as follows
\[ \eta(x, y) := (x-\mathbf{a}_1)^{k_1+1} \cdots (x-\mathbf{a}_{n-1})^{k_{n-1}+1} (x-0)^{k_n}.\]
Since $(f, \mathbf{a}_1, \dots, \mathbf{a}_{n-1})$
belongs to $\bigl(\overline{\mathsf{T}}_{k_1} \overline{\mathsf{T}}_{k_2}
\dots \overline{\mathsf{T}}_{k_{n-1}} \overline{\mathsf{T}}_{k_{n}-1}\bigr)_{\mathsf{Aff}}$,
$\gamma(t)$
also lies there for all $t$.
Furthermore, $\eta_{k_n, 0} \neq 0$, because the $\mathbf{a}_i$ are all different from zero.
This proves the proposition.
\end{proof}

We are now ready to prove Theorem \ref{theorem_for_many_Tks}.

\begin{proof}[Proof of Theorem \ref{theorem_for_many_Tks}]
We first show that \eqref{many_Tks} is valid on the set-theoretic level. Consider the first term
on the left-hand side, namely
$\pi^{\ast}[\mathsf{T}_{k_1}\dots \mathsf{T}_{k_n}]$.
It is represented by the closure of the following space: a line, a curve and
$n+1$ distinct points $(x_1,\dots, x_{n+1})$, such that the
curve is tangent to the line at the points
$(x_1,\dots, x_n)$ to orders $k_1,\dots, k_n$
respectively; the last point~$x_{n+1}$ is free (it does not have to lie on either the line or the curve).

Now consider second factor on the left-hand side of \eqref{many_Tks},
namely $\pi_{n+1}^{\ast}[\mathsf{T}_0]$.
This is simply represented by the following space: a
line, a curve and $n+1$ points $(x_1,\dots, x_{n+1})$,
such that the points $(x_1,\dots, x_n)$ are free, while the last point $x_{n+1}$ lies on the line
and the curve.

Consider the set-theoretic intersection of the above two spaces. There are two possibilities. The first
possibility is that the point $x_{n+1}$ is distinct from all the other points
$(x_1,\dots, x_n)$.
The closure of that space represents the first term on the right-hand side of
\eqref{many_Tks}. But there is another possibility. The point $x_{n+1}$ could be equal to one of the
$x_i$
(for $i \in \{1,\dots, n\}$).
That precisely gives us the second term on the right-hand side of \eqref{many_Tks}.

To see that equation \eqref{many_Tks} is valid on the level of homology,
we need to do the following.
To justify the first term on the right-hand side of \eqref{many_Tks},
we need to
show that the intersections are transverse;
this follows from the proof of Proposition \ref{prp_smth_mfld}.
To justify the second term,
we need to justify the multiplicity of the intersection.

Consider the situation of
$x_1$ coinciding with $x_{n+1}$. We take a chart that sends the point
$x_{n+1}$ to be the origin and sets the line to be the $x$-axis.
The situation now is that we have a curve $f$ that is tangent to the $x$-axis to order $k_1$
at the origin. We are now going to
study the multiplicity with which the evaluation map vanishes at the origin.
Hence, $f$ is such that $f_{00}, f_{10}, \dots , f_{k_1 0}$ all vanish. It is given by
\[
f(x, y)  =  \frac{f_{k_1+1, 0}}{(k_1+1)!} x^{k_1+1}
+\frac{f_{k_1+2, 0}}{(k_1+2)!} x^{k_1+2} + \dots +
\frac{f_{d, 0}}{d !} x^{d}  + y \mathcal{R}(x, y).
\]
Now consider the evaluation map
\[
\varphi(f, x) :=  f(x, 0)  = \frac{f_{k_1+1, 0}}{(k_1+1)!} x^{k_1+1}
+\frac{f_{k_1+2, 0}}{(k_1+2)!} x^{k_1+2} + \dots +
\frac{f_{d, 0}}{d !} x^{d}.
\]
The order of vanishing of $\varphi$ is clearly $k_1+1$, provided
$f_{k_1+1, 0} \neq 0$
(the values of $f_{k_1+2, 0}, f_{k_1+3, 0},\allowbreak \dots, f_{d,0}$ are not relevant for the order of vanishing
in a neighbourhood of the origin if $f_{k_1+1, 0}$ is non-zero).

The assumption $f_{k_1+1, 0} \neq 0$ is valid, because
to compute the order of vanishing, we will be intersecting with
cycles that correspond to constraints being generic. Hence, the order of vanishing is $k_1+1$.
This proves \eqref{many_Tks} on the level of homology.
\end{proof}

We are now ready to prove our next result.

\begin{thm}
\label{Tk1Tk2_etc_alt_way}
Let $n$ be a positive integer, $k_1, k_2,\dots, k_{n-1}$ nonnegative integers
and $k_n$ a~positive integer.
Define
$
k :=  k_1+k_2+\dots +k_n$.
Then the following equality of elements of $H_{2(\delta_d-k)}(\mathsf{M}_{n+1}, {\mathbb R})$ holds:
\begin{align}
[\mathsf{T}_{k_1} \dots \mathsf{T}_{k_{n-1}} \mathsf{T}_{k_n-1} \mathsf{T}_0] \cdot \bigl[\Delta_{n, {n+1}}^{\mathsf{L}}\bigr] & =
\pi^*[\mathsf{T}_{k_1} \dots \mathsf{T}_{k_n}] \cdot [\Delta_{n, n+1}],
\label{Tmu_T0_nu_alt_way_ag}
\end{align}
provided $d > k+n-1$.
\end{thm}

Before we prove Theorem \ref{Tk1Tk2_etc_alt_way}, a few things are explained.
Consider the special case of this theorem, when $n = 1$ and $k_1 = 1$.
In this case, \eqref{Tmu_T0_nu_alt_way_ag} simplifies to
\begin{align}
[\mathsf{T}_0 \mathsf{T}_0] \cdot \bigl[\Delta^{\mathsf{L}}_{12}\bigr] &  =  \pi^*[\mathsf{T}_1]\cdot [\Delta_{12}].
\label{T0T0_collides_T1_aagg}
\end{align}
If we draw an analogy with equation \eqref{A1FT0T0_coll_cycle_ver},
it might seem that the right-hand side
of equation~\eqref{T0T0_collides_T1_aagg}
has a missing term, namely
a term that corresponds to nodal curves lying
on a line. However, that is not the case; there are no missing terms.
The term that seems to be missing is actually present: it is present
inside the \emph{closure} $\overline{\mathsf{T}}_1$. However, since this locus
is one codimension higher, when we intersect
equation \eqref{T0T0_collides_T1_aagg} with a class of complementary dimension, we do not get
any contribution. We explain this more precisely. Define
$\mu := y_1^{2} y_d^{\delta_d-1}$. Now intersect both sides of
equation \eqref{T0T0_collides_T1_aagg} with $\mu$. That gives us
$[\mathsf{T}_0 \mathsf{T}_0] \cdot \bigl[\Delta^{\mathsf{L}}_{12}\bigr]\cdot \mu  =
[\mathsf{T}_1]\cdot \mu$.
The term that one might be worried that one has missed out, namely nodal curves
with the node lying on the line, giving empty intersection with $\mu$; this is
because we are making the curve pass through $\delta_d-1$ generic points \big(since we
are intersecting with $y^{\delta_d-1}$\big).

In contrast, look at equation \eqref{A1FT0T0_coll_cycle_ver}, namely
\begin{align*}
\bigl[\mathsf{A}_1^{\mathsf{F}}\mathsf{T}_0 \mathsf{T}_0\bigr] \cdot \bigl[\Delta^{\mathsf{L}}_{12}\bigr] &  =  \pi^*\bigl[\mathsf{A}_1^{\mathsf{F}}\mathsf{T}_1\bigr]\cdot [\Delta_{12}] + 2\pi^*\bigl[\mathsf{A}_1^{\mathsf{L}}\bigr]\cdot [\Delta_{12}]\cdot
\bigl[\Delta_{2}^1\bigr].
\end{align*}
The first geometric fact we note that $\mathsf{A}_1^{\mathsf{L}}$ is \emph{not} a subset of the
closure \smash{$\overline{\mathsf{A}_1^{\mathsf{F}}\mathsf{T}}_1$}.
What is true is that $\mathsf{A}_1^{\mathsf{F}}\mathsf{A}_1^{\mathsf{L}}$
is a subset of the closure (which for dimensional reasons, does not contribute when
we intersect with a class of complimentary dimension). In order to extract
numbers, define~$\mu := y_1^2 y_d^{\delta_d-2}$.
Intersecting both sides of equation \eqref{A1FT0T0_coll_cycle_ver}
with $\mu$ gives us
\begin{align*}
\bigl[\mathsf{A}_1^{\mathsf{F}}\mathsf{T}_0 \mathsf{T}_0\bigr] \cdot \mu  =
\bigl[\mathsf{A}_1^{\mathsf{F}}\mathsf{T}_1\bigr]\cdot \mu + 2 \bigl[\mathsf{A}_1^{\mathsf{L}}\bigr]\cdot \mu.
\end{align*}
In contrast to the earlier case, the intersection of
$\bigl[\mathsf{A}_1^{\mathsf{L}}\bigr]$ with $\mu$ is nonzero (or at least not necessarily zero)
because we are making the curve pass through $\delta_d-2$ points; that is precisely the right number
of points to enumerate $1$-nodal curves, with the node lying on a line.

\begin{proof}[Proof of Theorem \ref{Tk1Tk2_etc_alt_way}]
We first prove the theorem for $n = 1$ and $k_1 = 1$, namely
we prove equation \eqref{T0T0_collides_T1_aagg}.
The main set-theoretic statement that we need to prove is as follows:
consider the component of the closure $\overline{\mathsf{T}_0 \mathsf{T}}_0$.
where the two marked points are equal. Then the first derivative of the
polynomial (defining the curve) along the direction of the line
(evaluated at the marked point) is zero. In other words, if
$(H_1, H_d, p, p) \in \overline{\mathsf{T}_0 \mathsf{T}}_0$, then
$(H_1, H_d, p) \in \overline{\mathsf{T}}_1$. Let us prove this assertion.

We continue with the set-up of the proof of Proposition \ref{prp_smth_mfld}
and Theorem \ref{theorem_for_many_Tks}.
Choose coordinate where the designated line stays the $x$-axis.
Let $f_t$ be a curve that passes through the origin and also through the point $(t, 0)$.
The expression for $f_t$ is of the form
$
f_t(x, y)  =  \varphi_t(x)+ y \mathcal{R}_t(x, y)$,
where \smash{$\varphi_{t}(x) = \bigl(f_{t_{10}} x + \frac{f_{t_{20}}}{2} x^2 + \cdots\bigr)$}.
Since the curve passes through $(t, 0)$, it follows that
$
\varphi_{t}(x)  =  (\mathcal{K}_t(x))x(x-t)
$
for some function $\mathcal{K}_t(x)$. Denote $f_0$ by $f$. Hence,
$
f(x, y) =  (\mathcal{K}_0(x))x^2 + y \mathcal{R}_0(x, y)$.
It is a simple check to see that $f_{00}$ and $f_{10}$ are both zero.
Hence, if two $\mathsf{T}_0$ points collide, then we get a point which is at least as degenerate as a
$\mathsf{T}_1$ point (it could be even more degenerate). This proves the
assertion we made.

We now prove a more general statement.
We claim the following:
consider the component of the closure $\overline{\mathsf{T}_{k} \mathsf{T}}_0$
where the two marked points are equal. Then the $(k+1)$-th derivative of the
polynomial (defining the curve) along the direction of the line
(evaluated at the marked point) is zero. In other words, if
$(H_1, H_d, p, p) \in \overline{\mathsf{T}_{k} \mathsf{T}}_0$, then
$(H_1, H_d, p) \in \overline{\mathsf{T}}_{k+1}$.

In order to prove the above assertion, put the
$\mathsf{T}_k$ point at $(0, 0)$ and the $\mathsf{T}_0$ point at $(t, 0)$.
The expression for $f_t$ is going to be of the form
\begin{align*}
f_t(x, y)& = \varphi_t(x)+ y \mathcal{R}_t(x, y),\qquad
\textnormal{where} \quad \varphi_{t}(x) =  f_{t_{10}} x + \frac{f_{t_{20}}}{2} x^2 + \frac{f_{t_{30}}}{6} x^3 + \cdots .
\end{align*}
Since the curve is tangent of the $x$-axis to order $k$ and it also passes through $(t, 0)$,
it follows that
$
\varphi_{t}(x)  =  (\mathcal{K}_t(x))x^{k+1}(x-t)
$
for some function $\mathcal{K}_t(x)$. To see what happens in the limit as~$t$ goes to zero, denote $f_0$ by $f$.
Hence,
\[
f(x, y) =  (\mathcal{K}_0(x))x^{k+2} + y \mathcal{R}_0(x, y).
\]
It is a simple check that $f_{00}, f_{10}, \dots, f_{k+1, 0}$ are all zero.
Hence, if a $\mathsf{T}_k$ and a $\mathsf{T}_0$ point collide, then
we get a point which is at least as degenerate as a $\mathsf{T}_{k+1}$ point (it could be even more degenerate).

We now need to prove the converse of the above assertion.
We claim the following: if $(H_1, H_d, p) \in \overline{\mathsf{T}}_{k+1}$,
then $(H_1, H_d, p, p) \in \overline{\mathsf{T}_{k} \mathsf{T}}_0$.

First we note that to prove the above claim, it is sufficient to prove the following:
if $(H_1, H_d, p) \in \mathsf{T}_{k+1}$,
then $(H_1, H_d, p, p) \in \overline{\mathsf{T}_{k} \mathsf{T}}_0$.
This is because if $A$ is a subset of $B$, then closure of $A$ is a subset of
closure of $B$. Hence, if $B$ is a closed set, then to show that closure of~$A$
is a subset of $B$, it is sufficient to show that $A$ is a subset of $B$.

We start by proving the claim for $k = 0$, namely, we show that
every $\mathsf{T}_1$ point can be obtained as a limit to two $\mathsf{T}_0$ points.

Let $f \in (\mathsf{T}_1)_{\textnormal{Aff}}$. This means that $f_{00}$ and $f_{10}$ are both equal to zero.
Hence, $f$ is given by
\[
f(x, y)  =  \left(\frac{f_{20}}{2} x^2 + \frac{f_{30}}{6} x^3 + \cdots\right) + y \mathcal{R}(x, y).
\]
It will be shown that there exists a point $(f_t, (t, 0))$ close to $(f, (0, 0))$, such that
$f_t$ passes through the origin and $(t, 0)$. Note that
since $f_t$ passes through the origin, it is of the form
\[
f_t(x, y) = f_{t_{10}}x+ \left(\frac{f_{t_{20}}}{2} x^2 + \frac{f_{t_{30}}}{6} x^3 +\cdots\right) +
y \mathcal{R}_t(x, y).
\]
Furthermore, $f_{t_{10}}$ has to be small (since $f_t$ is close to $f$ and $f_{10}$ is zero).
Impose the condition that $f_{t}(t, 0) = 0$. Plugging in $(t, 0)$ inside $f_t$ and using the fact
that $t \neq 0$, it follows that
\begin{align}
f_{t_{10}}&= -\frac{f_{t_{20}}}{2} t + O\bigl(t^2\bigr). \label{ft_10}
\end{align}
Hence, we have constructed this nearby curve $f_t$ and a marked point $(t, 0)$ different from the origin
that lies on the curve and the line.

To find the multiplicity of the intersection, we note that using equation \eqref{ft_10},
using the fact that $f_{t_{20}} \neq 0$
and the implicit function theorem, we can rewrite it as
\begin{align}
t & = -\frac{2}{f_{t_{20}}} f_{t_{10}} + O\bigl(f_{t_{10}}^2\bigr). \label{ft_10_t}
\end{align}
Setting the two points to be equal is the same as setting $t$ to be equal to zero.
By equation~\eqref{ft_10_t}, the order of vanishing of $t$ is one. This
justifies the multiplicity.

We now prove the assertion for a general $k$,
i.e., we show that
every $\mathsf{T}_{k+1}$ curve is in the limit of a $\mathsf{T}_{k}$
and $\mathsf{T}_0$ point.
Let $f \in (\mathsf{T}_{k+1})_{\textnormal{Aff}}$.
This means that $f_{00}, f_{10}, \dots, f_{k+1,0}$ are all equal to zero.
Hence, $f$ is given by
\[
f(x, y)  =  \left(\frac{f_{k+2,0}}{(k+2)!} x^{k+2} +
\frac{f_{k+3,0}}{(k+3)!} x^{k+3} + \cdots\right) + y \mathcal{R}(x, y).
\]
It will be shown that there exists a point $(f_t, (t, 0))$ close to $(f, (0, 0))$, such that
$f_t \in (\mathsf{T}_{k})_{\textnormal{Aff}}$
and~$(t, 0)$.
Since $f_t \in (\mathsf{T}_{k})_{\textnormal{Aff}}$,
it is of the form
\[
f_t(x, y) = \frac{f_{t_{k+1, 0}}}{(k+1)!}x^{k+1}+
\left(\frac{f_{t_{k+2, 0}}}{(k+2)!} x^{k+2} + \frac{f_{t_{k+3,0}}}{(k+3)!} x^{k+3} +\cdots\right) +
y \mathcal{R}_t(x, y).
\]
Furthermore, $f_{t_{k+1, 0}}$ has to be small (since $f_t$ is close to $f$ and $f_{k+1, 0}$ is zero).
Impose the condition that $f_{t}(t, 0) = 0$. Plugging in $(t, 0)$ inside $f_t$ and using the fact
that $t \neq 0$, it follows that
\begin{align}
f_{t_{k+1, 0}}&= -\frac{f_{t_{k+2, 0}}}{(k+2)} t + O\bigl(t^2\bigr). \label{ft_k0}
\end{align}
Hence, we have constructed this nearby curve $f_t$ and a marked point $(t, 0)$ different from the origin
that lies on the curve and the line.

To find the multiplicity of the intersection, we note that using equation \eqref{ft_k0},
using the fact that $f_{t_{k+2, 0}} \neq 0$ (which is true because after we make the curves
pass through generic points)
and the implicit function theorem, we can rewrite it as
\begin{align}
t & = -\frac{2}{f_{t_{k+2, 0}}} f_{t_{k+1, 0}} + O\bigl(f_{t_{k+1, 0}}^2\bigr). \label{ft_k0_t}
\end{align}
Setting the two points to be equal is the same as setting $t$ to be equal to zero.
By equation~\eqref{ft_k0_t}, the order of vanishing of $t$ is one. This
justifies the multiplicity.

We summarize the set-theoretic statement we have just proved.
We have shown that $(H_1, H_d,\allowbreak  p, p)$ belongs to $\overline{\mathsf{T}_k \mathsf{T}}_0$
if and only if $(H_1, H_d, p)$ belongs to $\overline{\mathsf{T}}_{k+1}$.

We now examine what happens when there are more than two points involved.
We explain with the help of a single example; the general case follows in a similar way.
Consider the following assertion:
\[
[\mathsf{T}_2 \mathsf{T}_0 \mathsf{T}_0] \cdot \bigl[\Delta^{\mathsf{L}}_{23}\bigr]  =
\pi^*[\mathsf{T}_2\mathsf{T}_1]\cdot [\Delta_{23}].
\]
The closure claim that we need to prove is as follows:
$(H_1, H_d, q, p, p)$ belongs to \smash{$\overline{\mathsf{T}_2\mathsf{T}_0 \mathsf{T}}_0$}
if and only if $(H_1, H_d, p)$ belongs to~\smash{$\overline{\mathsf{T}}_{1}$}.
One direction is the same as before namely that if~${(H_1, H_d, q, p, p)}$ belongs to
$\overline{\mathsf{T}_2\mathsf{T}_0 \mathsf{T}}_0$, then~$(H_1, H_d, p)$ belongs to $\overline{\mathsf{T}}_{1}$. The fact that
$(H_1, H_d, q)$ belongs to $\overline{\mathsf{T}}_2$ makes no difference in the proof.
It is the converse that requires a little bit more argument. We need to show that
every $\mathsf{T}_2 \mathsf{T}_1$ point can be obtained as a limit of
$\mathsf{T}_2 \mathsf{T}_0 \mathsf{T}_0$. First see what we did when we
had to show every $\mathsf{T}_1$ point can be obtained as a limit of
$\mathsf{T}_0 \mathsf{T}_0$. We constructed the curve as given by equation
\eqref{ft_10}. The problem with equation \eqref{ft_10} is that if we set
$f_{t_{10}}$ as given by \eqref{ft_10} and the remaining $f_{t_{ij}}$ to be
complex numbers close to $f_{ij}$, then this curve does not satisfy the
$\mathsf{T}_2$ condition at $q$. The problem is the \smash{$f_{t_{ij}}$} are not all free.
However, we have shown through the proof of Proposition~\ref{prp_smth_mfld} that
restricted to the submanifold~$\mathsf{T}_2 \overline{\mathsf{T}}_0$, the section induced by $f_{10}$ is transverse to
zero. Hence, $f_{10}$ is an actual coordinate on the submanifold~$\mathsf{T}_2 \overline{\mathsf{T}}_0$.
Hence, \eqref{ft_10} defines a curve lying in~$\mathsf{T}_2 \overline{\mathsf{T}}_0$ that
converges to an element of~$\mathsf{T}_2 \overline{\mathsf{T}}_1$ as~$t$ goes to zero.
The multiplicity computation is the same. The general case follows from equation~\eqref{ft_k0} and using the fact that $f_{k+1, 0}$ is a~coordinate on~$\mathsf{T}_2 \overline{\mathsf{T}}_{k}$. When there are more than three points involved, the
same argument holds via the proof of Proposition~\ref{prp_smth_mfld},
namely that~$f_{k_n, 0}$ is a coordinate on
$\overline{\mathsf{T}}_{k_1} \overline{\mathsf{T}}_{k_2}
\dots \overline{\mathsf{T}}_{k_{n-1}}\overline{\mathsf{T}}_{k_{n}-1}$.
This completes the proof of Theorem~\ref{Tk1Tk2_etc_alt_way}.
 \end{proof}

\begin{rem} The pullback to $\mathsf{M}_n$ of the hyperplane classes in
$X_{i}$ for $i = 1, \dots, n$ are denoted by $a_1, \dots, a_n$. Note that \eqref{many_Tks} can be rewritten as
\begin{align}
[\mathsf{T}_{k_1} \mathsf{T}_{k_2} \dots \mathsf{T}_{k_n} \mathsf{T}_{0}]  ={}&
\pi^{\ast}[\mathsf{T}_{k_1} \mathsf{T}_{k_2} \dots \mathsf{T}_{k_n}]\cdot
\pi_{n+1}^{\ast}[\mathsf{T}_0] \nonumber\\
&-\sum_{i=1}^{n}(1+k_i) \pi^{\ast} [\mathsf{T}_{k_1} \mathsf{T}_{k_2} \dots \mathsf{T}_{k_n}] \cdot
[\Delta_{i, {n+1}}].\label{many_Tks_rewrite}
\end{align}
Also, $[\Delta_{i, {n+1}}] = a_i^2 + a_i a_{n+1} + a_{n+1}^2$.
Hence,
using \eqref{T0_cycle_expression_base_ag} and \eqref{many_Tks_rewrite}
we can recursively compute all the intersection numbers
involving
the class $[\mathsf{T}_{k_1} \mathsf{T}_{k_2} \dots \mathsf{T}_{k_n} \mathsf{T}_{0}]$.
Next, let $\alpha$ be a class in $\mathsf{M}_n$.
Then \eqref{Tmu_T0_nu_alt_way_ag} implies that
\begin{align}
[\mathsf{T}_{k_1} \dots \mathsf{T}_{k_n}] \cdot \alpha =  [\mathsf{T}_{k_1} \dots \mathsf{T}_{k_{n-1}} \mathsf{T}_{k_n-1}
\mathsf{T}_0] \cdot \bigl[\Delta_{n, {n+1}}^{\mathsf{L}}\bigr] \cdot \alpha. \label{Tmu_T0_nu_alt_way_ag8}
\end{align}
Hence, using equations
\eqref{Tmu_T0_nu_alt_way_ag8} and \eqref{Line_Diag_Defn}
and using the fact that all
intersection numbers
involving
the class $[\mathsf{T}_{k_1} \mathsf{T}_{k_2} \dots \mathsf{T}_{k_n} \mathsf{T}_{0}]$
are computable, we conclude that
all
intersection numbers
involving
the class $[\mathsf{T}_{k_1} \mathsf{T}_{k_2} \dots \mathsf{T}_{k_n}]$
are computable. From the procedure to compute, it is clear that these
intersection numbers are all polynomials in $d$.
\end{rem}

\section[Counting 1-nodal curves with multiple tangencies]{Counting $\boldsymbol{ 1}$-nodal curves with multiple tangencies}\label{nodal_tang}
\label{one_nodal_tang_comp}

We now consider plane curves with singularities.
Consider the following question: How many $1$-nodal degree $d$
curves are there in $\mathbb{P}^2$
that pass through $\delta_d-2$ (see \eqref{eb})
generic points and that are tangent to a given line?
For this, consider
the space of curves with three distinct points~$p_1$,~$x_1$, and $x_2$ such that the curve has a node at $p_1$ and intersects the given line
transversally at $x_1$ and~$x_2$.
The closure of this space represents a cycle.
Next, impose the condition
that $x_1$ becomes equal to~$x_2$ (which again, represents a cycle). We might naively expect that
the intersection of these two cycles give us the space of one nodal curves
tangent to a given line, as it would be suggested by the following picture:

\begin{figure}[!h]\centering
\includegraphics[scale = .9]{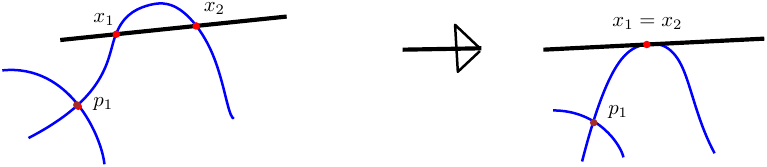}
\end{figure}

But there is an extra object that occurs. There is also the space of curves
with one node lying on the line, as shown by the following picture:

\begin{figure}[h!]\centering
\includegraphics[scale = .9]{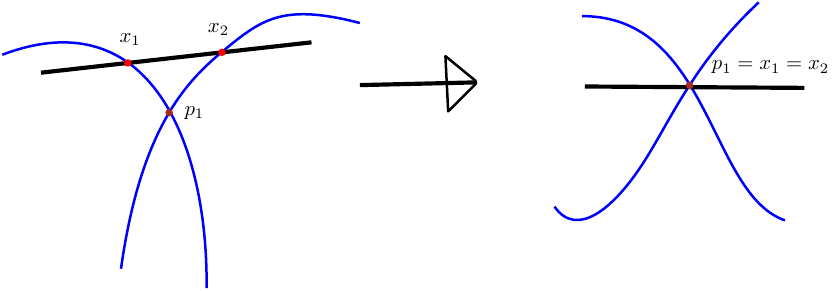}
\end{figure}

This results in an excess contribution to the intersection.
The same
thing happens if the curve has a more degenerate singularity.
We have been able to compute the excess contribution to the intersection in the following cases:
\begin{itemize}\itemsep=0pt
\item[$\bullet$] When the degree $d$ plane curve has
a node and is tangent to a given line at multiple points of any order.

\item[$\bullet$] When the degree $d$ plane curve has
a cusp and is tangent to a given line at multiple points (only tangency of order one).
\end{itemize}
When the singularity is a node,
our answers agree with those
predicted by the Caporaso--Harris formula. When the singularity is a cusp, our results are new (to the best of our knowledge).
We expect that this idea can be pursued further to enumerate curves with multiple nodes
and also enumerate singular curves tangent to a given line, when the singularities are even more degenerate
than a cusp. We hope to pursue these questions in future.

We now implement the idea which has just been described.
For that, recall a standard definition about singularities.
\begin{defn}
\label{singularity_defn}
Let $U$ be an open neighbourhood of the origin in $\mathbb{C}^2$
(open with respect to the usual topology of $\mathbb{C}^2\approx \mathbb{R}^4$ given by the
Euclidean metric)
and let $f \colon (U, \mathbf{0}) \longrightarrow (\mathbb{C}, 0)$ be a holomorphic
function. A point $q \in f^{-1}(0)$ has
an $\mathsf{A}_k$-singularity if there exists a coordinate system
$(x, y) \colon (V, \mathbf{0}) \longrightarrow \bigl(\mathbb{C}^2, \mathbf{0}\bigr)$ such that $f^{-1}(0) \cap V$ is given by
$
y^2 + x^{k+1} = 0$. An $\mathsf{A}_1$-singularity is also called a \textsf{node}, an $\mathsf{A}_2$-singularity is
called a \textsf{cusp} while an $\mathsf{A}_3$-singularity is also called a \textsf{tacnode}.
\end{defn}
Define
\[
\mathsf{M}_{n}^m :=  \mathcal{D}_1 \times \mathcal{D}_d \times \bigl(X^1 \times \cdots \times X^m\bigr) \times (X_1\times \cdots\times X_n),
\]
where each copy of $X_i$ and $X^j$ is $\mathbb{P}^2$;
the hyperplane classes are denoted by $a_i$ and $b_j$, respectively. Let $\mathsf{W}$ and
$\mathsf{S}_1, \mathsf{S}_2, \dots, \mathsf{S}_n$ be subsets of
$\mathcal{D}_1 \times \mathcal{D}_d \times \mathbb{P}^2$ (here $m=1$). We define
$\mathsf{W}  \mathsf{S}_1 \mathsf{S}_2 \dots \mathsf{S}_n \subset \mathsf{M}_n^1$
as follows:
it consists of all
$(H_1, H_d, p, x_1,\dots, x_n)$ such that
\begin{itemize}\itemsep=0pt
\item $(H_1, H_d, p) \in \mathsf{W}$,
\item $(H_1, H_d, x_i) \in \mathsf{S}_i$ for all $i \in \{1, \dots, n\}$,
\item The points $p, x_1, \dots , x_n$ are all distinct.
\end{itemize}
As before, we denote the homology class represented by the closure
by putting a square bracket, but without putting the cumbersome bar.

Next, define a few subsets of $\mathsf{M}^1_0$ (we are setting here $m$ to be equal to $1$).
First of all, we define
$\mathsf{A}_1^{\mathsf{F}} \subset \mathsf{M}^1_0$ to be the subset of all
$(H_1, H_d, p) \in \mathsf{M}^1_0$ such that
$H_d$ has a node at $p$. The letter~$\mathsf{F}$ is there to remind us that the nodal
point is free, i.e., it does not have to lie on the line.

Similarly, we define
$\mathsf{A}_1^{\mathsf{L}} \subset \mathsf{M}^1_0$ to be the subset of all
$(H_1, H_d, p) \in \mathsf{M}^1_0$ such that
\begin{itemize}\itemsep=0pt
\item The curve $H_d$ has a node at $p$.
\item The point $p$ lies on the line $H_1$.
\end{itemize}
Finally, given any nonnegative integer $r$, we define
$\mathsf{P}^{(r)}\mathsf{A}_1\!\subset\! \mathsf{M}^1_0$ to be the
subset of all
$(H_1, H_d, p)\! \in \mathsf{M}^1_0$ such that
\begin{itemize}\itemsep=0pt
\item The curve $H_d$ has a node at $p$.
\item The point $p$ lies on the line $H_1$.
\item One of the branches of the node is tangent to the line $H_1$ to order $r$.
\end{itemize}

We note that the closure of $\mathsf{P}^{(0)}\mathsf{A}_1$ is same as the
closure of $\mathsf{A}_1^{\mathsf{L}}$ \big(in $\mathsf{M}^1_0$\big).
Note that
\begin{align}
\bigl[\mathsf{A}_1^{\mathsf{L}}\bigr] &  =  \bigl[\mathsf{A}_1^{\mathsf{F}}\bigr]\cdot (y_1 + b_1);
\label{a1_line_cycle}
\end{align}
this is because intersecting with $(y_1+b_1)$ corresponds to the point $p$ lying on the line.

We now prove a few transversality results that we use to enumerate nodal curves
with tangencies.
In the following propositions,
$n$, $k_1, k_2,\dots, k_{n}$ and $r$ are nonnegative integers
and~${k := k_1+k_2 + \dots+ k_n}$.

\begin{prp}\label{prp_smth_mfld_with_PA1r_ag}
The space
\smash{$\overline{\mathsf{P}^{(r)}\mathsf{A}}_1\overline{\mathsf{T}}_{k_1}\dots \overline{\mathsf{T}}_{k_n}$}
is a smooth submanifold of $\mathsf{M}_{n}^{1}$, provided
$d \geq k+r+n+2$.
\end{prp}

\begin{prp}
\label{prp_smth_mfld_with_A1f_ag}
The space
\smash{$\overline{\mathsf{A}^{\mathsf{F}}_1}\overline{\mathsf{T}}_{k_1}\dots \overline{\mathsf{T}}_{k_n}$}
is a smooth submanifold of $\mathsf{M}_{n}^{1}$, provided $d \geq k+n+1$.
\end{prp}

\begin{proof}[Proof of Proposition \ref{prp_smth_mfld_with_PA1r_ag}]
We show that the space \smash{$\overline{\mathsf{P}^{(r)}\mathsf{A}}_1\overline{\mathsf{T}}_{k_1}\dots \overline{\mathsf{T}}_{k_n}$} is a codimension one submanifold of
$\overline{\mathsf{T}}_{r+1}\overline{\mathsf{T}}_{k_1}\dots \overline{\mathsf{T}}_{k_n}$.
The proof is very similar to
that of Proposition \ref{prp_smth_mfld}, where we
show that
$\overline{\mathsf{T}}_{r+2}\overline{\mathsf{T}}_{k_1}\dots \overline{\mathsf{T}}_{k_n}$
is a smooth codimension one submanifold
of $\overline{\mathsf{T}}_{r+1}\overline{\mathsf{T}}_{k_1}\dots \overline{\mathsf{T}}_{k_n}$.

We switch to affine space at set the $\overline{\mathsf{T}}_{r+1}$ point to be the origin
(and the line to be the $x$-axis). To show that
$\overline{\mathsf{T}}_{r+2}\overline{\mathsf{T}}_{k_1}\dots \overline{\mathsf{T}}_{k_n}$
is a smooth submanifold, we show that the section induced by taking the $(r+2)$-th
derivative (i.e., $f_{r+2,0}$) is transverse to zero. The procedure for doing that
was as follows: we considered the polynomial $\eta$, given by
\[ \eta(x,y):= (x-\mathbf{a}_1)^{k_1+1} \cdots (x-\mathbf{a}_{n-1})^{k_{n-1}+1} (x-\mathbf{a}_{n})^{k_{n}+1} (x-0)^{r+2}. \]
Here $(\mathbf{a}_i, 0)$ is the point at which the degree $d$ curve is tangent to the
$x$-axis to order $k_i$. Using this polynomial $\eta$, we are able to construct
a tangent vector, such that the differential of the section (that is induced by taking the
$(r+2)$-th derivative) evaluated on this tangent vector is nonzero. This
proves transversality.

In a similar way, we can show that the section
induced by taking the first derivative in the $y$ direction (i.e., $f_{01}$) is transverse
to zero. To do that, we define the polynomial $\eta$, given by
\[ \eta(x,y):= (x-\mathbf{a}_1)^{k_1+1} \cdots (x-\mathbf{a}_{n-1})^{k_{n-1}+1}(x-\mathbf{a}_{n})^{k_{n}+1} (x-0)^{r+1} (y-0). \]
Using this $\eta$, we can use a similar argument to compute the differential and prove transversality.
Hence,
$\overline{\mathsf{P}^{(r)}\mathsf{A}}_1\overline{\mathsf{T}}_{k_1}\dots \overline{\mathsf{T}}_{k_n}$
is a smooth codimension one submanifold of
$\overline{\mathsf{T}}_{r+1}\overline{\mathsf{T}}_{k_1}\dots \overline{\mathsf{T}}_{k_n}$.
Note that the bound on $d$ is required because to apply our argument, we need
$\overline{\mathsf{T}}_{r+1}\overline{\mathsf{T}}_{k_1}\dots \overline{\mathsf{T}}_{k_n}$
to be a~smooth manifold; this where the bound on $d$ is required
(which is bigger than what is required to simply construct the given $\eta$).
\end{proof}

\begin{proof}[Proof of Proposition \ref{prp_smth_mfld_with_A1f_ag}]
Start with $n = 0$, i.e., we show that \smash{$\overline{\mathsf{A}^{\mathsf{F}}}_1$}
is a smooth submanifold of $\mathsf{M}^1_0$ of codimension three.
The assertion is proved in \cite[pp.~216--217]{McSa},
but for the convenience of the reader, we include the proof here.
Switching to affine space, we consider the map
\[
\varphi \colon\ \mathcal{F}_d^{+} \times \mathbb{C}^2  \longrightarrow  \mathbb{C}^3, \qquad
(f, (x, y)) \longmapsto (f(x, y), f_x(x, y), f_y(x, y)).
\]
It suffices to show that $(0,0,0)$ is a regular value of $\varphi$.
Suppose $\varphi(f, (a,b)) =0$. We need to
show that the differential of $\varphi$, evaluated at $(f, (a,b))$ is surjective.
Consider the polynomials~$\eta_{ij}(x, y)$ given by
$
\eta_{00}(x, y) := 1$, $ \eta_{10}(x, y) := (x-a)$, and $\eta_{01}(x, y) := (y-b)$.
Let $\gamma_{ij}(t)$ be the curve given by
$\gamma_{ij}(t) :=  (f+t \eta_{ij}, (x,y))$.
We now note that
\begin{align}
\{{\rm d}\varphi|_{(f, (a,b))}\}(\gamma_{00}^{\prime}(0)) & = (1,0,0), \qquad
\{{\rm d}\varphi|_{(f, (a,b))}\}(\gamma_{10}^{\prime}(0)) = (0,1,0)
\qquad \textnormal{and} \nonumber \\
\{{\rm d}\varphi|_{(f, (a,b))}\}(\gamma_{01}^{\prime}(0)) & = (0,0,1). \nonumber
\end{align}
This shows that the differential is surjective.

Next, assume that $n = 1$ and $k_1 = 0$. We show that
\smash{$\overline{\mathsf{A}^{\mathsf{F}}_1}\overline{\mathsf{T}}_{0}$} is a smooth
submanifold of $\mathsf{M}_{1}^{1}$.
We switch to affine space and set the designated line to be the $x$-axis.
Define \smash{$\bigl(\overline{\mathsf{A}^{\mathsf{F}}_1\bigr)}_{\mathsf{Aff}}$} to be
\begin{align}
\bigl(\overline{\mathsf{A}^{\mathsf{F}}_1\bigr)}_{\mathsf{Aff}}& :=
\big\{(f, (x,y)) \in \mathcal{F}_d^+\times \mathbb{C}^2 \mid  f(x, y) = 0,\, f_x(x,y) = 0,\, f_y(x,y) = 0\big\}.
\label{A1F_affine}
\end{align}
Consider the map
\smash{$\psi \colon \bigl(\overline{\mathsf{A}^{\mathsf{F}}_1\bigr)}_{\mathsf{Aff}}
 \longrightarrow \mathbb{C}$}, given by
defined by
$(f, (x, y) \longmapsto  f(0, 0)$.
We need to show if $\psi(f, (a, b)) = 0$ and
$(a,b)\neq (0,0)$, then the differential of $\psi$
is surjective. Note that we are setting here the $\mathsf{T}_0$
point to be the origin. Let $\gamma (t)$ be the curve in
\smash{$\bigl(\overline{\mathsf{A}^{\mathsf{F}}_1\bigr)}_{\mathsf{Aff}}$}
given by~$ {\gamma(t) :=  (f+t \eta, (a,b))}$,
where is $\eta$ is as yet unspecified. We note that
$\{{\rm d}\psi|_{(f, (a,b))}\}(\gamma^{\prime}(0)) = \eta(0,0)$.
We now see what is our requirement on $\eta$.
First, we need that $\eta(0, 0) \neq 0$.
Second of all, we need the curve $\gamma(t)$ to lie in
\smash{$\bigl(\overline{\mathsf{A}^{\mathsf{F}}_1\bigr)}_{\mathsf{Aff}}$}.
For this, it is sufficient if the value of $\eta$ and both its
first partial derivatives evaluated at $(a, b)$ are equal to zero.
Define $\theta(x, y)$ as follows:
\begin{align}
\label{theta_defn}
\theta(x, y) := \begin{cases} (y-b)^2 &\text{if} \  b  \neq 0, \\
              (x-a)^2 & \text{if} \  b = 0, \  a \neq 0.
       \end{cases}
\end{align}
Define $\eta(x, y) := \theta(x, y)$. It is a simple check to see that $\eta$ satisfies
all the required conditions. Note that we separately define $\theta$, because it is used
later on for further purposes (in the subsequent
proofs, our definition of $\eta$ keep changing by multiplying
$\theta$ with appropriate factors).

Hence, we have shown that
\smash{$\overline{\mathsf{A}^{\mathsf{F}}_1}\overline{\mathsf{T}}_{0}$} is a smooth
submanifold of $\mathsf{M}_{1}^{1}$. We now show that
\smash{$\overline{\mathsf{A}^{\mathsf{F}}_1}\overline{\mathsf{T}}_{k_1}$} is a smooth
submanifold of $\mathsf{M}_{1}^{1}$. We show that
\smash{$\overline{\mathsf{A}^{\mathsf{F}}_1}\overline{\mathsf{T}}_{k_1}$} is a smooth submanifold
of \smash{$\overline{\mathsf{A}^{\mathsf{F}}_1}\overline{\mathsf{T}}_{k_1-1}$} of codimension one.
The proof is as before; the required curve is given by
$\eta(x, y) := \theta(x,y)x^{k_1+1}$.
This proves the claim when $n = 1$.

Finally, suppose there are more than one point of tangency (in addition to the nodal point).
The nodal point is at $(a,b)$ and suppose that the tangency points are at
$(\mathbf{a}_1, 0), (\mathbf{a}_2, 0), \dots,\allowbreak (\mathbf{a}_{n-1}, 0)$
and $(0,0)$. The assertion is proved by considering the polynomial
\[ \eta(x,y) := \theta(x,y)(x-\mathbf{a}_1)^{k_1+1} \cdots (x-\mathbf{a}_{n-1})^{k_{n-1}+1} (x-0)^{k_n}. \]
This completes the proof.
 \end{proof}

We are now ready to state and prove the results about enumerating $1$-nodal curves with tangencies.
In the following theorems, it is always be understood that whenever there is a~collection of
numbers $k_1, k_2, \dots, k_n$, then $k$ is defined to be
$k := \sum_{i=1}^n k_i$.

We also recall the projection maps and the diagonal
subspaces that we will be encountering. We denote
$\pi \colon \mathsf{M}^1_{n+1} \longrightarrow \mathsf{M}^1_n$
to be
the projection that forgets the last marked point.
We also denote
$\pi_{n+1} \colon \mathsf{M}^1_{n+1} \longrightarrow \mathsf{M}_1$
to be the map that forgets all the marked points, except the last one.
Next, for any $1 \leq i \leq n$, we denote
$\Delta_{i, n+1}$ to be the following subset of
$\mathsf{M}_{n+1}^1$:
It is the locus of all $(H_1, H_d, p, x_1, \dots, x_{n}, x_{n+1})$
such that $x_i = x_{n+1}$.
Finally, we denote
$\Delta_{n+1}^{1}$
to be the following subset of
$\mathsf{M}_{n+1}^1$:
It is the locus of all $(H_1, H_d, p, x_1, \dots, x_{n}, x_{n+1})$
such that $p = x_{n+1}$.

\begin{thm}\label{theorem_for_many_Tks_A1F}
Let
$n, k_1, k_2, \dots , k_n$ be nonnegative integers.
Then the following equality of homology classes in $H_*\bigl(\mathsf{M}_{{n+1}}^1; \mathbb{R}\bigr)$ holds:
\begin{align}
\pi^{\ast} \bigl[\mathsf{A}_1^{\mathsf{F}} \mathsf{T}_{k_1} \dots \mathsf{T}_{k_n}\bigr]\cdot
\pi_{n+1}^{\ast}[\mathsf{T}_0]  ={}&
\bigl[\mathsf{A}_1^{\mathsf{F}} \mathsf{T}_{k_1}\dots \mathsf{T}_{k_n} \mathsf{T}_{0}\bigr]\nonumber\\
&+ \sum_{i=1}^{n}(1+k_i) \pi^{\ast} \bigl[\mathsf{A}_1^{\mathsf{F}}
\mathsf{T}_{k_1}\dots \mathsf{T}_{k_n}\bigr] \cdot [\Delta_{i, {n+1}}],\label{many_Tks_A1F}
\end{align}
provided $d > k+n+1$.
\end{thm}

\begin{proof} This is simply a straightforward generalization of Theorem \ref{theorem_for_many_Tks}; the proof is identical.
\end{proof}

Next, we generalize Theorem \ref{Tk1Tk2_etc_alt_way}.

\begin{thm}
\label{Tk1Tk2_etc_alt_way_A1F}
Let $n$ be a positive integer, $k_1, k_2, \dots , k_{n-1}$ nonnegative integers
and $k_n$ a~positive integer.
Define
\begin{align*}
\mathsf{m}_{k_n} := \begin{cases} 2 & \text{if} \ k_n=1, \\
1 & \text{if} \  k_n > 1.
\end{cases}
\end{align*}
Then the following equality of homology classes in $H_*\bigl(\mathsf{M}_{{n+1}}^1; \mathbb{R}\bigr)$ holds:
\begin{gather}
\bigl[\mathsf{A}_1^{\mathsf{F}} \mathsf{T}_{k_1} \dots \mathsf{T}_{k_{n-1}} \mathsf{T}_{k_n-1} \mathsf{T}_0\bigr] \cdot
\bigl[\Delta_{n, {n+1}}^{\mathsf{L}}\bigr]  = \pi^*\bigl[\mathsf{A}_1^{\mathsf{F}} \mathsf{T}_{k_1} \dots \mathsf{T}_{k_n}\bigr]\cdot [\Delta_{n, n+1}]   \nonumber \\
\phantom{ \qquad =}{} + \mathsf{m}_{k_n} \pi^{\ast} \bigl[\mathsf{P}^{(k_n-1)}\mathsf{A}_1 \mathsf{T}_{k_1} \dots \mathsf{T}_{k_{n-1}}\bigr]
\cdot [\Delta_{n,n+1}]\cdot \bigl[\Delta^1_{n+1}\bigr], \label{Tmu_T0_nu_alt_way_A1F_aagg_Cycle}
\end{gather}
provided $d > k+n$.
\end{thm}

\begin{rem}
 Let us see how to extract numbers from this.
On the $(n+1)$-pointed space $\mathsf{M}_{n}^1$, let $\alpha$ and $\beta$ be the following classes:
\begin{align*}
\alpha& := y_1^{r} y_d^s b_1^{\nu_{1}}a_1^{\varepsilon_{1}} \dots a_n^{\varepsilon_n}
\qquad \textnormal{and} \qquad \beta := y_1^{r} y_d^s b_1^{\nu_{1}+\varepsilon_n}
a_1^{\varepsilon_{1}} \dots a_{n-1}^{\varepsilon_{n-1}}.
\end{align*}
Intersecting both sides of equation \eqref{Tmu_T0_nu_alt_way_A1F_aagg_Cycle}
with $\alpha$, gives us
\begin{align*}
&\bigl[\mathsf{A}_1^{\mathsf{F}} \mathsf{T}_{k_1} \dots \mathsf{T}_{k_{n-1}} \mathsf{T}_{k_n-1} \mathsf{T}_0\bigr] \cdot
\bigl[\Delta_{n, {n+1}}^{\mathsf{L}}\bigr]\cdot \alpha \\
&\qquad{} = \bigl[\mathsf{A}_1^{\mathsf{F}} \mathsf{T}_{k_1} \dots \mathsf{T}_{k_n}\bigr] \cdot \alpha
+ \mathsf{m}_{k_n} \bigl[\mathsf{P}^{(k_n-1)}\mathsf{A}_1 \mathsf{T}_{k_1} \dots \mathsf{T}_{k_{n-1}}\bigr]
\cdot \beta.
\end{align*}
\end{rem}

\begin{proof}[Proof of Theorem \ref{Tk1Tk2_etc_alt_way_A1F}]
We first prove the special case where
$n = 1$ and $k_n = 1$.
In that case, equation \eqref{Tmu_T0_nu_alt_way_A1F_aagg_Cycle} simplifies to
\begin{align}
\bigl[\mathsf{A}_1^{\mathsf{F}} \mathsf{T}_{0}\mathsf{T}_0\bigr] \cdot \bigl[\Delta_{12}^{\mathsf{L}}\bigr] & =
\pi^*\bigl[\mathsf{A}_1^{\mathsf{F}} \mathsf{T}_{1}\bigr] \cdot [\Delta_{12}] +
2 \pi^*\bigl[\mathsf{A}_1^{\mathsf{L}}\bigr]\cdot [\Delta_{12}]\cdot \bigl[\Delta_{2}^1\bigr].
\label{A1FT0T0_coll_cycle_ver}
\end{align}
The proof of \eqref{A1FT0T0_coll_cycle_ver} builds on what was already shown in
Theorem \ref{Tk1Tk2_etc_alt_way}, namely when two $\mathsf{T}_0$ points collide the first derivative along the line
vanishes. Now there are two possibilities. The first one is that the limiting point is a smooth point of the curve.
This corresponds to the locus~$\mathsf{A}_1^{\mathsf{F}}\mathsf{T}_1$. There is another possibility that
the limiting point is a singular point of the curve. This corresponds to the locus
$\mathsf{A}_1^{\mathsf{L}}$. On the set-theoretic level, this argument shows that the left-hand side of
\eqref{A1FT0T0_coll_cycle_ver} is a subset of its right-hand side. To show that the right-hand side is a~subset of
the left-hand side, we need to show that every element of
$\mathsf{A}_1^{\mathsf{F}}\mathsf{T}_1$ and $\mathsf{A}_1^{\mathsf{L}}$ can be obtained as a~limit of elements
in $\mathsf{A}_1^{\mathsf{F}} \mathsf{T}_{0}\mathsf{T}_0$. It was shown in the proof of Theorem
\ref{Tk1Tk2_etc_alt_way} that every element of~$\mathsf{T}_1$ arises as a limit of elements in
$\mathsf{T}_{0}\mathsf{T}_0$. To complete the proof here, it is enough to show that
every element of
$\mathsf{A}_1^{\mathsf{F}}\mathsf{T}_1$ can be obtained as a limit of elements
in $\mathsf{A}_1^{\mathsf{F}} \mathsf{T}_{0}\mathsf{T}_0$. The argument for it is the same as how we showed
(at the end of the proof of Theorem \ref{Tk1Tk2_etc_alt_way}) that
every element of $\mathsf{T}_2\mathsf{T}_1 $
arises as a limit of elements of $\mathsf{T}_2 \mathsf{T}_0 \mathsf{T}_0$.
The crucial fact that was used there is that~$f_{10}$ is indeed a local coordinate
on the space $\mathsf{T}_2 \overline{\mathsf{T}}_0$.
The proof of Proposition \ref{prp_smth_mfld_with_A1f_ag} shows that
$f_{10}$ is a local coordinate on
$\mathsf{A}_1^{\mathsf{F}}\mathsf{T}_2 \overline{\mathsf{T}}_0$.
This completes the proof about why every element
$\mathsf{A}_1^{\mathsf{F}}\mathsf{T}_1$
can be obtained as a limit of elements
in $\mathsf{A}_1^{\mathsf{F}} \mathsf{T}_{0}\mathsf{T}_0$.
In particular, this justifies the first term on the right-hand side of
equation \eqref{A1FT0T0_coll_cycle_ver}.

The new thing we need to do for completing the proof of
equation \eqref{A1FT0T0_coll_cycle_ver}
(on the set-theoretic level) is to show that every element of
$\mathsf{A}_1^{\mathsf{L}}$ can be obtained as a limit of elements
in $\mathsf{A}_1^{\mathsf{F}} \mathsf{T}_{0}\mathsf{T}_0$. To prove this assertion, switch to affine space.
Let $f$ belong to
$(\mathsf{A}_1^{\mathsf{L}})_{\textnormal{Aff}}$. As before, the line is the $x$-axis. For convenience, set the
nodal point to be the origin. Hence, the expression for $f$ is given by
\begin{align*}
f(x, y)& = \frac{f_{20}}{2} x^2 + f_{11} xy + \frac{f_{02}}{2}y^2 + \mathcal{R}(x, y),
\end{align*}
where the remainder term $\mathcal{R}(x, y)$ is of degree three or higher.

We now try to construct a curve $f_t$ close to $f$, such that $f_t$ passes through the origin, passes through
$(t, 0)$ and has a nodal point close to $(0, 0)$. The expression for $f_t$ is given by
\begin{align*}
f_{t}(x, y) & = f_{t_{10}} x + f_{t_{01}} y + \frac{f_{t_{20}}}{2} x^2 + f_{t_{11}} xy + \frac{f_{t_{02}}}{2}y^2 + \cdots.
\end{align*}
First of all, $f_{t_{10}}$ and $f_{t_{01}}$ are small.
It is required that $f_t$ has a nodal point close to the origin.
Hence, we need to find $(u, v) \neq (0, 0)$ but small, such that
\begin{align}
f_t(u, v) & = 0, \qquad
(f_{t})_x(u,v) = 0 \qquad \textnormal{and} \qquad (f_{t})_y(u, v) = 0. \label{f_fx_fy}
\end{align}
To solve \eqref{f_fx_fy}, using the facts that
$(f_{t})_x(u,v) = 0$
and
$(f_{t})_y(u,v) = 0$
it is deduced that
\begin{align}
f_{t_{10}} & = -vf_{t_{11}}-uf_{t_{20}}-\mathcal{R}_x(u, v) \qquad \textnormal{and} \qquad f_{t_{01}} = -uf_{t_{11}}-vf_{t_{02}}
-\mathcal{R}_y(u, v). \label{f01_10_value}
\end{align}
Plugging in these values for $f_{t_{10}}$ and $f_{t_{01}}$ from \eqref{f01_10_value}, and
using the fact that $f_t(u, v) = 0$, it follows that
\begin{align}
& \frac{f_{t_{20}}}{2} u^2 + f_{t_{11}} uv + \frac{f_{t_{02}}}{2} v^2 + \mathcal{R}_2(u,v)  =  0, \label{xy_zero}
\end{align}
where $\mathcal{R}_2(u, v) := -2\mathcal{R}(u, v)+2u \mathcal{R}_x(u, v) + 2v \mathcal{R}_y(u, v)$.

We now try to solve $u$ in terms of $v$ using \eqref{xy_zero}.
Since the curve has a genuine node at the origin, it may be assumed that the hessian is non-degenerate;
in other words, $f_{t_{20}} f_{t_{02}} - f_{t_{11}}^2$ is nonzero.
Hence, after making a change of coordinates and using the fact that the remainder terms $\mathcal{R}_2(u, v)$
is of order three, it is deduced that there are two solutions to equation \eqref{xy_zero} given~by
\begin{gather*}
u  = \left(\frac{-f_{t_{11}} + \sqrt{f_{t_{11}}^2-f_{t_{20}} f_{t_{02}}}}{ f_{t_{20}}}\right) v + O\bigl(v^2\bigr) \\
\textnormal{or} \qquad u  =
\left(\frac{-f_{t_{11}} - \sqrt{f_{t_{11}}^2-f_{t_{20}} f_{t_{02}}}}{ f_{t_{20}}}\right) v + O\bigl(v^2\bigr).
\end{gather*}
Here \smash{$\sqrt{f_{t_{11}}^2-f_{t_{20}} f_{t_{02}}}$} is a specific branch of the square root, which exists because
$f_{t_{11}}^2-f_{t_{20}} f_{t_{02}}$ is nonzero. Hence, the above solution can be re-written as
\begin{align}
u &  =  A v + O\bigl(v^2\bigr), \qquad \textnormal{where} \quad
A := \left(\frac{-f_{t_{11}} \pm \sqrt{f_{t_{11}}^2-f_{t_{20}} f_{t_{02}}}}{ f_{t_{20}}}\right). \label{u_Av}
\end{align}
Now impose the condition that the curve passes through the point $(t, 0)$.
Using \eqref{u_Av} and the condition that $f_t(t, 0) = 0$, it follows that
\begin{align}
t&  =  2\left(\frac{f_{t_{11}} + A f_{t_{20}}}{f_{t_{20}}} \right)v + O\bigl(v^2\bigr). \label{t_zero_mult}
\end{align}
Hence, the condition of making the two points equal (namely setting $t$ = 0) has a multiplicity, which is
is given by \eqref{t_zero_mult}. For each value of $A$,
the multiplicity is one. Since there are two possible values of $A$ (corresponding to the branch of the square-root chosen),
the total multiplicity is two. Hence, when we intersect with the class $[q_1 = q_2]$, each branch of
$\mathsf{A}_1^{\mathsf{L}}$ contributes with a multiplicity of $1$ resulting in a total multiplicity is $2$. By a branch, we refer to each distinct solution to equation \eqref{f_fx_fy} of a neighbourhood of $\mathsf{A}_1^{\mathsf{F}} \mathsf{T}_{0}\mathsf{T}_{0}$ inside \smash{$\overline{\mathsf{A}}_1^{\mathsf{L}}$}.
This property of multiplicity finally shows that \eqref{A1FT0T0_coll_cycle_ver} is true on the level of
homology.

We now prove
the next case where $n = 1$ and $k_n = k-1$ with $k \geq 3$.
In that case, equation~\eqref{Tmu_T0_nu_alt_way_A1F_aagg_Cycle}
simplifies to
\begin{equation}\label{A1FT0T1_coll}
\bigl[\mathsf{A}_1^{\mathsf{F}} \mathsf{T}_{k-2}\mathsf{T}_{0}\bigr]
\cdot \bigl[\Delta_{12}^{\mathsf{L}}\bigr]
 =  \bigl[\mathsf{A}_1^{\mathsf{F}} \mathsf{T}_{k-1}\bigr]\cdot [\Delta_{12}] +
\bigl[\mathsf{P}^{(k-2)}\mathsf{A}_1\bigr]\cdot [\Delta_{12}] \cdot \bigl[\Delta_{2}^1\bigr].
\end{equation}
In order to prove \eqref{A1FT0T1_coll}, we
build on what has already been proved in Theorem \ref{Tk1Tk2_etc_alt_way}.
The justification for the first term on the right-hand side of
\eqref{A1FT0T1_coll} is the same as in the proof of Theorem \ref{Tk1Tk2_etc_alt_way}.
In order to show that every $\mathsf{A}_1^{\mathsf{F}}\mathsf{T}_{k-1}$ point can be
obtained as a limit of
$\mathsf{A}_1^{\mathsf{F}}\mathsf{T}_{k-1}$, we use the fact that
$f_{k-1, 0}$ is indeed a local coordinate on
$\mathsf{A}_1^{\mathsf{F}} \overline{\mathsf{T}}_{k-2}$ as seen in the proof of Proposition~\ref{prp_smth_mfld_with_A1f_ag}.

The new thing needed is to justify the second term on the right-hand side of \eqref{A1FT0T1_coll}.
In particular, it is needed to show that
every element of
$\mathsf{P}^{(k-2)}\mathsf{A}_1$ can be obtained as a limit of elements
in $\mathsf{A}_1^{\mathsf{F}} \mathsf{T}_{0}\mathsf{T}_{k-2}$. To prove this assertion,
restrict, as before, to affine space.

Let $f$ be a curve that belongs to $\mathsf{P}^{(k-2)}\mathsf{A}_1$; the line is set to be the $x$-axis and the nodal point
is set to be the origin. We try to construct a curve
$f_t$ that is tangent to the
$x$-axis at the origin to order $k-2$ (i.e., it is an element of $\mathsf{T}_{k-2}$). It is given by
\begin{align*}
f_t(x, y)& = u y + \mathsf{P} x^{k-1} + f_{t_{11}} x y + \frac{f_{t_{02}}}{2} y^2 +
\mathsf{Q} x^k + \mathcal{R} (x, y), \qquad
\textnormal{where} \\
\mathcal{R} (x, y)& := x^{k+1}\mathsf{A} (x) + y x^2 \mathsf{B} (x) + y^2 \mathsf{C} (x,y) \qquad \textnormal{and}
\qquad \mathsf{C} (0, 0) = 0.
\end{align*}
Notice that $f_{t_{01}}$ is written as $u$. Assume that $\mathsf{Q} \neq 0$.
Next, impose the condition that the curve also passes through $(t, 0)$.
In other words, $f_t(t, 0) = 0$. Using this equation and dividing out by~$t$, it follows that
$
\mathsf{P}  =  - \mathsf{Q} t + O\bigl(t^2\bigr)$.
Now impose the condition that the curve has a node at~$(x, y)$. Hence,
$
f_{t}(x, y)  = 0$, $ (f_{t})_x(x, y) = 0 $ and $(f_{t})_y(x, y) = 0$.
We are looking for solutions where $u$, $t$, $x$ and $y$ are small and $(x, y) \neq (0, 0)$.
Using the equation~${(f_{t})_y(x, y) = 0}$, it follows that
\begin{align}
x &  =  -\frac{u}{f_{t_{11}}} -\frac{y f_{t_{02}}}{f_{t_{11}}} + \mathcal{E}_2 (u, y),
\label{x_u}
\end{align}
where the error term $\mathcal{E}_2 (u, y)$ is of second order in $(u, y)$.
Next, plug this in the equation $(f_{t})_x(x, y) = 0$ and solve for $y$ in terms of $u$ and $t$. This produces
\begin{align}
(-1)^ky &  =  \frac{k \mathsf{Q}}{f_{t_{11}}^k} u^{k-1} + \frac{(k-1)\mathsf{Q}}{f_{t_{11}}^{k-1}} u^{k-2} t + \mathcal{E}_{k} (u, t), \label{y_u}
\end{align}
where the error term $\mathcal{E}_k(u, t)$ is of order $k$ in $(u, t)$.
Plugging all this in $f_{t}(x, y) = 0$ gives an implicit relationship between $u$ and $t$.
Note that $u$ cannot be zero.
This is because if $u$ were zero, then $y$ would be zero (equation \eqref{y_u})
and as a result $x$ could be zero (equation \eqref{x_u}). This is a contradiction,
since we are looking for solutions
where $x$ and $y$ are not both equal to zero.

Next, notice that the expression for $f_{t}(x, y)$ contains a factor of $u^3$.
Since $u \neq 0$, we can cancel off the factor of $u^3$ and get a
simplified implicit expression for $u$ and $t$. Now we can directly solve for $u$ in terms of $t$ and conclude that
$
u  =  -f_{t_{11}} t + O\bigl(t^2\bigr)$.
Plugging this back into the expression for $x$, gives
\begin{align}
x &  =  t + O\bigl(t^2\bigr). \label{x_mult}
\end{align}
These solutions are the only solutions.
Hence, from the expression for $x$ (namely \eqref{x_mult}), it follows that the multiplicity of the intersection is one.
This proves equation \eqref{A1FT0T1_coll}.

Notice that the nearby curve was constructed by specifying the value of
$u$ (which is~$f_{01}$) and~$\mathsf{P}$ (which, up to a constant factor of~$(k-1)!$ is $f_{k-1, 0}$). Now we note that on top of~$\overline{\mathsf{T}}_{k-2}$, both~$f_{01}$ and $f_{k-1,0}$ are indeed
local coordinates (as seen in the proofs of
Propositions \ref{prp_smth_mfld} and~\ref{prp_smth_mfld_with_PA1r_ag}).
The general case of Theorem \ref{Tk1Tk2_etc_alt_way_A1F} (when $n>1$)
now follows in an identical way.
\end{proof}

These next few results enable us to enumerate curves
tangent to a line at multiple points and one node lying on the line, such that one of the
branches of the node is tangent to the line to some given order.
We state the first one of these results.

\begin{thm}\label{theorem_for_many_tks_ag_pakL}
Let $n$, $r$ and $k_1, k_2,\dots, k_n$ be nonnegative integers.
Then the following equality of homology classes of $\mathsf{M}_{n+1}^1$ holds:
\begin{gather}
\pi^{\ast} \bigl[\mathsf{P}^{(r)} \mathsf{A}_1 \mathsf{T}_{k_1} \mathsf{T}_{k_2} \dots \mathsf{T}_{k_n} \bigr]\cdot
\pi_{n+1}^{\ast} [\mathsf{T}_0]\nonumber \\
\qquad =
\bigl[\mathsf{P}^{(r)} \mathsf{A}_1 \mathsf{T}_{k_1}\mathsf{T}_{k_2} \dots \mathsf{T}_{k_n} \mathsf{T}_{0}\bigr]
 + \sum_{i=1}^{n} (k_i+1) \pi^{\ast} \bigl[\mathsf{P}^{(r)} \mathsf{A}_1 \mathsf{T}_{k_1} \mathsf{T}_{k_2}\dots \mathsf{T}_{k_n}\bigr] \cdot [\Delta_{i, {n+1}}] \nonumber \\
\phantom{\qquad =}{} + (r+2) \pi^{\ast} \bigl[\mathsf{P}^{(r)} \mathsf{A}_1 \mathsf{T}_{k_1} \mathsf{T}_{k_2}
\dots \mathsf{T}_{k_n}\bigr] \cdot \bigl[\Delta_{{n+1}}^1\bigr],\label{many_tks_ag_pakL}
\end{gather}
provided $d > n+k+r+2$.
\end{thm}

\begin{proof}
This is a generalization of Theorem \ref{theorem_for_many_Tks}.
The new thing needed is to justify the third term on the right-hand side of \eqref{many_tks_ag_pakL}.
The special case of the theorem where $n = 0$ will be proved first. In this special case
\eqref{many_tks_ag_pakL} simplifies to
\begin{align}
\pi^*\bigl[\mathsf{P}^{(r)}\mathsf{A}_1\bigr] \cdot \pi_1^*[\mathsf{T}_0] &  =  \bigl[\mathsf{P}^{(r)}\mathsf{A}_1 \mathsf{T}_0\bigr] +
(r+2) \pi^*\bigl[\mathsf{P}^{(r)}\mathsf{A}_1\bigr]\cdot \bigl[\Delta_{1}^1\bigr]. \label{pa1_t0_sp}
\end{align}
On the set-theoretic level, the justification is the same as before
(the first term corresponds to when the two marked points are distinct, while the
second term corresponds to the case where the two marked points coincide).
The reason that
the first term on the right-hand side of \eqref{pa1_t0_sp}
appears with a multiplicity of one is because the intersections are transverse
(this is the content of Proposition \ref{prp_smth_mfld_with_PA1r_ag}).

The new thing that we have to justify is the multiplicity of $(r+2)$
for the second term on the right-hand side; let us justify that.
As before, we switch to affine space.
Set the $\mathsf{P}^{(r)}\mathsf{A}_1$ point to
be the origin.
The situation now is that we have a nodal curve such that one of the branches of the node is tangent to
the $x$-axis to order $r$ (at he origin).
Hence, the curve $f$ is such that~${f_{00}, f_{10}, f_{01}, f_{20},\dots, f_{r+1, 0}}$ all vanish. The function
$f$ is given by
\begin{align*}
f(x, y)&  =  \frac{f_{r+2, 0}}{(r+2)!} x^{r+2} + y\mathcal{R}(x, y),
\end{align*}
where $y \mathcal{R}(x, y)$ is a remainder term. Now consider the evaluation map
\begin{align*}
\varphi(f, x) :=  f(x,0)  = \frac{f_{r+2, 0}}{(r+2)!} x^{r+2}.
\end{align*}
The order of vanishing of $\varphi$ is clearly $r+2$, provided
$f_{r+2, 0} \neq 0$. But that assumption is valid, since to compute the order of vanishing, we will be intersecting with
cycles that correspond to constraints being generic. Hence, the order of vanishing is $r+2$. This proves
\eqref{pa1_t0_sp} on the level of homology.

The general statement of Theorem \ref{theorem_for_many_tks_ag_pakL} is now similar to how Theorem \ref{theorem_for_many_Tks}
is proved.
\end{proof}

We prove the next theorem.

\begin{thm}
\label{Node_whose_one_brunch_tangent at multi}
Let $n$ be a positive integer, $k_1, k_2,\dots, k_{n-1}$ nonnegative integers
and $k_n$ a~positive integer.
Then the following equality of homology classes
holds in $H_*\bigl(\mathsf{M}^1_{n+1}; \mathbb{R}\bigr)$:
\begin{align}
\bigl[\mathsf{P}^{(r)} \mathsf{A}_1 \mathsf{T}_{k_1} \dots \mathsf{T}_{k_{n-1}} \mathsf{T}_{k_n-1} \mathsf{T}_0\bigr] \cdot
\bigl[\Delta_{n, {n+1}}^{\mathsf{L}}\bigr]&  =
\pi^*\bigl[\mathsf{P}^{(r)} \mathsf{A}_1 \mathsf{T}_{k_1} \dots \mathsf{T}_{k_n}] \cdot [\Delta_{n, n+1}\bigr], \label{Tmu_T0_nu_alt_way_A1F}
\end{align}
provided $d > n+k+r+1$.
\end{thm}
\begin{proof}
This is a generalization of Theorem \ref{Tk1Tk2_etc_alt_way}; the proof is the same.
\end{proof}

The final result is as follows.

\begin{thm}
\label{Node_whose_one_brunch_tangent}
Let $r$ be a positive integer.
Then the following equality of homology classes hold in $H_*\bigl(\mathsf{M}^1_1; \mathbb{R}\bigr)$:
\begin{align}
\bigl[\mathsf{P}^{(r-1)} \mathsf{A}_1 \mathsf{T}_0\bigr]\cdot
\bigl[\bigl(\Delta^{1}_{1}\bigr)^{\mathsf{L}}\bigr]
& =  \bigl[\mathsf{P}^{(r)} \mathsf{A}_1\bigr] \cdot \bigl[\Delta^1_1\bigr], \label{p_tang_plower_tang}
\end{align}
provided $d > r+1$ and where the class
$\bigl[\bigl(\Delta^{1}_{1}\bigr)^{\mathsf{L}}\bigr]$ is defined
as
\[ \bigl[\bigl(\Delta^{1}_{1}\bigr)^{\mathsf{L}}\bigr]:= b_1 + a_1-y_1. \]
\end{thm}

\begin{proof}
First \eqref{p_tang_plower_tang} will be proved on the set-theoretic level.
We switch to affine space. As before, the line is the $x$-axis.
Set the $\mathsf{P}^{(r-1)} \mathsf{A}_1$ point to be the origin
and the $\mathsf{T}_0$ point to be~$(t, 0)$. Let $f_t$ be a curve that has a
$\mathsf{P}^{(r-1)} \mathsf{A}_1$ point at the origin and a $\mathsf{T}_0$ point at $(t, 0)$.
The former condition says that the first $r$ derivatives with respect to $x$ vanish at the origin.
The fact that the curve also passes through $(t, 0)$ tells that the $(r+1)$-th derivative
is given by
\[f_{r+1,0}  =  -\left(\frac{f_{r+2,0}}{r+2}\right)t + O\bigl(t^2\bigr).\]
Hence, as $t$ goes to zero, $f_{r+1,0}$ vanishes. Furthermore, the curve has a node at the origin. Hence, in the limit,
the curve belongs to $\mathsf{P}^{(r)} \mathsf{A}_1$.

To complete the proof on the set-theoretic level, it suffices
to show that every element of~$\mathsf{P}^{(r)} \mathsf{A}_1$ arises as a limit of elements of
$\mathsf{P}^{(r-1)} \mathsf{A}_1 \mathsf{T}_0$. The proof is exactly the same as how we show every element of
$\mathsf{T}_{r+1}$ is a limit of elements of $\mathsf{T}_r \mathsf{T}_0$. The multiplicity of the intersection also follows in the
same way. This proves \eqref{p_tang_plower_tang} on the level of homology and completes the proof Theorem~\ref{Node_whose_one_brunch_tangent}.
\end{proof}

Finally, we note that the expression for the homology class
$\bigl[\mathsf{A}_1^{\mathsf{F}}\bigr]$
can be computed using the
results of \cite{R.M}.
We give a new way to derive that expression in
Section \ref{singularity_in_our_format}.
For now, assume that
the expression for $\bigl[\mathsf{A}_1^{\mathsf{F}}\bigr]$ is known
(which in Section \ref{node_new}, given by \eqref{a1_free_cycle}).

We now explain how to compute
all the characteristic numbers
involving the class $\bigl[\mathsf{A}_1^{\mathsf{F}} \mathsf{T}_{k_1} \dots\allowbreak \mathsf{T}_{k_n}\bigr]$.
Using equations
\eqref{many_Tks_A1F} and \eqref{Tmu_T0_nu_alt_way_A1F}, we can reduce it to a
question of computing characteristic numbers
involving the classes
$\bigl[\mathsf{A}_1^{\mathsf{F}} \mathsf{T}_{k_1} \dots \mathsf{T}_{k_n-1}\bigr]$
and
\smash{$\bigl[\mathsf{P}^{(k_n-1)}\mathsf{A}_1 \mathsf{T}_{k_1} \dots \mathsf{T}_{k_{n-1}}\bigr]$}.
Using equations~\eqref{many_tks_ag_pakL} and \eqref{Tmu_T0_nu_alt_way_A1F}, this
ultimately reduces to
the computation of all intersection numbers involving the class $\bigl[\mathsf{P}^{(r)}\mathsf{A}_1\bigr]$.
Using equations \eqref{p_tang_plower_tang} and
\eqref{many_tks_ag_pakL},
this finally reduces
to the computation of~$\bigl[\mathsf{A}_1^{\mathsf{L}}\bigr]$; this can be computed from
equations \eqref{a1_free_cycle} and
\eqref{a1_line_cycle}. Hence, all the intersection numbers can be computed.

\section[Counting 1-cuspidal curves with multiple tangencies of first order]{Counting $\boldsymbol{ 1}$-cuspidal curves\\ with multiple tangencies of first order}
\label{Cuspidal_tang}
In this section, we show how to enumerate one cuspidal curves with first-order tangencies.
We continue with the set up and notation of Section \ref{nodal_tang}.
In addition, we need to define two new spaces.
Define
$\mathsf{A}_2^{\mathsf{F}} \subset \mathsf{M}^1_0$ to be the subset of all
$(H_1, H_d, p) \in \mathsf{M}^1_0$ such that
$H_d$ has a cusp at $p$.
Similarly, define
$\mathsf{A}_2^{\mathsf{L}} \subset \mathsf{M}^1_0$ to be the subset of all
$(H_1, H_d, p) \in \mathsf{M}^1_0$ such that
\begin{itemize}\itemsep=0pt
\item The curve $H_d$ has a cusp at $p$.
\item The point $p$ lies on the line $H_1$.
\end{itemize}
Note that
\begin{align}
\bigl[\mathsf{A}_2^{\mathsf{L}}\bigr] &  =  \bigl[\mathsf{A}_2^{\mathsf{F}}\bigr]\cdot (y_1 + b_1);
\label{a2_line_cycle}
\end{align}
this is because intersecting with $(y_1+b_1)$ corresponds to the point $p$ lying on the line.
Let us now prove a few transversality results.
In the following propositions,
$n$, $k_1, k_2,\dots, k_{n}$ are nonnegative integers
and
$k := k_1+k_2+\dots +k_n$.

\begin{prp}
\label{prp_smth_mfld_with_A2f_ag}
The space
$\mathsf{A}^{\mathsf{F}}_2\overline{\mathsf{T}}_{k_1}\dots \overline{\mathsf{T}}_{k_n}$
is a smooth submanifold of $\mathsf{M}_{n}^{1}$, provided $d \geq k+n+2$.
\end{prp}

\begin{prp}
\label{prp_smth_mfld_with_A2L_ag}
The space
$\mathsf{A}^{\mathsf{L}}_2\overline{\mathsf{T}}_{k_1}\dots \overline{\mathsf{T}}_{k_n}$
is a smooth submanifold of $\mathsf{M}_{n}^{1}$, provided $d \geq k+n+2$.
\end{prp}

\begin{proof}[Proof of Proposition \ref{prp_smth_mfld_with_A2f_ag}]
We show that
$\mathsf{A}^{\mathsf{F}}_2\overline{\mathsf{T}}_{k_1}\dots \overline{\mathsf{T}}_{k_n}$
is a codimension one submanifold of~\smash{$\overline{\mathsf{A}^{\mathsf{F}}_1}\overline{\mathsf{T}}_{k_1}\dots \overline{\mathsf{T}}_{k_n}$}.
We first prove it for $n = 0$, i.e., we show that
$\mathsf{A}^{\mathsf{F}}_2$ is smooth codimension one submanifold of
\smash{$\overline{\mathsf{A}^{\mathsf{F}}_1}$}.
As before, we switch to affine space.
Let \smash{$\bigl(\overline{\mathsf{A}_1^{\mathsf{F}}}\bigr)_{\textnormal{Aff}}$}
be as defined in equation \eqref{A1F_affine}. Define the map
\[
\varphi \colon\ \bigl(\overline{\mathsf{A}_1^{\mathsf{F}}}\bigr)_{\textnormal{Aff}} \longrightarrow \mathbb{C},\qquad
(f, (x, y)) \longmapsto \bigl(f_{xx} f_{yy} -f_{xy}^2\bigr)(x, y).
\]
It suffices to show that whenever $(a, b)$ is a genuine cuspidal point of $f$, the differential of
$\varphi$
is surjective.

In order to prove that claim,
assume that $(f, (a, b)) \in \varphi^{-1}(0)$ and that
$(a, b)$ is a genuine cuspidal point of $f$.
In that case, $f_{xx}(a, b)$ and $f_{yy}(a, b)$
both can't be zero, because otherwise, even $f_{xy}(a, b)$ would be zero, making $(a, b)$ a triple
point of $f$ (i.e., it is not a genuine cusp). Assume that $f_{yy}(a, b) \neq 0$.
Define the polynomial $\eta$ given by
$\eta(x,y):= (x-a)^2$.
Define the curve $\gamma$ given by
$\gamma(t) :=  (f+t \eta, x, y)$.
We note that $\gamma(t)$ lies in
\smash{$\bigl(\overline{\mathsf{A}_1^{\mathsf{F}}}\bigr)_{\textnormal{Aff}}$} since $t$ is nonzero but small.
We now note that
\begin{align*}
\{{\rm d}\varphi|_{(f,(a,b))}\}(\gamma^{\prime}(0))& = f_{yy}(a,b) \eta_{xx}(a,b).
\end{align*}
Since $f_{yy}(a,b) \neq 0$ by assumption and
$\eta_{xx}(a,b) \neq 0$ by our construction of $\eta$, we conclude that the differential is surjective.
Note that if $f_{yy}(a,b)=0$ (which means in turn that $f_{xx}(a,b)\neq0$), then
we would have defined $\eta(x,y):= (y-b)^2$. In that case, we would get{\samepage
\begin{align*}
\{{\rm d}\varphi|_{(f,(a,b))}\}(\gamma^{\prime}(0))& = f_{xx}(a,b) \eta_{yy}(a,b).
\end{align*}
This is again nonzero. This completes the proof of the proposition
for $n = 0$.}

The rest of the proof (for $n \geq 1$)
is now identical to the proof of
Proposition \ref{prp_smth_mfld_with_A1f_ag}. The only change we need to make is in
equation \eqref{theta_defn} where we define $\theta$. That definition is replaced by
\begin{align}
\label{theta_defn_ag}
\theta(x,y)&:= \begin{cases} (y-b)^3 &  \text{if} \  b \neq 0, \\
              (x-a)^3 &  \text{if} \ b =0, \ a\neq 0.
       \end{cases}
\end{align}
Modulo that redefinition of $\theta$, the proof is identical.
\end{proof}

\begin{proof}[Proof of Proposition \ref{prp_smth_mfld_with_A2L_ag}]
We show that
$\mathsf{A}^{\mathsf{L}}_2\overline{\mathsf{T}}_{k_1}\dots \overline{\mathsf{T}}_{k_n}$
is a codimension one submanifold of~\smash{$\overline{\mathsf{A}^{\mathsf{L}}_1}\overline{\mathsf{T}}_{k_1}\dots \overline{\mathsf{T}}_{k_n}$}.
We first prove it for $n = 0$, i.e., we show that
$\mathsf{A}^{\mathsf{L}}_2$ is smooth codimension one submanifold of
\smash{$\overline{\mathsf{A}^{\mathsf{L}}_1}$}.
This is identical to the first part of the
proof of Proposition \ref{prp_smth_mfld_with_A2f_ag}, where we show that
$\mathsf{A}^{\mathsf{F}}_2$ is smooth codimension one submanifold of
\smash{$\overline{\mathsf{A}^{\mathsf{F}}_1}$}.

The proof for $n\geq 1$ is also identical to the second part of the proof of
Proposition \ref{prp_smth_mfld_with_A2f_ag}. The only point to note here is that
now since all the points lie on the line, $\theta$ is unambiguously defined
via equation \eqref{theta_defn_ag}, i.e., $\theta(x, y) := (x-a)^3$.
\end{proof}

We are now ready to present our main results.
The following result enumerates all one cuspidal curves with tangencies
at multiple points but all are of
first order.

\begin{thm}
\label{theorem_for_many_Tks_A2F}
Let $n$ be a nonnegative integer.
Then the following equality of classes in $H_* \bigl(\mathsf{M}_{{n+1}}^1, \mathbb{R}\bigr)$ holds:
\begin{align*}
\pi^{\ast}\bigl[\mathsf{A}_2^{\mathsf{F}} \underbrace{ \mathsf{T}_{1} \mathsf{T}_{1} \dots \mathsf{T}_{1}}_{n}\bigr]\cdot
\pi_{n+1}^{\ast} [\mathsf{T}_0]  =
\bigl[\mathsf{A}_2^{\mathsf{F}} \underbrace{\mathsf{T}_{1} \mathsf{T}_{1} \dots \mathsf{T}_{1}}_{n} \mathsf{T}_{0}\bigr]
 + \sum_{i=1}^{n} 2 \pi^{\ast}\bigl[\mathsf{A}_2^{\mathsf{F}}
\underbrace{\mathsf{T}_{1} \mathsf{T}_{1} \dots \mathsf{T}_{1}}_{n}\bigr] \cdot [\Delta_{i, {n+1}}],
\end{align*}
provided $d > 2n+2$.
\end{thm}

\begin{proof} This is simply a generalization of Theorems \ref{theorem_for_many_Tks} and
\ref{theorem_for_many_Tks_A1F}; the proof is identical.
\end{proof}

\begin{thm}
\label{theorem_for_many_Tks_A2F_2_T0}
Let $n$ be a nonnegative integer.
Then the following equality of classes in $H_*\bigl(\mathsf{M}_{{n+2}}^1, \mathbb{R}\bigr)$ holds:
\begin{gather*}
\pi^{\ast}\bigl[\mathsf{A}_2^{\mathsf{F}} \underbrace{ \mathsf{T}_{1} \mathsf{T}_{1} \dots \mathsf{T}_{1}}_{n} \mathsf{T}_0\bigr]\cdot
\pi_{n+2}^{\ast} [\mathsf{T}_0]\\
\qquad =
\bigl[\mathsf{A}_2^{\mathsf{F}} \underbrace{\mathsf{T}_{1} \mathsf{T}_{1} \dots \mathsf{T}_{1}}_{n} \mathsf{T}_{0} \mathsf{T}_0\bigr] + \sum_{i=1}^{n} 2 \pi^{\ast}\bigl[\mathsf{A}_2^{\mathsf{F}}
\underbrace{\mathsf{T}_{1} \mathsf{T}_{1} \dots \mathsf{T}_{1}}_{n} \mathsf{T}_0\bigr] \cdot [\Delta_{i, {n+2}}] \nonumber \\
 \phantom{\qquad =}{}+ \pi^{\ast}\bigl[\mathsf{A}_2^{\mathsf{F}}
\underbrace{\mathsf{T}_{1} \mathsf{T}_{1} \dots \mathsf{T}_{1}}_{n} \mathsf{T}_0\bigr] \cdot [\Delta_{{n+1}, {n+2}}], \label{many_Tks_A2F_2_T0}
\end{gather*}
provided $d > 2n+3$.
\end{thm}

\begin{proof} This is simply a generalization of Theorems \ref{theorem_for_many_Tks} and \ref{theorem_for_many_Tks_A1F}; the proof is identical.
\end{proof}

\begin{thm}\label{Tk1Tk2_etc_alt_way_A2F}
Let $n$ be a positive integer.
Then the following equality of homology classes hold in
$H_*\bigl(\mathsf{M}^1_{n+1}; \mathbb{R}\bigr)$:
\begin{align}
&\bigl[\mathsf{A}_2^{\mathsf{F}} \underbrace{\mathsf{T}_{1} \dots \mathsf{T}_{1}}_{n-1} \mathsf{T}_{0} \mathsf{T}_0\bigr]
\cdot \bigl[\Delta_{n, {n+1}}^{\mathsf{L}}\bigr]\nonumber\\
&\qquad{} = \bigl[\mathsf{A}_2^{\mathsf{F}} \underbrace{\mathsf{T}_{1} \dots\mathsf{T}_{1}}_n\bigr] \cdot \alpha
+ 3 \pi^{\ast}\bigl[\mathsf{A}_2^{\mathsf{L}} \underbrace{ \mathsf{T}_{1} \dots \mathsf{T}_{1}}_{n-1}\bigr]
\cdot [\Delta_{n, n+1}]\cdot \bigl[\Delta^1_{n+1}\bigr], \label{Tmu_T0_nu_alt_way_A2F}
\end{align}
provided $d > 2n+2$.
\end{thm}

\begin{rem} Let us explain how to extract numbers.
On $\mathsf{M}^1_{n}$, consider the following classes, given by
\[ \alpha := y_1^{r} y_d^{s}b_1^{\nu_1} a_1^{\varepsilon_1} \dots a_n^{\varepsilon_n}
\qquad \textnormal{and} \qquad \beta := y_1^{r} y_d^{s}b_1^{\nu_1+ \varepsilon_n}
a_1^{\varepsilon_1} \dots a_{n-1}^{\varepsilon_{n-1}}.\]
Intersecting both sides of equation \eqref{Tmu_T0_nu_alt_way_A2F} gives us
\begin{align*}
\bigl[\mathsf{A}_2^{\mathsf{F}} \underbrace{\mathsf{T}_{1} \dots \mathsf{T}_{1}}_{n-1} \mathsf{T}_{0} \mathsf{T}_0\bigr]
\cdot \bigl[\Delta_{n, {n+1}}^{\mathsf{L}}\bigr]\cdot \alpha
& = \bigl[\mathsf{A}_2^{\mathsf{F}} \underbrace{\mathsf{T}_{1} \dots\mathsf{T}_{1}}_n\bigr] \cdot \alpha
+ 3 \bigl[\mathsf{A}_2^{\mathsf{L}} \underbrace{ \mathsf{T}_{1} \dots \mathsf{T}_{1}}_{n-1}\bigr]
\cdot \beta.
\end{align*}
\end{rem}
\begin{proof}
This is a generalization of Theorems \ref{Tk1Tk2_etc_alt_way} and \ref{Tk1Tk2_etc_alt_way_A1F}.
We first prove
the special case where~${n = 1}$.
Then \eqref{Tmu_T0_nu_alt_way_A2F} simplifies to
\begin{align}
\bigl[\mathsf{A}_2^{\mathsf{F}} \mathsf{T}_{0}\mathsf{T}_0\bigr] \cdot \bigl[\Delta_{12}^{\mathsf{L}}\bigr] &  =
\pi^*\bigl[\mathsf{A}_2^{\mathsf{F}} \mathsf{T}_{1}\bigr] \cdot [\Delta_{12}] +
3 \pi^*\bigl[\mathsf{A}_2^{\mathsf{L}}\bigr]\cdot [\Delta_{12}]\cdot \bigl[\Delta_{1}^1\bigr].
\label{A2FT0T0_coll_cycle_ver}
\end{align}
To prove \eqref{A2FT0T0_coll_cycle_ver},
we proceed in a similar way to how we proved
equation \eqref{A1FT0T0_coll_cycle_ver}.
The first term on the right-hand side of
\eqref{A2FT0T0_coll_cycle_ver} is justified in the same way as the first term on the
right-hand side of equation \eqref{A1FT0T0_coll_cycle_ver}. Next, we justify the
second term.

On the set-theoretic level, we need to show that every
point of $\mathsf{A}_2^{\mathsf{L}}$ arises as a limit of $\mathsf{A}_2^{\mathsf{F}}
\mathsf{T}_0 \mathsf{T}_0$. As before, we will be working in affine space. Take \smash{$f \in
\bigl(\mathsf{A}_2^{\mathsf{L}}\bigr)_{\textsf{Aff}}$}. As usual, the line is the $x$-axis and the cuspidal point is the
origin. Assume that $f_{20} \neq 0$; here $f_{ij}$ denotes the $(i, j)$-th derivative at
the origin. Let $f_t$ be a curve close to $f$ that passes through the origin and also passes through~${(t, 0)}$.
The Taylor expansion of $f_t$ is
\begin{align*}
f_{t}(x, y) ={}& uy + \mathsf{P} x + \frac{f_{t_{20}}}{2} x^2 + f_{t_{11}} xy +
\frac{1}{2}\left(s + \frac{f_{t_{11}}^2}{f_{t_{20}}} \right) y ^2 +
\frac{f_{t_{30}}}{6} x^3
\\
&+\frac{f_{t_{21}}}{2} x^2 y + \frac{f_{t_{12}}}{2} x y^2 + \frac{f_{t_{03}}}{6} y^3 +
\mathcal{E}_4(x, y),
\end{align*}
where the error term $\mathcal{E}_4(x, y)$ is fourth order in $(x, y)$. Here
$f_{t_{01}}$ is denoted by $u$ and $f_{t_{10}}$ is denoted by $\mathsf{P}$.

Now impose the condition that the curve passes through $(t, 0)$, so $f_t(t, 0) = 0$.
This implies that
\[
\mathsf{P}  =  -\frac{f_{t_{20}}}{2}t + O\bigl(t^2\bigr).
\]
Now impose the condition that the curve has a cuspidal point at $(x, y)$. This implies that
\begin{align*}
&f_{t}(x, y)  = 0, \qquad (f_{t})_x(x, y) = 0, \qquad
(f_{t})_y(x, y)= 0  \qquad \text{and} \nonumber \\
&H_{f_t}(x,y) := \bigl((f_{t})_{xx} (f_{t})_{yy} - (f_{t})_{xy}^2\bigr)(x,y)= 0.
\end{align*}
Using the conditions that $(f_{t})_x(x,y) = 0$ and $(f_{t})_y(x,y) = 0$, it follows that
\begin{gather}
\frac{f_{t_{20}}}{2} t  = yf_{t_{11}} +x f_{t_{20}} + \frac{f_{t_{12}}}{2} y^2 + f_{t_{21}} xy +
\frac{f_{t_{30}}}{2} x^2 +
\mathcal{E}_3(x, y) \qquad \text{and}\nonumber\\
-u  = \left(s+\frac{f_{t_{11}}^2}{f_{t_{20}}}\right) y + x f_{t_{11}} + \frac{f_{t_{03}}}{2} y^2 + f_{t_{12}} xy+ \frac{f_{t_{21}}}{2} x^2 +
\widetilde{\mathcal{E}}_3(x, y).\label{Hj1}
\end{gather}
Next, use the condition $H_{f_t}(x,y) = 0$ to conclude that
\[
-s = \left(f_{t_{12}} + \frac{f_{t_{11}}^2 f_{t_{30}}}{f_{t_{20}}^2} \right) x + \left(f_{t_{30}}
-\frac{2 f_{t_{11}} f_{t_{12}}}{f_{t_{20}}} + \frac{f_{t_{11}}^2 f_{t_{21}}}{f_{t_{20}}^2} \right) y.
\]
Finally, plugging all these in $f_t(x, y) = 0$, it follows that
\begin{gather}
-\frac{(f_{t_{20}}x + f_{t_{11}} y)^2}{2 f_{t_{20}}} - \frac{f_{t_{30}}}{3} x^3 -f_{t_{21}}x^2 y
+ xy^2 \left(\frac{f_{t_{11}}^2 f_{t_{30}}}{2 f_{t_{20}}^2} - \frac{f_{t_{12}}}{2}\right) \nonumber \\
\qquad+ \left( \frac{f_{t_{03}}}{6} - \frac{f_{t_{11}} f_{t_{12}}}{f_{t_{20}}}+
\frac{f_{t_{11}}^2 f_{t_{21}}}{2 f_{t_{20}}^2} \right) y^3 = 0. \label{a2_coll}
\end{gather}
Next, make the following change of coordinates:
\begin{align}
\widehat{x} & :=  f_{t_{20}}x + f_{t_{11}} y \qquad \textnormal{and} \qquad \widehat{y}
 :=  y. \label{hat_x_Hj3}
\end{align}
Note that this is a valid change of coordinate because $f_{t_{20}} \not= 0$.
Under a further genericity assumption on the third derivatives, we can make a change of coordinates
(centred around the origin) so that \eqref{a2_coll} can be rewritten as
\begin{align}\label{ez}
\widehat{x}^2 - \widehat{y}^3 &  =  0.
\end{align}
Equation \eqref{ez} has exactly one solution close to the origin, namely
\begin{align}
\widehat{x}&  =  v^3 \qquad \textnormal{and} \qquad \widehat{y}  =  v^2.\label{v_Hj4}
\end{align}
Using \eqref{Hj1}, \eqref{hat_x_Hj3} and \eqref{v_Hj4}, it follows that
\begin{align}
t &  =  \frac{2}{f_{t_{20}}} v^3 + O\bigl(v^4\bigr). \label{t_mult_A2L}
\end{align}
Since we are setting $t$ to be equal to zero to obtain the $\mathsf{A}_2^{\mathsf{L}}$ point,
\eqref{t_mult_A2L} implies that the multiplicity of the intersection in the second term on the
right-hand side of \eqref{A2FT0T0_coll_cycle_ver} is $3$. This proves \eqref{A2FT0T0_coll_cycle_ver}.
Note that the assumption $f_{t_{20}} \not= 0$ and the
genericity assumption of the third derivative is valid, since to compute the multiplicity of
intersections, we will intersect with a~generic cycle. The general case now follows as before.
\end{proof}

\begin{thm}
\label{Tk1Tk2_etc_alt_way_A2F_ah}
Let $n$ be a nonnegative integer.
Then, the following equality of homology classes hold in
$H_*\bigl(\mathsf{M}_{n+2}^1, \mathbb{R}\bigr)$:
\begin{align}
\bigl[\mathsf{A}_2^{\mathsf{L}} \underbrace{\mathsf{T}_{1} \dots \mathsf{T}_{1}}_{n} \mathsf{T}_{0} \mathsf{T}_0\bigr]
\cdot \bigl[\Delta_{{n+1}, {n+2}}^{\mathsf{L}}\bigr] &  =
\pi^{\ast}\bigl[\mathsf{A}_2^{\mathsf{L}} \underbrace{\mathsf{T}_{1} \dots\mathsf{T}_{1}}_{n+1}\bigr] \cdot
[\Delta_{n+1, n+2}], \label{Tmu_T0_nu_alt_way_A2F_ah}
\end{align}
provided $d > 2n+4$.
\end{thm}

\begin{proof} This is a generalization of Theorems \ref{Tk1Tk2_etc_alt_way} and
\ref{Node_whose_one_brunch_tangent at multi}; the proof is the same.
\end{proof}

\begin{rem} Let us explain how to extract numbers. On $\mathsf{M}^1_{n+2}$
consider the following classes, given by
\[\alpha := y_1^{r} y_d^{s}b_1^{\nu_1} a_1^{\varepsilon_1} \dots a_{n+2}^{\varepsilon_{n+2}}
\qquad \text{and} \qquad \beta := y_1^{r} y_d^{s}b_1^{\nu_1}
a_1^{\varepsilon_1} \dots a_{n+1}^{\varepsilon_{n+1} + \varepsilon_{n+2}}. \]
Intersecting both sides of \eqref{Tmu_T0_nu_alt_way_A2F_ah} with $\alpha$
gives us
\[
\bigl[\mathsf{A}_2^{\mathsf{L}} \underbrace{\mathsf{T}_{1} \dots \mathsf{T}_{1}}_{n} \mathsf{T}_{0} \mathsf{T}_0\bigr]
\cdot \bigl[\Delta_{{n+1}, {n+2}}^{\mathsf{L}}\bigr] \cdot \alpha  =
\bigl[\mathsf{A}_2^{\mathsf{L}} \underbrace{\mathsf{T}_{1} \dots\mathsf{T}_{1}}_{n+1}\bigr] \cdot \beta. \label{Tmu_T0_nu_alt_way_A2F_ah_Num}
\]
\end{rem}
 The next theorem is as follows.

\begin{thm}
\label{theorem_for_many_tks_ag_pakL_A2}
Let $n$ be a nonnegative integer.
Then the following equality of classes in $H_*(\mathsf{M}_{{n+1}}^1, \mathbb{R})$ holds:
\begin{gather}
\pi^{\ast}\bigl[\mathsf{A}_2^{\mathsf{L}} \underbrace{\mathsf{T}_{1}\mathsf{T}_{1} \dots \mathsf{T}_{1}}_{n}\bigr]\cdot
\pi_{n+1}^{\ast}[\mathsf{T}_0] \nonumber\\[-1mm]
 \qquad =
\bigl[\mathsf{A}_2^{\mathsf{L}} \underbrace{\mathsf{T}_{1}\mathsf{T}_{1} \dots \mathsf{T}_{1}}_{n} \mathsf{T}_{0}\bigr]
 + \sum_{i=1}^{n} 2 \pi^{\ast}\bigl[\mathsf{A}_2^{\mathsf{L}} \underbrace{\mathsf{T}_{1}\mathsf{T}_{1} \dots \mathsf{T}_{1}}_{n}\bigr] \cdot [\Delta_{i, {n+1}}] \nonumber \\
\phantom{ \qquad =}{} + 2 \pi^{\ast} \bigl[\mathsf{A}_2^{\mathsf{L}} \underbrace{\mathsf{T}_{1}\mathsf{T}_{1} \dots \mathsf{T}_{1}}_{n}\bigr] \cdot
\bigl[\Delta_{{n+1}}^1\bigr],\label{many_tks_ag_pakL_A2}
\end{gather}
provided $d > 2n+3$.
\end{thm}

\begin{proof}
This is a generalization of Theorem \ref{theorem_for_many_tks_ag_pakL}.
The new thing we need to show is to
justify the third term on the right-hand side of \eqref{many_tks_ag_pakL_A2}.

We first prove the special case of Theorem \ref{theorem_for_many_tks_ag_pakL_A2} where $n = 0$. In
this case, \eqref{many_tks_ag_pakL_A2} simplifies to
\begin{align}
\pi^*\bigl[\mathsf{A}_2^{\mathsf{L}}\bigr] \cdot \pi_1^*[\mathsf{T}_0] & =  \bigl[\mathsf{A}_2^{\mathsf{L}} \mathsf{T}_0\bigr] +
2 \pi^*\bigl[\mathsf{A}_2^{\mathsf{L}}\bigr]\cdot \bigl[\Delta_{1}^1\bigr]. \label{pa1_t0_sp_a2}
\end{align}
We prove equation \eqref{pa1_t0_sp_a2} in a similar way to how we proved
equation \eqref{pa1_t0_sp}. The justification for the first term on the
right-hand side of \eqref{pa1_t0_sp_a2} is same as the justification for the first term on the
right-hand side of equation \eqref{pa1_t0_sp}. The new thing we need to do is justify the
second term.\looseness=-1

On the set-theoretic level, the second term is clear. What we need to do now is justify the multiplicity
of $2$.
For convenience, set the $\mathsf{A}_2^{\mathsf{L}}$ point to
be the origin. We are now going to
study the multiplicity with which the evaluation map vanishes at the origin.
Note that the curve $f$ is such that $f_{00}$, $f_{10}$, $f_{01}$, $f_{20}f_{02} - f_{11}^2$ all vanish.
Further assume that $f_{20} \neq 0$. Consequently, the curve is given by
\begin{align*}
f(x, y)& = \frac{f_{20}}{2} x^2 + f_{11} xy + \frac{f_{11}^2}{f_{02}} y^2 + \mathcal{R}(x, y),
\end{align*}
where $\mathcal{R}(x, y)$ is a remainder term of order three. The order of vanishing of the evaluation map should be $2$.
To see why that is so, consider the evaluation map
\begin{align*}
\varphi(f, x)& := f(x, 0) = \frac{f_{20}}{2} x^2 + O\bigl(x^3\bigr).
\end{align*}
Since by assumption $f_{20} \neq 0$, we conclude that the order of vanishing is $2$. The assumption that
$f_{20} \neq 0$ is valid, since to compute the multiplicity of intersections, we intersect with generic cycles.
This proves \eqref{pa1_t0_sp_a2} on the level of homology. The general statement now follows similarly.
\end{proof}

\begin{thm}\label{theorem_for_many_tks_ag_pakL_A2_T0}
Let $n$ be a nonnegative integer.
Then the following equality of classes in $H_*\bigl(\mathsf{M}_{{n+2}}^1, \mathbb{R}\bigr)$ holds:
\begin{gather*}
\pi^{\ast}\bigl[\mathsf{A}_2^{\mathsf{L}} \underbrace{\mathsf{T}_{1}\mathsf{T}_{1} \dots \mathsf{T}_{1}}_{n} \mathsf{T}_0\bigr]\cdot
\pi_{n+2}^{\ast}[\mathsf{T}_0] \\
\qquad =
\bigl[\mathsf{A}_2^{\mathsf{L}} \underbrace{\mathsf{T}_{1}\mathsf{T}_{1} \dots \mathsf{T}_{1}}_{n} \mathsf{T}_{0} \mathsf{T}_{0}\bigr]
+ 2 \pi^{\ast} \bigl[\mathsf{A}_2^{\mathsf{L}} \underbrace{\mathsf{T}_{1}\mathsf{T}_{1} \dots \mathsf{T}_{1}}_{n} \mathsf{T}_0\bigr]
\cdot \bigl[\Delta_{n+2}^1\bigr]  \\[-1mm]
\phantom{\qquad =}{} + \pi^{\ast} \bigl[\mathsf{A}_2^{\mathsf{L}} \underbrace{\mathsf{T}_{1}\mathsf{T}_{1} \dots \mathsf{T}_{1}}_{n} \mathsf{T}_0\bigr]
\cdot \bigl[\Delta_{n+1, {n+2}}\bigr]
  + \sum_{i=1}^{n} 2 \pi^{\ast}\bigl[\mathsf{A}_2^{\mathsf{L}} \underbrace{\mathsf{T}_{1}\mathsf{T}_{1} \dots
\mathsf{T}_{1}}_{n} \mathsf{T}_0\bigr] \cdot [\Delta_{i, {n+2}}],
\end{gather*}
provided $d > 2n+4$.
\end{thm}

\begin{proof}
The proof is the same as that of Theorem \ref{theorem_for_many_tks_ag_pakL_A2}.
\end{proof}

Finally, we note that
the expression for the homology class $\bigl[\mathsf{A}_2^{\mathsf{F}}\bigr]$
\big(which is an element of~${H_{*}\bigl(\mathsf{M}^1_{0}, \mathbb{R}\bigr)}$\big) can be computed using the
results of \cite{R.M}, given by
\[
\bigl[\mathsf{A}_2^{\mathsf{F}}\bigr]  = \bigl(12d^2-36d+24\bigr) y_d^2 b_1^2 + \bigl(8d-12\bigr)y_d^3 b_1 + 2 y_d^4.
\]
Using equation \eqref{a2_line_cycle}, we can compute
$\bigl[\mathsf{A}_2^{\mathsf{L}}\bigr]$ as well.
Hence, using the same reasoning given at the end of Section \ref{one_nodal_tang_comp}
and using the theorems proved in this section, we conclude that all intersection numbers
involving \smash{$\bigl[\mathsf{A}_2^{\mathsf{F}} \underbrace{\mathsf{T}_{1} \dots \mathsf{T}_{1}}_{n}\bigr]$}
can be computed.

\section{Counting singular curves}\label{singularity_in_our_format}
In this section, we derive our Main Result \ref{m_rslt2}.
We begin by asking the following question:
{\it How many $1$-nodal degree $d$ curves are there in $\mathbb{P}^2$
that pass through $\delta_d-1$ $($see \eqref{eb}$)$ generic points?}\looseness=1

We solve the above question by treating it as a special case of the
following more general question: how many pairs are there, a line and a degree
$d$ curve, such that the line passes through $m$ generic points and the
curve passes through $\delta_d-m$
points and the curve has a nodal point lying on the line?
Note that the answer to this question is automatically zero if $m \geq 3$, since a
line does not pass through three generic points. Fixing the line corresponds to setting~$m$ to be equal to $2$ (since a unique line passes through two points).
The answer to our original question is obtained when $m=1$.

In order to solve the above question, we first consider
the space of curves with a marked point~$p$ at which the curve is tangent to the line.
On this space, we
impose the condition that the derivative of the curve in the normal direction vanishes.
That precisely means that the curve has a node at $p$.
\begin{figure}[h!]
\centering
\includegraphics[scale = .7]{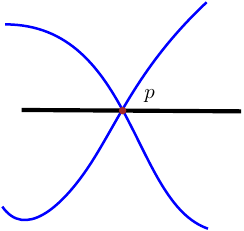}
\end{figure}

Taking the derivative in the normal direction induces a section of an appropriate
bundle.
This bundle is described in detail in Section~\ref{node_new}
(see equation \eqref{bundle_normal}).
Hence, computation of the Euler class of this relevant bundle yields the desired number.

We now explain how to enumerate curves with one
tacnode. Before that there is small digression. We first try to solve the question
of enumerating curves with two nodes. However, both the nodes are required to lie on a given line. Consider
the space of curves with three marked points $x_1$, $x_2$ and $x_3$, such that all the three points lie on the
curve and the line while the curve has a node at the point $x_1$. Pictorially, such a curve looks
as follows:
\begin{figure}[h!]\centering
\includegraphics[scale = .75]{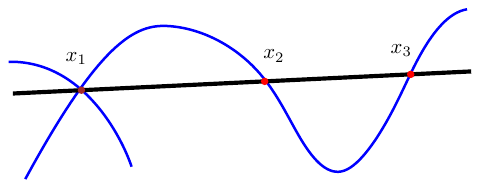}
\end{figure}

{\samepage Now impose the condition that $x_2$ and $x_3$ become equal.
This results in the following objects:
\begin{center}
\includegraphics[scale = .75]{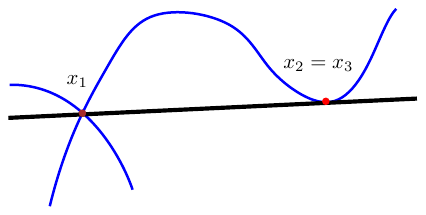}
\end{center}

}

Next impose the condition that the derivative of the curve at $x_2$ in the normal direction is zero. This corresponds to the
curve having a node at $x_2$. However, there is a degenerate contribution to the Euler class, which occurs when $x_1$ and $x_2$
come together, as is seen by the following picture:
\begin{figure}[h!]\centering
\includegraphics[scale = .75]{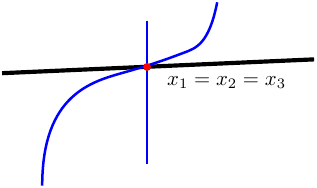}
\end{figure}

In the above picture, one of the branches of the node is tangent to the given line of second order.
Subtracting off this degeneracy allows us to solve the problem.
In particular, we can count the number of pairs consisting of a line and a degree $d$ curve, such that the line
passes through $m$ points, the curve passes through $n$ points with $m+n = \delta_d-2$
(see \eqref{eb}) while the curve has two nodes lying on the line.
When $m = 2$ and $n = \delta_d-4$, it corresponds to case where the curve has two nodes on the same
fixed line. On the other hand,
setting $m = 0$ and $n = \delta_d-2$ corresponds to the case where the two nodes are free.

For tacnodes, we use the fact that they occur
precisely when two nodes collide. Consider the space of curves with two nodes $x_1$ and $x_2$,
both of them lying on a given line. Now impose the condition that $x_1 = x_2$, and
the following object is obtained:
\begin{figure}[h!]\centering
\includegraphics[scale = .75]{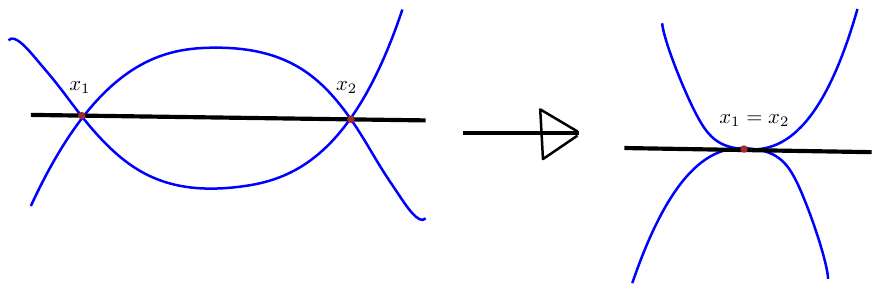}
\caption{Two nodes on the line limiting to a tacnode.} \label{tacnode_pic}
\end{figure}

Next assume that the line passes through $m$ points and the curve passes through $n$ points, where $m+n = \delta_d-3$.
The case where~${m = 0}$ and $n = 3$ gives the number of degree $d$ curves through $\delta_d-3$ generic points
that have a~tacnode.

To summarize, we have been able to solve the following questions about enumerating singular curves:
\begin{itemize}\itemsep=0pt
\item[$\bullet$] Plane curve of degree $d$ having
a node.

\item[$\bullet$] Plane curve of degree $d$ having
two distinct nodes.

\item[$\bullet$] Plane curve of degree $d$ having a tacnode.
\end{itemize}

Let us now implement this idea precisely.
Similar to $\mathsf{A}_1^{\mathsf{F}}$ and $\mathsf{A}_2^{\mathsf{F}}$, define
$\mathsf{A}_3^{\mathsf{F}} \subset \mathsf{M}^1_0$ to be the locus $(H_1, H_d, p)$
such that $H_d$ has a tacnode at $p$.
Furthermore, define $\mathsf{A}_1^{\mathsf{F}} \mathsf{A}_1^{\mathsf{F}} \subset \mathsf{M}^2_0$, to be
the locus $(H_1, H_d, p_1, p_2)$
such that $H_d$ has a node at $p_1$ and $p_2$ and $p_1$ and $p_2$ are distinct.
The following formulas will be proven in this section:
\begin{gather}
\bigl[\mathsf{A}_1^{\mathsf{F}}\bigr]  = \bigl(3d^2-6d+3\bigr) y_d b_1^2 + (3d-3)y_d^2 b_1 + y_d^3, \label{a1_free_cycle} \\
\bigl[\mathsf{A}_3^{\mathsf{F}}\bigr]  = \bigl(50d^2-192d+168\bigr) y_d^3 b_1^2 + (25d-48)y_d^4 b_1 + 5 y_d^5 \qquad \textnormal{and}
\label{A3F_cycle}\\
\bigl[\mathsf{A}_1^{\mathsf{F}} \mathsf{A}_1^{\mathsf{F}}\bigr]  = \bigl(9d^4-36d^3+12d^2+81d-66\bigr)y_d^2 b_1^2 b_2^2 \nonumber \\
\phantom{\bigl[\mathsf{A}_1^{\mathsf{F}} \mathsf{A}_1^{\mathsf{F}}\bigr]  = }{} + \bigl(9d^3-27d^2-d-30\bigr)y_d^3 b_1 b_2^2
+ \bigl(9d^3-27d^2-d-30\bigr) y_d^3 b_1^2 b_2 \nonumber \\
\phantom{\bigl[\mathsf{A}_1^{\mathsf{F}} \mathsf{A}_1^{\mathsf{F}}\bigr]  = }{} + \bigl(3d^2-6d-4\bigr)y_d^4 b_1^2 + \bigl(9d^2-18d+2\bigr)y_d^4 b_1 b_2 +
\bigl(3d^2-6d-4\bigr)y_d^4 b_2^2 \nonumber \\
\phantom{\bigl[\mathsf{A}_1^{\mathsf{F}} \mathsf{A}_1^{\mathsf{F}}\bigr]  = }{} + (3d-3)y_d^5 b_1 + (3d-3)y_d^5 b_2 + y_d^6. \label{a1fa1f_ag87}
\end{gather}

\subsection[Counting 1-nodal curves]{Counting $\boldsymbol{ 1}$-nodal curves}\label{node_new}

We prove \eqref{a1_free_cycle}. Recall that \eqref{a1_line_cycle}
says that the class $\bigl[\mathsf{A}_1^{\mathsf{L}}\bigr]$ can be computed by multiplying
the class $\bigl[\mathsf{A}_1^{\mathsf{F}}\bigr]$ with $(y_d + b_1)$. Note that
knowing $\bigl[\mathsf{A}_1^{\mathsf{L}}\bigr]$ does not -- a priori -- give
$\bigl[\mathsf{A}_1^{\mathsf{F}}\bigr]$. Nevertheless, it will be shown that knowing
$\bigl[\mathsf{A}_1^{\mathsf{L}}\bigr]$ in fact does give the class $\bigl[\mathsf{A}_1^{\mathsf{F}}\bigr]$.

First of all, we note that $\bigl[\mathsf{A}_1^{\mathsf{F}}\bigr]$
is a codimension $3$ class in $\mathsf{M}^1_0$. Hence, it is of the form
\[
\bigl[\mathsf{A}_1^{\mathsf{F}}\bigr]  = \sum_{\substack{i+j= 3, \\ i, j \geq 0, \\
j \leq 2}}C_{ij} y_d^{i} b_1^j.\vspace{-1mm}
\]
There is no term involving $y_1$, because it is the pullback of a class in $\mathcal{D}_d \times X^1$.
We also note that $j$ can not be greater than $2$, since $b_1^3$ is zero.
Hence, it suffices to compute the following three numbers
$C_{12}$, $ C_{21} $ and $C_{30}$.
Next, we note that
\[
C_{12} = \bigl[\mathsf{A}_1^{\mathsf{F}}\bigr] \cdot y_1^2 y_d^{\delta_d-1}, \qquad C_{21} =
\bigl[\mathsf{A}_1^{\mathsf{F}}\bigr] \cdot y_1^2 y_d^{\delta_d-2} b_1 \qquad \textnormal{and}
\qquad C_{30} = \bigl[\mathsf{A}_1^{\mathsf{F}}\bigr] \cdot y_1^2 y_d^{\delta_d-3} b_1^2.
\]
It is rather straightforward that $C_{21}$ is the number of
degree $d$ curves passing through $\delta_d-2$ generic points with a node lying on a line; intersecting with
$y_d^{\delta_d}-2$ makes the curve pass through~${\delta_d-2}$ points and intersecting with $b_1$
makes the nodal point lie on a line (the line represents the class $b_1$).
But that can also be obtained as an intersection number involving the class $\bigl[\mathsf{A}_1^{\mathsf{L}}\bigr]$, namely
\begin{align}
C_{21}&  =  \bigl[\mathsf{A}_1^{\mathsf{L}}\bigr] \cdot y_1^2 y_d^{\delta_d-2}. \label{j7}
\end{align}
Hence, knowing $\bigl[\mathsf{A}_1^{\mathsf{L}}\bigr]$ gives the coefficient $C_{21}$.

Next, observe that $C_{12}$ is the number of degree $d$ curves passing through $\delta_d-1$ generic points with a node.
Hence, one concludes that
\begin{equation}\label{j1}
C_{12} =  \bigl[\mathsf{A}_1^{\mathsf{L}}\bigr] \cdot y_1 y_d^{\delta_d-1}.
\end{equation}
Indeed, intersecting with $y_d^{\delta_d-1}$ makes the nodal curve pass through $\delta_d-1$
points. For each such curve, the nodal point is now fixed. Intersecting with $y_1$ now fixes a unique line.
Hence, the right-hand side of \eqref{j1} gives us $C_{2}$.

Finally, note that $C_{30}$ is the number of degree $d$ curves passing through $\delta_d-3$ generic points with a node
located at a fixed point. A similar argument gives that{\samepage
\begin{align}
C_2&  =  \bigl[\mathsf{A}_1^{\mathsf{L}}\bigr] \cdot y_1^2 b_1 y_d^{\delta_d-3}. \label{j8}
\end{align}
Hence, in order to compute
$\bigl[\mathsf{A}_1^{\mathsf{F}}\bigr]$ it suffices to compute $\bigl[\mathsf{A}_1^{\mathsf{L}}\bigr]$.}

We do the intersection theory on
$\mathsf{M}^1_0$. Think of $\overline{\mathsf{T}}_1$ as a subspace of
$\mathsf{M}^1_0 := \mathcal{D}_1 \times \mathcal{D}_d \times X^1$.
Let~${(H_1, H_d, p) \in \overline{\mathsf{T}}_1}$ and
suppose that the curve $H_d$ is given by the zero set of the polynomial~$f_d$.
Furthermore, suppose that the line is given by the zero set of $f_1$.
This gives us the following short exact sequence
\begin{align}
0 \longrightarrow & \textnormal{Ker} \nabla f_1|_{p} \longrightarrow TX^1|_{p} \longrightarrow
 \gamma_{\mathcal{D}_1}^* \otimes \gamma_{X^1}^* \longrightarrow 0. \label{ses_line_bundle_ag}
\end{align}
Here $\gamma_{\mathcal{D}_1}$ and $\gamma_{X^1}$ denote the tautological line bundles over $\mathcal{D}_1$
and $X^1$ respectively and $\gamma_{\mathcal{D}_1}^*$ and~$\gamma_{X^1}^*$ denote their duals.
The first nontrivial map in equation \eqref{ses_line_bundle_ag} denotes the inclusion map into the tangent space
$T X^1|_{p}$. The second map denotes the vertical derivative
$\nabla f_1|_{p}$. This map is surjective since the point $p$ is not a singular point of the line
$f_1^{-1}(0)$ (all points of a line are smooth). Define the line bundle
$\mathbb{L} \longrightarrow \overline{\mathsf{T}}_1$, whose fiber over each point is
$\textnormal{Ker} \nabla f_1|_{p}$.

We now impose the condition that the derivative of $f_d$ in the normal direction
to the line is zero.
This means that
$
\nabla f_d|_{p}(u)  = 0 $, $\forall  u \in \bigl(T X^1/\mathbb{L}\bigr)$.
Taking the derivative along $u$ induces a~section of the vector bundle
\begin{align}
\mathcal{V} &:= \bigl(TX^1/\mathbb{L}\bigr)^* \otimes \gamma_{\mathcal{D}_d}^* \otimes (\gamma_{X^1}^{*})^{d}. \label{bundle_normal}
\end{align}
Hence,
\begin{align}
\bigl[\mathsf{A}_1^{\mathsf{L}}\bigr] &  =  [\mathsf{T}_1]\cdot (y_d-y_1 + (d-1)b_1).\label{en_a1}
\end{align}
The second term on the right-hand side of \eqref{en_a1} is the Euler class of $\mathcal{V}$, which can
be computed using \eqref{ses_line_bundle_ag}.
Using the results of Section \ref{smooth_tangencies}, all intersection numbers involving the class
$[\mathsf{T}_1]$ can be computed. Hence, \eqref{en_a1}
enables us to compute all intersection numbers involving the class~$\bigl[\mathsf{A}_1^{\mathsf{L}}\bigr]$.
As a result, using \eqref{j7}, \eqref{j1} and \eqref{j8}
one obtains \eqref{a1_free_cycle}.

It remains to prove that the intersection in \eqref{en_a1} is transverse.
But this is precisely the content of the proof of Proposition \ref{prp_smth_mfld_with_PA1r_ag},
with $n = 0$ and $r = 0$.

\subsection[Counting 2-nodal curves]{Counting $\boldsymbol{ 2}$-nodal curves}
\label{binodal_comp}

We now prove \eqref{a1fa1f_ag87}. First of all, note that $\bigl[\mathsf{A}_1^{\mathsf{F}} \mathsf{A}_1^{\mathsf{F}}\bigr]$
is a codimension $6$ class in $\mathsf{M}^2_0$. Hence, it is of the form
\begin{align}
\bigl[\mathsf{A}_1^{\mathsf{F}} \mathsf{A}_1^{\mathsf{F}}\bigr] & = \sum_{\substack{i+j+k= 6, \\ i, j, k \geq 0, \\
j , k \leq 2}}C_{ijk} y_d^{i} b_1^j b_2^k.\label{gijk_defn}
\end{align}
There is no term involving $y_1$, because it is the pullback of a class in $\mathcal{D}_d \times X^1 \times X^2$.
We also note that $j$ or $k$ can not be greater than $2$, since $b_1^3$ and $b_2^3$ are both zero.
We also note that~${C_{ijk} = C_{ikj}}$; this is because the map from
$\mathsf{A}_1^{\mathsf{F}} \mathsf{A}_1^{\mathsf{F}}$ to itself, that permutes the two marked points is a bijection.
Hence, it suffices to compute the following five numbers
$C_{222}$, $ C_{312}$, $ C_{420}$, $ C_{411}$, $ C_{510}$ and $ C_{600}$.
We perform intersection theory on $\mathsf{M}^2_0$. Define
$\mathsf{A}_1^{\mathsf{L}} \mathsf{A}_1^{\mathsf{L}} \subset \mathsf{M}^2_0$ to be the locus~${(H_1, H_d, p_1, p_2)}$
such that curve $H_d$ has a node at the two distinct points $p_1$ and $p_2$. Furthermore, the two points $p_1$
and $p_2$ also lie on the line $H_1$. Let us for the moment assume that we can compute all intersection
numbers involving the class $\bigl[\mathsf{A}_1^{\mathsf{L}} \mathsf{A}_1^{\mathsf{L}}\bigr]$. It will now be shown how to
compute the numbers $C_{ijk}$ from this information. In particular, it will be shown that
\begin{gather}
C_{222} = \bigl[\mathsf{A}_1^{\mathsf{L}} \mathsf{A}_1^{\mathsf{L}}\bigr]\cdot y_d^{\delta_d-2},
\qquad C_{312}= \bigl[\mathsf{A}_1^{\mathsf{L}} \mathsf{A}_1^{\mathsf{L}}\bigr]\cdot y_d^{\delta_d-3}b_1, \qquad
C_{420} = \bigl[\mathsf{A}_1^{\mathsf{L}} \mathsf{A}_1^{\mathsf{L}}\bigr]\cdot y_d^{\delta_d-4}b_1^2, \nonumber \\
C_{411} = \bigl[\mathsf{A}_1^{\mathsf{L}} \mathsf{A}_1^{\mathsf{L}}\bigr]\cdot y_d^{\delta_d-4}b_1 b_2, \qquad
C_{510} = \bigl[\mathsf{A}_1^{\mathsf{L}} \mathsf{A}_1^{\mathsf{L}}\bigr]\cdot y_d^{\delta_d-5}b_1b_2^2
\qquad \text{and}\nonumber\\
C_{600} = \bigl[\mathsf{A}_1^{\mathsf{L}} \mathsf{A}_1^{\mathsf{L}}\bigr]\cdot y_d^{\delta_d-6}b_1^2b_2^2. \label{Cijk}
\end{gather}
We start by justifying the expression for $C_{222}$.
Note that by definition
\smash{$
C_{222}   =  \bigl[\mathsf{A}_1^{\mathsf{F}} \mathsf{A}_1^{\mathsf{F}}\bigr]\cdot y_1^2 y_d^{\delta_d-2}$}.
Hence, $C_{222}$ is equal to the number of degree $d$ curves passing through $\delta_d-2$ points and having two
(ordered) nodal points. This is the same as the right-hand side of the first equation of~\eqref{Cijk}.
Let us see why this is true. Intersecting
$\bigl[\mathsf{A}_1^{\mathsf{L}} \mathsf{A}_1^{\mathsf{L}}\bigr]$
with $y_d^{\delta_d-2}$ corresponds to making the curve pass through
$\delta_d-2$ points. By definition of $\bigl[\mathsf{A}_1^{\mathsf{L}} \mathsf{A}_1^{\mathsf{L}}\bigr]$,
both the nodal points lie on the line. Since there are two nodal points, the
line is now fixed. This proves the first assertion of \eqref{Cijk}.
The remaining five assertions of \eqref{Cijk} can be seen similarly.

Hence, we have shown that to prove \eqref{a1fa1f_ag87}, it suffices to compute all intersection numbers
involving the class $\bigl[\mathsf{A}_1^{\mathsf{L}} \mathsf{A}_1^{\mathsf{L}}\bigr]$. To explain how to compute those
numbers, think of $\mathsf{A}_1^{\mathsf{L}} \mathsf{T}_1$ as a~subspace of $\mathsf{M}^2_0$.
Let \smash{$(H_1, H_d, p_1, p_2) \in \overline{\mathsf{A}_1^{\mathsf{L}} \mathsf{T}}_1$}.
We now impose the condition that the derivative of the polynomial defining the curve $H_d$
in the normal direction to the line $H_1$ vanishes at~$p_2$. Analogous to \eqref{en_a1},
it is tempting to conclude that
\smash{$
\bigl[\mathsf{A}_1^{\mathsf{L}} \mathsf{T}_1\bigr] \cdot (y_d-y_1+ (d-1)b_2 ) =
\bigl[\mathsf{A}_1^{\mathsf{L}} \mathsf{A}_1^{\mathsf{L}}\bigr]$}.
Unfortunately, the above equation is incorrect. This is because when $p_1$ and $p_2$ collide in the closure
we get a~$\mathsf{P}^{(1)}\mathsf{A}_1$, as can be seen intuitively by the following picture:
\begin{center}\includegraphics[scale = .75]{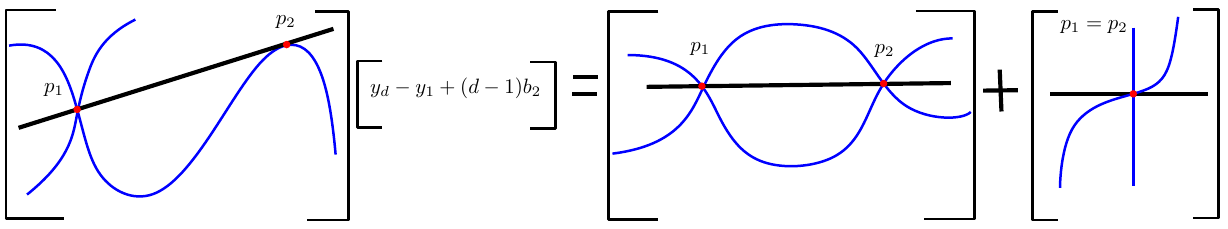}\end{center}

In other words, we are claiming that on $H_*\bigl(\mathsf{M}^2_0; \mathbb{R}\bigr)$, the
following equality of homology classes hold
\begin{align}
\bigl[\mathsf{A}_1^{\mathsf{L}} \mathsf{T}_1\bigr] \cdot (y_d-y_1+ (d-1)b_2 ) & =
\bigl[\mathsf{A}_1^{\mathsf{L}} \mathsf{A}_1^{\mathsf{L}}\bigr] +
\bigl[\mathsf{P}^{(1)}\mathsf{A}_1\bigr]\cdot \bigl[\Delta^{11}\bigr]. \label{A1A1_new_Homology}
\end{align}
Here $\Delta^{11}$ denotes the locus of points $(H_1, H_d, p_1, p_2) \in \mathsf{M}^2_0$
such that $p_1 = p_2$. We first explain how to extract numbers from this equation.
Let $\alpha$ and $\beta$ be classes in $H_*\bigl(\mathsf{M}^2_0; \mathbb{R}\bigr)$
given by
\[\alpha := y_1^r y_d^s b_1^{\varepsilon_1} b_2^{\varepsilon_2} \qquad \textnormal{and} \qquad
\beta := y_1^r y_d^s b_1^{\varepsilon_1+ \varepsilon_2}.\]
Using equation \eqref{A1A1_new_Homology}, we conclude that
\begin{gather}
\bigl[\mathsf{A}_1^{\mathsf{L}} \mathsf{T}_1\bigr] \cdot (y_d-y_1+ (d-1)b_2 ) \cdot \alpha  =
\bigl[\mathsf{A}_1^{\mathsf{L}} \mathsf{A}_1^{\mathsf{L}}\bigr] \cdot \alpha + \bigl[\mathsf{P}^{(1)}\mathsf{A}_1\bigr]\cdot \beta \label{A1A1_new} \\
\qquad\implies \bigl[\mathsf{A}_1^{\mathsf{L}} \mathsf{A}_1^{\mathsf{L}}\bigr] \cdot \alpha  =
\bigl[\mathsf{A}_1^{\mathsf{L}} \mathsf{T}_1\bigr] \cdot (y_d-y_1+ (d-1)b_2 ) \cdot \alpha
-\bigl[\mathsf{P}^{(1)}\mathsf{A}_1\bigr]\cdot \beta. \label{A1A1_new2}
\end{gather}
By the results of Section~\ref{nodal_tang}, we can compute all intersection
numbers involving the classes~${\bigl[\mathsf{A}_1^{\mathsf{L}} \mathsf{T}_1\bigr]}$ and
\smash{$\bigl[\mathsf{P}^{(1)}\mathsf{A}_1\bigr]$}. Hence, using \eqref{A1A1_new2} we can compute all intersection numbers
involving the class~${\bigl[\mathsf{A}_1^{\mathsf{L}} \mathsf{A}_1^{\mathsf{L}}\bigr]}$.
Hence, using \eqref{Cijk} and \eqref{gijk_defn}, we get \eqref{a1fa1f_ag87}.

We now justify \eqref{A1A1_new_Homology}.
First, it will be proved on the set-theoretic level.
We switch to affine space. As before, the line is the $x$-axis.
We have already shown that $\bigl(\mathsf{A}_1^{\mathsf{L}} \overline{\mathsf{T}}_1\bigr)_{\mathsf{Aff}}$ is a smooth
submanifold of $\mathcal{F}_d^{+} \times \mathbb{C}^2$. The argument that the intersection of the cycles on the
open part is transverse is similar to how we have shown the earlier transversality statements. This justifies the
first term on the right-hand side of \eqref{A1A1_new} (on the level of homology).

The second term will now be justified. One needs to figure out what happens when the nodal point $p_1$
and the $\mathsf{T}_1$ point~$p_2$ coincide. The following fact has already been shown: if at $p_1$ and~$p_2$
the first derivatives of the polynomial defining the curve vanish (along the direction of the line),
then when $p_1$ and $p_2$ coincide, the first and second derivatives of the polynomial vanish.
This was shown while proving that $\mathsf{T}_1$ and $\mathsf{T}_1$ collide to form a $\mathsf{T}_3$.
Notice that one does not need $p_1$ and $p_2$ to be smooth points of the curve for the argument to work.

Now note that $p_1$ is a singular point of the curve. Hence, when the two points collide, it will continue to
remain a singular point of the curve. Hence, we conclude the following: when $p_1$ and $p_2$ coincide,
the first two derivatives of the polynomial defining the curve
vanish (along the direction of the line). Furthermore, it is a
singular point of the curve. Note that any element of
$\mathsf{P}^{(1)}\mathsf{A}_1$ satisfies these conditions. We now show that every curve in
$\mathsf{P}^{(1)}\mathsf{A}_1$ actually lies in the closure.

Let $(f, (0, 0))$ belong to $\bigl(\mathsf{P}^{(1)}\mathsf{A}_1\bigr)_{\textsf{Aff}}$.
Note that we are setting the $\mathsf{P}^{(1)}\mathsf{A}_1$ point to be $(0,0)$.
Hence, the Taylor expansion of $f$ is given by
\begin{align*}
f(x,y)& = f_{11} xy + \frac{f_{02}}{2} y^2 + \frac{f_{21}}{2} x^2y + \frac{f_{12}}{2} xy^2
+ \frac{f_{03}}{6} y^3 + \frac{f_{40}}{24} x^4 + \cdots.
\end{align*}
We now try to find a nearby curve $f_t$ that has a nodal point at $(0,0)$ and has a $\mathsf{T}_1$ point at~${(t,0)}$.
The Taylor expansion of $f_t$ is given by
\begin{align*}
f_t(x,y)& = \frac{f_{t_{20}}}{2} x^2 + f_{t_{11}} xy + \frac{f_{t_{02}}}{2} y^2 +
\frac{f_{t_{30}}}{6} x^3 + \frac{f_{t_{21}}}{2} x^2y + \frac{f_{t_{12}}}{2} xy^2
+ \frac{f_{t_{03}}}{6} y^3 + \frac{f_{t_{40}}}{24} x^4+ \cdots.
\end{align*}
We now impose the condition $f_t(t,0)=0$ and $(f_{t})_x(t,0) =0$. We can solve for this
and get an expression for $f_{t_{20}}$ and $f_{t_{30}}$. Moreover, every solution is constructed by this
procedure. Hence, there is only one branch.

It remains to justify the multiplicity of the intersection. Consider the condition of taking the derivative in the
normal direction at the point $(t, 0)$. This is given by $(f_{t})_y(t, 0)$. Written explicitly, it is given by the map
$t  \longrightarrow  f_{t_{11}} t$.
Assuming that $f_{t_{11}} \neq 0$, the order of vanishing of the above function at $t = 0$ is one.
This is a valid assumption, since to compute the intersection multiplicity, we will be intersecting with
generic cycles. This proves \eqref{A1A1_new} on the level of homology and hence, completes the proof of~\eqref{a1fa1f_ag87}.

\subsection[Counting 1-tacnodal curve]{Counting $\boldsymbol{ 1}$-tacnodal curves}
\label{tacnode_counts_bi_node}

We now prove \eqref{A3F_cycle}.
First of all, we note that $\bigl[\mathsf{A}_3^{\mathsf{F}}\bigr]$
is a codimension $5$ class in $\mathsf{M}^1_0$. Hence, it is of the form
\begin{align*}
\bigl[\mathsf{A}_3^{\mathsf{F}}\bigr] & = \sum_{\substack{i+j= 5, \\ i, j \geq 0, \\
j \leq 2}}C_{ij} y_d^{i} b_1^j.
\end{align*}
There is no term involving $y_1$, because it is the pullback of a class in $\mathcal{D}_d \times X^1$.
We also note that~$j$ can not be greater than $2$, since $b_1^3$ is zero.
Hence, it suffices to compute the following three numbers
$C_{32}$, $ C_{41} $ and $C_{50}$.
We make a small digression and discuss the condition of a singularity being a tacnode.
Recall Definition \ref{singularity_defn}:
the zero set of a holomorphic function~${f \colon U \longrightarrow \mathbb{C}}$
has a tacnode at the origin if after a local (analytic) change of coordinates, the function
can be written
as $f(x, y) := y^2 + x^4$. Here $U$ is an open subset of $\mathbb{C}^2$
(with the usual topology of $\mathbb{C}$).
A~tacnode satisfies the condition that the
kernel of the hessian is precisely one dimensional (i.e., the hessian is degenerate, but not
identically zero). We call the kernel of the hessian of~$f$ to be the distinguished
direction of the tacnode. For the tacnode $y^2+ x^4 = 0$, the tangent vector~$\frac{\partial}{\partial x}$ is the distinguished direction.

We now define
$\mathsf{P}\mathsf{A}_3 \subset \mathsf{M}^1_0$ to be the locus $(H_1, H_d, p_1)$
such that curve $H_d$ has a tacnode at $p_1$ and the distinguished
direction of the tacnode is given by the line
$H_1$.
Pictorially, it is denoted by the following picture:
\begin{center}
\includegraphics[scale = .85]{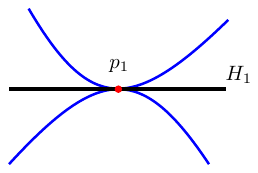}\vspace*{-0.2cm}\end{center}

Notice the difference between the spaces \smash{$\overline{\mathsf{A}^{\mathsf{L}}_3}$}
and $\mathsf{P}\mathsf{A}_3$; the latter lies in the closure of the former.
The space $\mathsf{A}_3^{\mathsf{L}}$ is pictorially represented as follows:
\begin{center}\includegraphics[scale = .85]{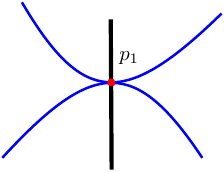}\vspace*{-0.2cm}\end{center}

It will now be shown that computing intersection numbers involving the class
$[\mathsf{P}\mathsf{A}_3]$ enables us to compute the numbers $C_{ij}$.
In particular, it will be shown that
\begin{align}
C_{32}& =[\mathsf{P}\mathsf{A}_3]\cdot y_d^{\delta_d-3}, \qquad
C_{41} = [\mathsf{P}\mathsf{A}_3]\cdot y_d^{\delta_d-4} b_1 \qquad \textnormal{and} \qquad
C_{50} = [\mathsf{P}\mathsf{A}_3]\cdot y_d^{\delta_d-5} b_1^2. \label{pa3_int}
\end{align}
Let us justify the first term of \eqref{pa3_int}, namely the computation of $C_{32}$.
Note that by definition,
\[
C_{32} = \bigl[\mathsf{A}_3^{\mathsf{F}}\bigr]\cdot y_1^2 y_d^{\delta_d-3}.
\]
Hence, $C_{32}$ denotes the number of degree $d$ curves passing through $\delta_d-3$ generic points
and having a tacnode. We now note that intersecting $[\mathsf{P}\mathsf{A}_3]$
with $y_d^{\delta_d-3}$ makes the curve pass through~${\delta_d-3}$ points.
By definition of $\mathsf{P}\mathsf{A}_3$, the line is now fixed because there is a unique
line that passes through the tacnodal point and is the branch of the tacnode. This proves the first equation of~\eqref{pa3_int}.
The remaining two equations follow similarly.

Hence, it has been shown that to prove \eqref{A3F_cycle}, it suffices to compute all intersection numbers
involving the class $[\mathsf{P}\mathsf{A}_3]$. It will now be explained how to compute those
numbers.

Although we are enumerating curves with one singularity, the intersection theory will be done on the
two pointed space $\mathsf{M}^2_0$.
Let \smash{$(H_1, H_d, p_1, p_2) \in \overline{\mathsf{A}_1^{\mathsf{L}} \mathsf{A}^{\mathsf{L}}_1}$}.
Now impose the condition that the two points $p_1$ and $p_2$ come together.
It will be shown shortly that when that happens, we get a~curve in $\mathsf{PA}_3$ (see Figure \ref{tacnode_pic}).
In $H_*\bigl(\mathsf{M}^2_0; \mathbb{R}\bigr)$,
the following equality of homology classes holds:
\begin{align}
\bigl[\mathsf{A}_1^{\mathsf{L}} \mathsf{A}_1^{\mathsf{L}}\bigr] \cdot (b_1+b_2-y_1)
& =  [\mathsf{PA}_3]\cdot \bigl[\Delta^{11}\bigr].
\label{PA3_new_Homology}
\end{align}
We first explain how to extract numbers.
Let $\alpha$ be a class in $H_*\bigl(\mathsf{M}^1_0; \mathbb{R}\bigr)$.
Intersecting both sides of equation \eqref{PA3_new_Homology} with $\alpha$, we get that
\begin{align}
[\mathsf{PA}_3]\cdot \alpha & =
\bigl[\mathsf{A}_1^{\mathsf{L}} \mathsf{A}_1^{\mathsf{L}}\bigr] \cdot (b_1+b_2-y_1) \cdot \alpha .
\label{PA3_new}
\end{align}
Using the results of Section \ref{binodal_comp}, we can compute all intersection numbers involving the class
$\bigl[\mathsf{A}_1^{\mathsf{L}} \mathsf{A}_1^{\mathsf{L}}\bigr]$. Hence, using \eqref{PA3_new},
we can compute all intersection numbers involving the class $[\mathsf{PA}_3]$.

We now justify \eqref{PA3_new_Homology}.
First of all recall the proof of the fact that when a
$\mathsf{T}_1$ point and another $\mathsf{T}_1$ point collide, we get a $\mathsf{T}_3$ point.
The proof in fact shows the following: suppose the first derivative of $f$ vanishes at $p_1$
and $p_2$, then when the two points coincide, the first, second and third derivatives coincide.
The proof does not in any way require the points to be smooth points of the curve. Hence, when two
nodal points lying on a line coincide, the first, second and third derivatives along the line vanish.
Furthermore, the point is a singular point of the curve. Note that any curve in $\mathsf{PA}_3$ satisfies these
conditions.

We now show that every element of $\mathsf{PA}_3$ lies in the closure. We switch to affine space.
Let $f$ be a curve that has a $\mathsf{PA}_3$ point at the origin. Hence, the Taylor expansion of $f$ is given by
\begin{align*}
f(x,y) ={} &f_{11} xy + \frac{f_{02}}{2} y^2 + \frac{f_{21}}{2} x^2 y + \frac{f_{12}}{2} x y^2 + \frac{f_{03}}{6} y^3 \\
& +\frac{f_{40}}{24}x^4+ \frac{f_{31}}{6} x^3 y + \frac{f_{22}}{4} x^2 y^2 + \frac{f_{13}}{6} x y^3 +
\frac{f_{04}}{24} y^4 + \cdots.
\end{align*}
We now try to construct a nearby curve $f_t$ that has a nodal point at $(0, 0)$ and at $(t, 0)$. We also show that
every nearby curve is of the type we have constructed.

The Taylor expansion of $f_t$ is given by
\begin{align*}
f_t(x,y) ={}& \frac{f_{t_{20}}}{2} x^2 + f_{t_{11}} xy + \frac{f_{t_{02}}}{2} y^2 +
\frac{f_{t_{30}}}{6} x^3 + \frac{f_{t_{21}}}{2} x^2 y + \frac{f_{t_{12}}}{2} x y^2 + \frac{f_{t_{03}}}{6} y^3 \\
& + \frac{f_{t_{40}}}{24} x^4 +\frac{f_{t_{31}}}{6} x^3 y + \frac{f_{t_{22}}}{4} x^2 y^2 + \frac{f_{13}}{6} x y^3 +
\frac{f_{t_{04}}}{24} y^4 +
\cdots.
\end{align*}
We now impose the condition $f_t(t, 0) = 0$, $(f_{t})_x(t, 0) = 0$ and $(f_{t})_y(t, 0) = 0$.
Using the equation~${(f_{t})_y(t, 0) = 0}$, we can uniquely solve for $f_{t_{11}}$.
Next, using the equation $(f_{t})_x(t, 0) = 0$, we can uniquely solve for $f_{t_{20}}$.
Plugging these two solutions in the equation
$f_t(t, 0) = 0$, we can uniquely solve for $f_{t_{30}}$.

This gives us a procedure to construct a curve $f_t$, close to $f$,
that has a nodal point at $(0, 0)$ and at $(t, 0)$. Since our solution was unique, this implies that every
nearby curve is of this type, i.e., there is only one branch.

It remains to compute the multiplicity
of the intersection. We are basically setting the $x$-coordinate to be equal to zero. But the
$x$-coordinate is $t$. Hence, the order of vanishing is one (since there is exactly one branch).
This completes the proof of \eqref{PA3_new}.

\section{Counting rational curves}\label{count_st_mp_tang}
Finally, in this section,
we give an alternative approach to
enumerate stable maps with first-order tangency. Before getting into the details, the idea will be outlined.

First of all, it may be recalled that we are counting \textit{maps}, not zero sets of polynomials.
We will be considering the Kontsevich moduli space of maps of rational curves with two marked points, i.e., we will be doing
intersection theory on $\overline{M}_{0,2}\bigl(\mathbb{P}^2, d\bigr)$. Denote an element of
$\overline{M}_{0,2}\bigl(\mathbb{P}^2, d\bigr)$ as~${[u, y_1, y_2]}$. The letter $u$ denotes a
map from a possibly singular genus zero Riemann surface
to~$\mathbb{P}^2$ while $y_1$ and $y_2$ are two distinct points on the domain. The square bracket is there since we are
looking at equivalence classes of such maps. Denote $x_1 := u(y_1)$ and $x_2 := u(y_2)$. Note that~$x_1$ and $x_2$ are points on the
target space \big(i.e., $\mathbb{P}^2$\big).
Now consider the space of maps of rational curves with two marked points and a line, such that the image of the curve evaluated at those
two points intersects the line.
A
pictorial representation of an element of this space is
as follows:
\begin{figure}[h]\centering
\includegraphics[scale = .85]{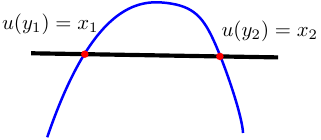}
\caption{Stable maps intersecting a line at two points.} \label{T1_A1L_st}
\end{figure}

On this space, now impose the condition that the points $x_1$ and $x_2$ become equal.
There are two possibilities now that can be pictorially seen as follows:
In the first case the corresponding points on the domain also become equal,
i.e., $y_1 = y_2$. This corresponds to the curve having a~tangency. In the second case, the corresponding points on the
domains are not the same, i.e., $y_1 \neq y_2$. This corresponds to the image of the curve having a self intersection, so the curve has a~node. Hence, imposing the condition $x_1 = x_2$, what we get can be summarized by the following picture:
\begin{center}\includegraphics[scale = .85]{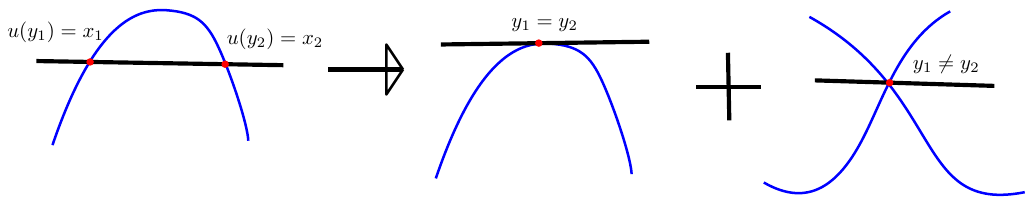}\end{center}

Using the results of \cite{AiM_m_fold_pt},
all intersection numbers involving the second term on the right-hand side of Figure \ref{T1_A1L_st} can be computed.
Hence, we can compute the characteristic number of rational curves tangent to a given line.
Sections \ref{count_st_mp_tang} and \ref{low_deg_st_map} contain a detailed computation along the above line.

Finally, in Section~\ref{Rel_GW_WDVV}, we pursue this idea again by making the points in the
\textit{domain} come together.
This is implemented by extending the idea behind the derivation of Kontsevich's recursion formula.
We choose a suitable subspace of the four pointed moduli space $\overline{M}_{0,4}\bigl(\mathbb{P}^2, d\bigr)$
and intersect it with the pullback of two divisors from $\overline{M}_{0,4}$.
Equating those two intersection numbers, gives a recursive formula for the characteristic number
of rational curves tangent to a~given line.

We now implement these ideas precisely. But first, we make a digression and review how Gathmann
enumerates rational curves tangent to a divisor.

\subsection{A review of Gathmann's approach to count curves with tangencies}
\label{Gathman_generalize}

In his papers \cite{AnGaPhD,AnGa1} and \cite{AnGa2}, Andreas Gathmann gives a systematic approach to solve the following question:
Let $Y$ be a hypersurface inside $\mathbb{P}^n$. What are the characteristic number of rational degree $d$ curves in
$\mathbb{P}^n$ that are tangent to $Y$ at a given point to order $k$? Gathmann successfully solves the above question
for any ample hypersurface and any $k$. He goes on to use this study to compute Gromov--Witten invariants of the quintic threefold.

Consider a special case of Gathmann's result when $Y$ is a line inside $\mathbb{P}^2$ and
ask the following question: How many rational degree $d$ curves are there in $\mathbb{P}^2$, that pass through $3d-2$
generic points and are tangent to a given line?
In this subsection we recapitulate Gathmann's approach to solve this question.
The next subsection describes an alternative approach to the question
based on applying Figure \ref{kk_pic} in the setting of stable maps
(which is also discussed in \cite[pp.~41]{Gath_blow_up} and \cite[pp.~1179--1180]{MK_published}).

Let us now describe Gathmann's idea. The setup is modified in order to solve a slightly
more general question. The question we solve is as follows:
How many pairs -- consisting of a line and a rational degree $d$ curve, passing through $m$ points and $n$
points respectively -- are there such that the line is tangent to the curve and $m+n = 3d+1$?
The special case of $m = 2$ corresponds to the line being fixed.

We start by describing the ambient space.
Recall that $\overline{M}_{0,0}\bigl(\mathbb{P}^2, d\bigr)$ is the compactification of the
Kontsevich moduli space of maps of rational curves (with no marked points). Let $\mathcal{H}$ denote the
divisor that corresponds to the subspace of curves that pass through a generic point. We note that the intersection number
$
\bigl[\overline{M}_{0,0}\bigl(\mathbb{P}^2, d\bigr)\bigr]\cdot \mathcal{H}^m
$
is computable via Kontsevich's recursion formula. For dimensional reasons,
the above number is nonzero only when $m = 3d-1$. On the zero pointed moduli space, this is the only intersection number
that is relevant for our purposes; it is also called a primary Gromov--Witten invariant.

Now consider $\overline{M}_{0,1}\bigl(\mathbb{P}^2, d\bigr)$, the one marked moduli space. As before, we have the
divisor $\mathcal{H}$ which corresponds to the subspace of curves whose image passes through a generic point.
But now, there are two other things as well.
Denote the pullback (via the evaluation map) of the hyperplane class in $\mathbb{P}^2$ by
${\rm ev}^*(b_1)$. Finally, consider
$\mathcal{L}\longrightarrow \overline{M}_{0,1}\bigl(\mathbb{P}^2, d\bigr)$,
the universal tangent bundle, whose fibre
over each point is the tangent space at that marked point. Denote the first Chern class of the dual of this bundle
by $\psi$, i.e.,
$
\psi :=  c_1(\mathcal{L}^*)$.
It is a~standard fact that all the intersection numbers
\begin{align}
\bigl[\overline{M}_{0,1}\bigl(\mathbb{P}^2, d\bigr)\bigr]\cdot \mathcal{H}^m \cdot {\rm ev}^*(b_1)^n \cdot \psi^{\theta} \label{primary_sec_GW}
\end{align}
are computable for any choice of $m$, $n$ and $\theta$.
This can be seen from the paper \cite[p.~311, Proposition 2.2]{Ionel_genus_one}.
When $\theta$ is greater than zero, the above number is also called a descendant Gromov--Witten invariant.

We now explain the geometric idea behind Gathmann's method to enumerate rational curves that are tangent to a fixed line,
and how to modify his method when the line is not fixed but is free to move in a family.
Denote by $\overline{M}_{0,k}\bigl(\mathbb{P}^2, d\bigr)$ the $k$ marked moduli space.
The pullback of the hyperplane classes (via the evaluation map) are denoted by
$
{\rm ev}^{*}(b_1^{n_1}), \dots, {\rm ev}^{*}(b_k^{n_k})$.
Now define $\mathsf{M}_1$ as
\begin{align*}
\mathsf{M}_1& :=  \mathcal{D}_1 \times \overline{M}_{0,1}\bigl(\mathbb{P}^2, d\bigr) \times \mathbb{P}^2_1,
\end{align*}
where $\mathbb{P}^2_i$ denotes a copy of $\mathbb{P}^2$
and $\mathcal{D}_1$ denotes the space of lines in $\mathbb{P}^2$.
The corresponding hyperplane classes are denoted by $a_i$ and $y_1$.

Note that an element of
$\mathsf{M}_1$ consists of a line, a one pointed rational curve (namely an element of
$\overline{M}_{0,1}\bigl(\mathbb{P}^2, d\bigr)$), and a point of $\mathbb{P}^2_1$. The relevant classes that live in
$\mathsf{M}_1$ are
$
y_1$, $\mathcal{H}$, ${\rm ev}^*(b_1)$, $\psi$, and $a_1$.
Since we can compute all the primary and descendant Gromov--Witten invariants
(i.e., the numbers in equation \eqref{primary_sec_GW}), we can compute all the following intersection numbers:
\begin{align}
[\mathsf{M}_1]\cdot \mathcal{H}^m \cdot {\rm ev}^*(b_1^{n_1})\cdot \psi^{\theta}\cdot y_1^{r}\cdot a_1^{s}.
\label{primary_sec_GW_M1}
\end{align}
Now define $(\mathsf{T}_0)_{\textnormal{st}}$ to be the following subspace of
$\mathsf{M}_1$:
\begin{align*}
(\mathsf{T}_0)_{\textnormal{st}}& :=  \{ ([f_1], [u, y_1], x_1) \in \mathsf{M}_1
\mid  u(y_1) = x_1,\,  f_1(x_1) = 0\}.
\end{align*}
\begin{rem}
 Note the following fact: we are typically going to denote the marked point of the domain by the letter $y_i$.
It is not going to cause any confusion with the other place where the letter $y_1$ is used, namely for the hyperplane class
of $\mathcal{D}_1$.
\end{rem}

Returning to the discussion, an element of $(\mathsf{T}_0)_{\textnormal{st}}$
can be pictorially described as follows:
\begin{center}\includegraphics[scale = .85]{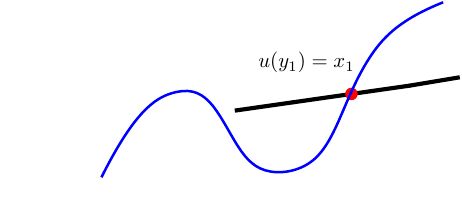}\end{center}

Let us now see how we can go about describing the class
$[(\mathsf{T}_0)_{\textnormal{st}}]$. First of all, note that the condition~${f_1(q_1)=0}$ is same as intersecting
with the class $(y_1+a_1)$ (see equation \eqref{incidence}). It remains to figure out how to express the condition $u(y_1) = q_1$.
Consider the map
\begin{align*}
{\rm ev}\times \textnormal{id}_{\mathbb{P}^2_{1}}\colon\ \mathsf{M}_1 &
\longrightarrow  \mathbb{P}^2 \times \mathbb{P}^2_{1},
  \qquad ([f_1], [u, y_1], x_1) \longmapsto (u(y_1), x_1).
\end{align*}
The condition $u(y_1) = x_1$ is same as intersecting with the pullback of the diagonal, namely
${
({\rm ev}\times \textnormal{id}_{\mathbb{P}^2_{1}})^*(\Delta_{b_1 a_1})}$.
Hence, we conclude that
\begin{align}
[(\mathsf{T}_0)_{\textnormal{st}}] &
 = \bigl({\rm ev}^*\bigl(b_1^2\bigr) + {\rm ev}^*(b_1) a_1 + a_1^2\bigr) \cdot (y_1 + a_1).
\label{T0_st}
\end{align}
Since all the intersection numbers in equation \eqref{primary_sec_GW_M1} are computable,
we conclude from equation~\eqref{T0_st} that
all the intersection numbers
\begin{align}
[(\mathsf{T}_0)_{\textnormal{st}}]\cdot \mathcal{H}^m \cdot {\rm ev}^*(b_1^{n_1})
\cdot \psi^{\theta}\cdot y_1^{r}\cdot a_1^{s} \label{T0_st_int_num}
\end{align}
are computable.

We now define $(\mathsf{T}_1)_{\textnormal{st}}$. It is the subspace of
$(\mathsf{T}_0)_{\textnormal{st}}$, where the curve is tangent to the line at the point $x_1$.
It is pictorially described as follows:
\begin{center}\includegraphics[scale = .9]{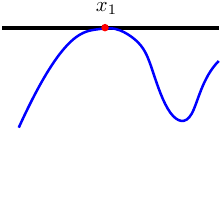}
\end{center}

Next, we explain Gathmann's approach to compute intersection numbers involving the
class $[(\mathsf{T}_1)_{\textnormal{st}}]$.
In the next subsection, we give an alternative approach to compute these intersection numbers.
Gathmann's approach is best summarized by the following equation:
\begin{align}
[(\mathsf{T}_1)_{\textnormal{st}}] & = [(\mathsf{T}_0)_{\textnormal{st}}]\cdot (\psi+y_1+a_1).
\label{T1_st_cycle}
\end{align}
We explain why equation \eqref{T1_st_cycle} is true. We note that for a rational curve $[u, y_1]$ to be tangent to the line, the differential
${\rm d}u|_{y_1}$ should take values in the tangent space of the line.
In other words, the differential ${\rm d}u|_{y_1}$ has to vanish in the normal direction of the line.
This condition is interpreted as the vanishing of a section of an appropriate line bundle.

First of all, consider the line bundle $\mathbb{L} \longrightarrow (\mathsf{T}_0)_{\textnormal{st}}$
whose fibre over each point $([f_1], [u, y_1], x_1)$ is the tangent space of the line $f_1^{-1}(0)$ at the point $x_1$.
This is basically the same line bundle we defined in Section~\ref{node_new}. For the convenience of the
reader, we review the definition,
namely the short exact sequence into which the line bundle fits
\begin{align}
0\longrightarrow & \mathbb{L} \longrightarrow T \mathbb{P}^2_{1}|_{x_1} \longrightarrow \gamma_{\mathcal{D}_1}^*
\otimes \gamma_{\mathbb{P}^2_{1}}^*\longrightarrow 0. \label{ses_line_bundle_ag2}
\end{align}
The condition that $[u, y_1]$ is tangent to the line at $x_1$ (namely that the differential ${\rm d}u|_{y_1}$ vanishes in the
normal direction to the line) can be interpreted as a section of the following line bundle~$
\mathcal{L}^*\otimes {\rm ev}^*\bigl( T \mathbb{P}^2_{1}/\mathbb{L}\bigr)$.
Using equation \eqref{ses_line_bundle_ag2}, we conclude that the Euler class of the above line bundle is equal to~${(\psi+ y_1 + a_1)}$
which gives us equation \eqref{T1_st_cycle}.

Note that using equations \eqref{T1_st_cycle}, \eqref{T0_st}, and the fact that all the intersection numbers of equation
\eqref{primary_sec_GW_M1} are computable, we conclude that all the following intersection numbers
$
[(\mathsf{T}_1)_{\textnormal{st}}]\cdot \mathcal{H}^m \cdot {\rm ev}^*(b_1^{n_1})
\cdot \psi^{\theta}\cdot y_1^{r}\cdot a_1^{s}
$
are computable.

This idea can be pushed further to enumerate rational curves with higher order tangency.
However, starting from second-order tangency, there is a non-trivial geometric phenomenon that occurs.
In the closure of curves tangent to a given line, there are
bubble maps. By a bubble map, we refer to the following stable maps $u \in \overline{M}_{0,n}\bigl(\mathbb{P}^2, d\bigr)$ of degree $d$ (see \cite[Section 5.1]{McSa}):
\begin{itemize}\itemsep=0pt
\item[$\bullet$] The domain has finite number of components $C_i$
each isomorphic to $\mathbb{P}^1$ and they are meeting each other at nodes.
\item[$\bullet$] Let $u_i = u|_{C_i}$ of degree $d_i$. Then the image of $u_i$ and $u_j$
intersect at a common point in~$\mathbb{P}^2$, which is the image of the nodal point and the total degree of $u$
is $d = \sum_{i} d_i$. If $u_i$ is constant for some $i$, then it is referred as ghost bubble or ghost component.
\end{itemize}
These bubble maps are in the zero locus
of the section that computes the second derivative. Hence, one has to analyse a degenerate locus and subtract off from the
Euler class. Gathmann does that successfully in his papers \cite{AnGaPhD,AnGa1,AnGa2}
and is able to enumerate rational curves tangent to any order.

\subsection{A new method to count rational curves with first-order tangency}\label{Rational-curve-with-a-choice-of-a-node}

In this subsection, we give an alternate approach to enumerate rational curves with tangencies based on
Figure \ref{kk_pic}. This idea has been discussed in
\cite{Gath_blow_up} (see p.~41) and \cite[pp.~1179--1180]{MK_published}.
The idea presented has an obvious difficulty (namely the formation of self intersection).
We overcome that difficulty using the result of our paper \cite{AiM_m_fold_pt}.

We continue with the setup of Section \ref{Gathman_generalize}.
We perform intersection theory on $\mathsf{M}_2$, which is defined as
\begin{align*}
\mathsf{M}_2 & := \mathcal{D}_1 \times \overline{M}_{0,2}\bigl(\mathbb{P}^2, d\bigr) \times \mathbb{P}^2_{1}\times \mathbb{P}^2_{2}.
\end{align*}
The relevant classes that live in
$\mathsf{M}_2$ are
$
y_1$, $  \mathcal{H}$, $  {\rm ev}^*(b_1)$, ${\rm ev}^*(b_2)$, $a_1$ and $a_1$.
Define the following projection map
$
\pi_{21} \colon \mathsf{M}_2 \longrightarrow \mathsf{M}_1$,given by $\pi_{21}([f_1], [u, y_1,
 y_2], x_1, x_2) := ([f_1], [u, y_1], x_1)$.
Basically, the map forgets the second marked point on
$\overline{M}_{0,2}\bigl(\mathbb{P}^2, d\bigr)$. Furthermore, it forgets the second factor
$\mathbb{P}^2_{2}$. The cycles that have been defined in the one pointed moduli space $\mathsf{M}_1$
can be pulled back to the two pointed moduli space $\mathsf{M}_2$ via this projection map.

Now define the following divisor in the two pointed moduli space
$\overline{M}_{0,2}\bigl(\mathbb{P}^2, d\bigr)$. We denote this divisor by the symbol
$
[ y_1 = y_2]$.
The above symbol allows us to guess the divisor we will be talking about: this divisor is represented by
the space of stable maps with two marked points, where the two marked points have coincided. It can be pictorially
represented as follows:
\begin{center}\includegraphics[scale = .9]{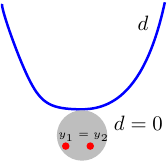}\end{center}

Usually, in the literature, it is denoted by $D(\{\varnothing\},{2};d,0)$ (see \cite{FuPa}).
Note that this space can be identified with the one marked moduli space
$\overline{M}_{0,1}\bigl(\mathbb{P}^2, d\bigr)$.
We now explain how to intersect with this divisor.
It is a standard fact that all the primary intersection numbers
\begin{align}
\bigl[\overline{M}_{0,2}\bigl(\mathbb{P}^2, d\bigr)\bigr]\cdot \mathcal{H}^{m}\cdot {\rm ev}^*(b_1^{n_1}) \cdot
{\rm ev}^*(b_2^{n_2}) \label{primary_GW_M2}
\end{align}
are computable.

Now define $\bigl(\mathsf{T}_0 \mathsf{T}_0\bigr)_{\textnormal{st}}$ to be the following subspace of $\mathsf{M}_2$:
it consists of a line and a rational curve and two marked points, where the two marked points lie on the line (see Figure~\ref{T1_A1L_st}).

It will be shown that the following equality of classes holds in $\mathsf{M}_2$:
\begin{gather}
\bigl({\rm ev}^*\bigl(b_1^2\bigr) + {\rm ev}^*(b_1) a_1 + a_1^2\bigr)\cdot (y_1 + a_1)
 \cdot \bigl({\rm ev}^*\bigl(b_2^2\bigr) + {\rm ev}^*(b_2) a_2 + a_2^2\bigr) \cdot (y_1 + a_2)\nonumber \\
\qquad = [(\mathsf{T}_0 \mathsf{T}_0)_{\textnormal{st}} ] +
[(\mathsf{T}_0)_{\textsf{st}}]. \label{T0_T0_st_map}
\end{gather}
A slight abuse of notation is being made here. Equation
\eqref{T0_T0_st_map} is an equality of classes in $\mathsf{M}_2$.
The second term on the right-hand side, namely $[(\mathsf{T}_0)_{\textsf{st}}]$ is a
class in
$\mathsf{M}_1$. However, we can identify that as a class in $\mathsf{M}_2$ via the following process: we attach a ghost bubble
at the marked point and add one more marked point on the ghost component.
This is the intended meaning of the second term of the right-hand side.

The reason for why equation \eqref{T0_T0_st_map} is true is similar to why
Theorem \ref{theorem_for_many_Tks} is true (for the special case $n = 1$ and $k_1 = 0$).
It has been shown in Section~\ref{Gathman_generalize} that intersecting with
\[\bigl({\rm ev}^*\bigl(b_1^2\bigr) + {\rm ev}^*(b_1) a_1 + a_1^2\bigr)\cdot (y_1 + a_1)\]
is imposing the condition that the first marked point lies on the line. Similarly, intersecting with
\[\bigl({\rm ev}^*\bigl(b_2^2\bigr) + {\rm ev}^*(b_2) a_2 + a_2^2\bigr) \cdot (y_1 + a_2)\]
imposes the condition that the second marked point lies on a line. Intersecting with both of them gives us the condition that
both the marked points lie on the line. However, the two marked points can coincide. That gives us the degenerate locus
corresponding to the second term on the right-hand side of equation \eqref{T0_T0_st_map}.
Rewriting equation \eqref{T0_T0_st_map}, one concludes that
\begin{align}
[(\mathsf{T}_0 \mathsf{T}_0)_{\textnormal{st}} ] = {}&
\bigl({\rm ev}^*\bigl(b_1^2\bigr) + {\rm ev}^*(b_1) a_1 + a_1^2\bigr)\cdot (y_1 + a_1) \nonumber \\
& \times \bigl({\rm ev}^*\bigl(b_2^2\bigr) + {\rm ev}^*(b_2) a_2 + a_2^2\bigr) \cdot (y_1 + a_2)
- [(\mathsf{T}_0)_{\textsf{st}}]. \label{T0_T0_st_map_ag}
\end{align}
Using equation \eqref{T0_T0_st_map_ag},
the fact that all the intersection numbers of equation \eqref{primary_GW_M2} are computable
and the fact that all the intersection numbers in equation \eqref{T0_st_int_num} are computable,
one concludes that all the intersection numbers
\begin{align}
[(\mathsf{T}_0 \mathsf{T}_0)_{\textnormal{st}} ]\cdot
\mathcal{H}^{m}\cdot {\rm ev}^*(b_1^{n_1}) \cdot
{\rm ev}^*(b_2^{n_2}) \cdot y_1^r \cdot a_1^{n_1} \cdot a_2^{n_2} \label{T0_T0_st_map_int_numbers}
\end{align}
are computable.

Next, define $\bigl(\mathsf{A}_1\bigr)_{\textnormal{st}}$ to be the following subspace of
the two pointed moduli space
$M_{0,2}\bigl(\mathbb{P}^2, d\bigr)$
\begin{align*}
(\mathsf{A}_1)_{\textnormal{st}} &:= \big\{ [u, y_1, y_2] \in
M_{0,2}\bigl(\mathbb{P}^2, d\bigr)\mid u(y_1)=u(y_2)\big\}.
\end{align*}
Also, define
$\bigl(\mathsf{A}_1^{\mathsf{L}}\bigr)_{\textnormal{st}}$ to be the following subspace of $\mathsf{M}_2$
\begin{align*}
\bigl(\mathsf{A}_1^{\mathsf{L}}\bigr)_{\textnormal{st}} :={}& \{ ([f_1], [u, y_1, y_2], x_1, x_2) \in
\mathsf{M}_2 \mid [u, y_1, y_2] \in (\mathsf{A}_1)_{\textnormal{st}}, \\
& u(y_1) = x_1,\, u(y_2) = x_2, \,f_1(x_1) = 0,\, f_1(x_2) = 0\}.
\end{align*}
Pictorially, the space $\bigl(\mathsf{A}_1^{\mathsf{L}}\bigr)_{\textnormal{st}}$ can be represented as follows:
\begin{center}\includegraphics[scale = .85]{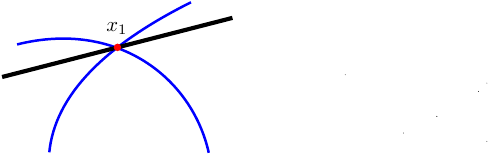}\end{center}

We now explain how to compute the characteristic number of rational curves with first-order tangency.
On the space \smash{$\overline{\bigl(\mathsf{T}_0 \mathsf{T}_0\bigr)}_{\textnormal{st}}$},
impose the additional condition that $x_1 = x_2$.
By the collision lemma,
this is same as intersecting with $(a_1 + a_2-y_1)$. Hence,
\begin{align}
[(\mathsf{T}_0 \mathsf{T}_0)_{\textnormal{st}} ] \cdot (a_1+a_2-y_1)
& = [(\mathsf{T}_1)_{\textsf{st}}]+
2 \bigl[\bigl(\mathsf{A}_1^{\mathsf{L}}\bigr)_{\textnormal{st}}\bigr]. \label{A1L_st_cycle}
\end{align}
Again, we are making an abuse of notation here. Equation \eqref{A1L_st_cycle} is an equality of
classes in~$\mathsf{M}_2$.
The first term on the right-hand side, namely $[(\mathsf{T}_1)_{\textsf{st}}]$ is a
class in
$\mathsf{M}_1$. However, we can identify that as a class in $\mathsf{M}_2$
via the following process: we attach a ghost bubble
at the marked point and add one more marked point on the ghost component.
This is the intended meaning of the first term of the right-hand side.

We now see why \eqref{A1L_st_cycle} is true. We note that when we impose the condition $x_1 = x_2$, there are two possibilities.
The first possibility is that the corresponding points in the domain are also the same and hence it corresponds to a point of tangency
as given by the following picture:
\begin{center}\includegraphics[scale = .85]{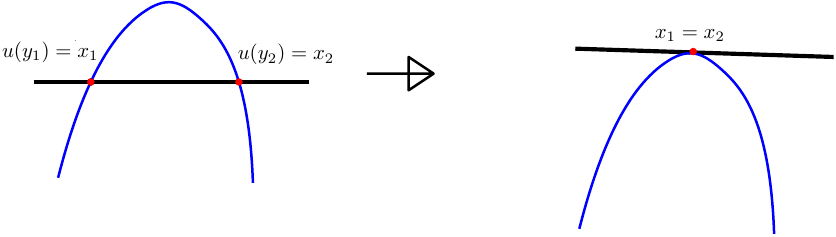}\end{center}

This corresponds to the first term of the right-hand side of equation \eqref{A1L_st_cycle}.
But there is a second possibility. The points in the domain can be distinct. This is a point of self
intersection, which is a nodal point. This corresponds to the second term in the right-hand side of
\eqref{A1L_st_cycle}. The intersection occurs with a multiplicity of $2$, because
if a line intersects a nodal point of a~curve, then it contributes $2$ to the intersection.
Denote $\mu$ as follows
$
\mu := \mathcal{H}^{m}\cdot {\rm ev}^*(b_1^{n_1}) \cdot y_1^r \cdot a_1^{n_1}$.
Rewriting equation \eqref{A1L_st_cycle} and intersecting with $\mu$, we conclude that
\begin{align}
[(\mathsf{T}_1)_{\textsf{st}}]\cdot \mu & =
[(\mathsf{T}_0 \mathsf{T}_0)_{\textnormal{st}} ] \cdot (a_1+a_2-y_1) \cdot \mu-
2 \bigl[\bigl(\mathsf{A}_1^{\mathsf{L}}\bigr)_{\textnormal{st}}\bigr] \cdot \mu. \label{A1L_st_cycle_ag}
\end{align}
Since all the intersection numbers in \eqref{T0_T0_st_map_int_numbers} are computable,
we conclude that
the first term on the right-hand side of \eqref{A1L_st_cycle_ag} is computable.
It will be shown in a moment that using the result of \cite{AiM_m_fold_pt}, one can also compute the second term.
Hence, the right-hand side of \eqref{A1L_st_cycle_ag} is computable for any $\mu$.
This gives us an alternative way to compute the characteristic number of rational curves tangent to a given line to first order.
The reader can refer to Section~\ref{low_deg_st_map}, where we have tabulated explicit numbers.

This idea can actually be pursued further. Suppose we wish to enumerate curves with second-order tangency.
Then as expected, we define the space $(\mathsf{T}_1 \mathsf{T}_0)_{\mathsf{st}}$. We then impose the
condition of the two points coming together. This results in two things. First of all we have curves with second-order tangency. We also encounter curves with a node lying on the line, such that the line is one of the
branches of the node. Hence, we need to compute that number. In order to do that, we define the space
$\bigl(\mathsf{A}_1^{\mathsf{L}} \mathsf{T}_0\bigr)_{\mathsf{st}}$ and require the points to come together. The two things
we get are, a nodal curve, with the line being one of the branches of the node. The second thing we get is a curve
with a triple point. To compute the latter, we can use the result of \cite{AiM_m_fold_pt}. We have
actually carried out this entire computation. The details of the computation are available on request.
However, the analysis of the degenerate locus and its contribution to the intersection is more non-trivial.
We hope to pursue this approach in a more thorough way in future
and figure out how to enumerate rational curves with $k$-th order tangency.

It remains to be shown how to compute the second term on the right-hand side of \eqref{A1L_st_cycle_ag}.
First of all, we note that intersecting $\bigl[\bigl(\mathsf{A}_1^{\mathsf{L}}\bigr)_{\textnormal{st}}\bigr]$ with
$a_1$ or ${\rm ev}^*(b_1)$ is the same thing.
Hence, it suffices to give a procedure to find intersection numbers with
$\mathcal{H}^{m} \cdot y_1^r \cdot a_1^{n}$.
Let $\alpha := (d, 2)$ and let~${N_{\alpha}(m,n)}$
denote the intersection number defined in \cite[pp.~5]{AiM_m_fold_pt}.
Here $m$ and $n$ are non-negative integers.
From the definition of $N_{\alpha}(m,n)$ and the proof of the
correspondence result, \cite[pp.~13--16]{AiM_m_fold_pt},
we conclude that
\begin{subnumcases}{\hspace*{-43mm}\big[\bigl(\mathsf{A}_1^{\mathsf{L}}\bigr)_{\textnormal{st}}\big] \cdot \mathcal{H}^{m} \cdot y_1^r \cdot a_1^{n} = }
0 & \textnormal{if}  $r = 0$, \label{base_case_rec1_pt5} \\
\textnormal{$N_{\alpha}(m, n)$} & \textnormal{if}  $r = 1$, \label{base_case_rec1_pt1} \\
\textnormal{$N_{\alpha}(m, 1)$} & \textnormal{if}  $r = 2$  ~\textnormal{and} $n = 0$, \label{base_case_rec1_pt4} \\
\textnormal{$N_{\alpha}(m, 2)$} & \textnormal{if}  $r = 2$  ~\textnormal{and} $n = 1$, \label{base_case_rec1_pt2}\\
0 & \textnormal{if}  $r = 2$  ~\textnormal{and} $n \geq 2$, \label{base_case_rec1_pt3} \\
 0 & \textnormal{if}  $r \geq 3$. \label{base_case_rec1_pt6}
\end{subnumcases}
Let us see why this is true. The first case, \eqref{base_case_rec1_pt5} follows from dimensional reasons.

We justify the next case, \eqref{base_case_rec1_pt1}.
From the proof of the correspondence result, \cite[pp.~13--16]{AiM_m_fold_pt},
we conclude that $N_{\alpha}(m, n)$ denotes the number of rational curves
passing through $m$ generic points and
with a choice of a node lying at the intersection of
$n$ generic lines. Once such a~nodal curve is fixed, a unique line passes through one point and the nodal point.
That precisely corresponds to the left-hand side of the equation.

Next, we justify \eqref{base_case_rec1_pt4}. We start by unwinding the left-hand side.
Intersecting with $y_1^2$ corresponds to fixing a line. Now we intersect with $\mathcal{H}^m$.
This is equal to the number of rational curves passing through $m$ generic points and
with a choice of a node lying on a line. That is precisely equal to $N_{\alpha}(m, 1)$; again this follows from the
proof of the correspondence result.

Next, let us justify \eqref{base_case_rec1_pt2}. Again, we begin by unwinding the left-hand side.
Intersecting with~$y_1^2$ corresponds to fixing a line. We recall that the nodal point lies on this line.
Intersecting with~$a_1$ is restricting the nodal point to lie on another line; in other words, we are restricting the
nodal point to lie on a fixed point. Now we intersect with $\mathcal{H}^m$.
This is equal to the number of rational curves passing through $m$ generic points and
with a choice of a node lying on a fixed point.
That is precisely equal to $N_{\alpha}(m, 2)$; again this follows from the
proof of the correspondence result. The next case, \eqref{base_case_rec1_pt3} follows immediately,
since intersecting with $y_1^2$ restricts the nodal point to lie on a line and intersecting with
$a_1^2$ is restricting the nodal point to lie on a point; since a line and a point do not intersect, this number is clearly zero.

Finally, \eqref{base_case_rec1_pt6} follows immediately since $y_1^3$ is zero.

\subsection{Another method to count rational curves with first-order tangency}\label{Rel_GW_WDVV}

We conclude this paper by illustrating that there is yet another way to exploit Figure \ref{kk_pic} to
enumerate stable maps with tangencies. Recall that in the previous subsection, we made the two points
in the \textit{images} come together. This results in two things; curves that are tangent to the line and
curves that have a point of self intersection lying on the line. The former occurs when the points in the
domain also coincide, while the latter occurs when the points in the domain remain distinct.

What if we could simply make the points in the domain come together? In this final part of the paper, we
do precisely that. The idea will be implemented by extending Kontsevich's idea to enumerate rational curves, i.e.,
the WDVV equation. The details are as follows.

Recall that Kontsevich's recursion formula gives an answer to the following question:
How many degree $d$ rational curves are there in $\mathbb{P}^2$ that pass through $3d-1$ generic points?
Denote this number by $n_d$. Also denote by $\mathsf{N}_d^{\mathsf{T}_1}$ the
number of rational degree $d$ curves in $\mathbb{P}^2$ that pass through $3d-2$ and that are tangent to a given line.
A recursive formula for $\mathsf{N}_d^{\mathsf{T}_1}$ will be obtained by using the WDVV equation.
Note that for simplicity, the line will be kept fixed; the more general case of keeping the line variable
can be worked out with very little extra effort (the only issue would be notational).

Consider $\overline{M}_{0,4}\bigl(\mathbb{P}^2, d\bigr)$, the moduli space of genus zero
stable maps, with $4$ marked points.
Let $\mathsf{L}$ be a fixed line inside $\mathbb{P}^2$.
Define $\mathsf{X}_d$ to be the following subspace of $M_{0,4}\bigl(\mathbb{P}^2, d\bigr)$:
it is the subspace of rational degree $d$ curves where
the image of the first two marked points lie on two distinct points of the line $\mathsf{L}$.
It is pictorially represented as follows:
\begin{center}\includegraphics[scale = .85]{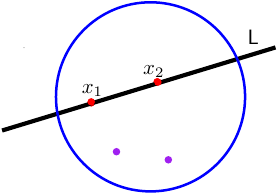}\end{center}

We denote $\overline{\mathsf{X}}_d$
to be the closure of $\mathsf{X}_d$ inside $\overline{M}_{0,4}\bigl(\mathbb{CP}^2, d\bigr)$.

Following the idea behind Kontsevich's recursion formula, we consider the forgetful map
\smash{$
\pi \colon \overline{M}_{0,4}\bigl(\mathbb{P}^2, d\bigr) \longrightarrow  \overline{M}_{0,4}$}.
Let
$[(ij|kl)]$
denote the divisor in $\overline{M}_{0,4}$ corresponding to the wedge of two spheres and where the
marked points $(y_i, y_j)$ lie on the one sphere and $(y_k, y_l)$ lie on the other sphere.
In $\overline{M}_{0,4}\bigl(\mathbb{P}^2, d\bigr)$,
define the class $\mathcal{Z}$ as
$
\mathcal{Z}:= {\rm ev}_3^*(\textnormal{pt})\cdot {\rm ev}_4^*(\textnormal{pt})\cdot \mathcal{H}^{3d-4}$.
Since~$\overline{M}_{0,4}$ is isomorphic to $\mathbb{P}^1$, any two points determine the same divisor.
Hence, $[(12|34)]$ is equal to~$[(13|24)]$ as divisors. Hence,
\begin{align}
\bigl[\overline{\mathsf{X}}_d\bigr]\cdot [\pi^*(12|34)] \cdot \mathcal{Z} & =
\bigl[\overline{\mathsf{X}}_d\bigr]\cdot [\pi^*(13|24)] \cdot \mathcal{Z}. \label{T1_WDVV}
\end{align}
Note that the left-hand side and right-hand side of \eqref{T1_WDVV}
denote intersection numbers in $\overline{M}_{0,4}\bigl(\mathbb{P}^2, d\bigr)$.
We now unravel both the sides of the equation and get a recursive formula.
Before that, let us recapitulate a terminology about bubble maps.
We define a bubble map to be of type $(d_1, d_2)$ if it is the following object:
a holomorphic map from a wedge of two spheres ($\mathbb{P}^1$), such that the map
is of degree $d_1$ on the first component and is of degree $d_2$ on the second component.
Holomorphic here means that restricted to each of the components, the map is holomorphic.

We unwind the left-hand side of \eqref{T1_WDVV} by looking at it geometrically.
First, consider the possibility when the two points $y_3$ and $y_4$ come together.
That results in a bubble map of type~$(d_1, d_2)$ such that
\begin{itemize}\itemsep=0pt
 \item The marked points $y_1$ and $y_2$ lie on the $d_1$ component.
 \item The marked points $y_3$ and $y_3$ lie on the $d_2$ component.
 \item The image of the marked points $y_1$ and $y_2$ intersect the line $\mathsf{L}$.
 \item The image of $y_3$ and $y_4$ coincide with two generic points (this corresponds to intersection with
 ${\rm ev}_3^*(\textnormal{pt})$ and ${\rm ev}_4^*(\textnormal{pt})$).
 \item The entire configuration passes through $3d-4$ points \big(this corresponds to intersecting with $\mathcal{H}^{3d-4}$\big).
\end{itemize}
The total number of such bubble maps is given by
\begin{align*}
\sum_{d_1+d_2 = d} \binom{3d-4}{3d_1-1} n_{d_1} n_{d_2} (d_1 d_2) d_1 (d_1 -1).
\end{align*}
The factor $d_1 d_2$ is to encode the number of choices for the bubble point. The factor $d_1(d_1-1)$
is there to encode the number of choices where the image of the points $y_1$ and then $y_2$ can lie.

We now see what happens when $y_1$ and $y_2$ come together.
What happens is that we get a~rational curve, tangent to the given line. The point of tangency is
given by the image of~$y_1$ (which is also the same as the image of~$y_2$).
Furthermore, the image of $y_3$ and $y_4$ coincide with two generic points
(which corresponds to intersection with
 ${\rm ev}_3^*(\textnormal{pt})$ and ${\rm ev}_4^*(\textnormal{pt})$)
and the entire configuration passes through $3d-4$ points
\big(which corresponds to intersecting with $\mathcal{H}^{3d-4}$\big). The total number of such objects is
precisely equal to the number of rational curves passing through~${3d-2}$ generic points, tangent to a given line.
That is precisely equal to $\mathsf{N}_d^{\mathsf{T}_1}$.
Hence, the left-hand side of \eqref{T1_WDVV} is
\begin{align}
\mathsf{N}_d^{\mathsf{T}_1}+\sum_{d_1+d_2 = d} \binom{3d-4}{3d_1-1} n_{d_1} n_{d_2} (d_1 d_2) d_1(d_1 -1). \label{1234_lhs}
\end{align}
Let us now unwind what is the
right-hand side of \eqref{T1_WDVV} by looking at it geometrically.
When the two points $y_1$ and $y_3$ come together (or $y_2$ and $y_4$ come together), a bubble
map of type~$(d_1, d_2)$ is obtained such that:
\begin{itemize}\itemsep=0pt
 \item The marked points $y_1$ and $y_3$ lie on the $d_1$ component.
 \item The marked points $y_2$ and $y_4$ lie on the $d_2$ component.
 \item The image of the marked point $y_1$ intersects the line $\mathsf{L}$.
\item The image of the marked point $y_2$ intersects the line $\mathsf{L}$.
\item The image of $y_3$ coincides with a generic points (this corresponds to intersection with ${\rm ev}_3^*(\textnormal{pt})$).
 \item The image of $y_4$ coincides with a generic points (this corresponds to intersection with
 ${\rm ev}_4^*(\textnormal{pt})$).
 \item The entire configuration passes through $3d-4$ points \big(this corresponds to intersecting with $\mathcal{H}^{3d-4}$\big).
\end{itemize}
The total number of such configurations is given by
\begin{align}
\sum_{d_1+d_2 = d} \binom{3d-4}{3d_1-2} n_{d_1} n_{d_2} (d_1 d_2) (d_1) (d_2). \label{1324_rhs}
\end{align}
This is precisely the right-hand side of \eqref{T1_WDVV}. Equating \eqref{1234_lhs} and \eqref{1324_rhs},
we conclude that
\begin{align}
\mathsf{N}_d^{\mathsf{T}_1}& =
\sum_{d_1+d_2 = d} \biggl(\binom{3d-4}{3d_1-2} d_1 d_2-\binom{3d-4}{3d_1-1}d_1(d_1-1)\biggr)
n_{d_1} n_{d_2} d_1 d_2.\label{NT1_WDVV_formula}
\end{align}

\section{Low degree checks}\label{verification}

In this section, we perform various non-trivial low degree checks which will be compared with the existing known results.
We remind the reader that the bound we impose on $d$ to obtain our results are sufficient
conditions for our formulas to be valid; they are not necessary.
Many of the numbers we have computed are obtained by applying the formula where the value of $d$
is lower than what is required to apply our theorem. Nevertheless, we display these values
and point out to the reader that the numbers we obtain at the end agree with the expected values
(obtained by other means).

\subsection[Counting smooth curves with tangencies: confirmation with Caporaso-Harris]{Counting smooth curves with tangencies:\\ confirmation with Caporaso--Harris}

In Section \ref{smooth_curves_enum}, it was shown that the following numbers
\smash{$
[\mathsf{T}_{k_1} \dots \mathsf{T}_{k_n}] \cdot y_1^2 y_d^{\delta_d-k}$},
can be computed. We remind the reader that $k:= k_1+\dots + k_n$.
A few values are displayed in the following table:
\begin{center}\renewcommand{\arraystretch}{1.2}
\begin{tabular}{|c|c|c|c|c|c|c|}
\hline
$d$ &$4$ &$5$ & $6$ & $7$ & $8$ & $9$ \\
\hline
$\mu$ &$y_1^2 y_4^{10}$ &$y_1^2 y_5^{16}$ & $y_1^2 y_6^{23}$ & $y_1^2 y_7^{31}$ & $y_1^2 y_8^{40}$ & $y_1^2 y_9^{50}$ \\
\hline
$\frac{1}{2}\bigl[\mathsf{T}_{1} \mathsf{T}_{1} \mathsf{T}_{2}\bigr] \cdot \mu$
& $0$ & $0$ & $0$ & $36$ & $144$ & $360$ \\
\hline
\end{tabular}
\captionof{table}{\textnormal{ }}
\label{IRCAA_tab}
\end{center}
Note that we have divided by a factor of $2$, because in the computation of
$\bigl[\mathsf{T}_{1} \mathsf{T}_{1} \mathsf{T}_{2}\bigr] \cdot \mu$ all the tangency points are ordered.
It is more natural to consider the two $\mathsf{T}_1$ points as unordered.

The corresponding nonzero numbers obtained from the Caporaso--Harris
by setting
\[\delta := 0, \qquad \alpha := (0) \qquad \textnormal{and} \qquad \beta := (i, 2, 1), \qquad \forall i = 0, 1, 2.\]
are as follows:
\begin{center}\renewcommand{\arraystretch}{1.2}
\begin{tabular}{|c|c|c|c|c|c|c|}
\hline
$d$ & $7$ & $8$ & $9$ \\
\hline
$N^{d,\delta}(\alpha, \beta)$
 & $36$ & $144$ & $360$ \\
\hline
\end{tabular}
\captionof{table}{\textnormal{ }}
\label{tab_CH}
\end{center}
We are assuming that the reader is familiar with the notation developed by
Caporaso--Harris in their paper \cite{CH}; we have followed that notation
in Table \ref{tab_CH}. Notice that the values in the last row of Tables \ref{IRCAA_tab} and \ref{tab_CH}
are in agreement.

Next, the value of the following numbers
\smash{$
[\mathsf{T}_{k_1} \dots \mathsf{T}_{k_n}] \cdot y_1^2 y_d^{\delta_d-k-1} a_1$},
will be displayed for certain values. This corresponds to fixing the line (intersecting with $y_1^2$)
and fixing the location of the first tangency point (intersecting with $a_1$). The remaining tangency points
are free. The values obtained by
using the theorems of Section \ref{smooth_curves_enum} are
\begin{center}\renewcommand{\arraystretch}{1.2}
\begin{tabular}{|c|c|c|c|c|c|c|}
\hline
$d$ & $7$ & $8$ & $9$ \\
\hline
$\mu$ & $y_1^2 y_7^{30}a_1$ & $y_1^2 y_8^{39}a_1$ & $y_1^2 y_9^{49}a_1$ \\
\hline
$\bigl[\mathsf{T}_{1} \mathsf{T}_{1} \mathsf{T}_{2}\bigr] \cdot \mu$ & $12$ & $36$ & $72$ \\
\hline
\end{tabular}
\captionof{table}{\textnormal{ }}\label{IRCAA_tab2}
\end{center}

Note that in this case, there is no significance of dividing by $2$, because the two $\mathsf{T}_1$ points
are different; the first $\mathsf{T}_1$ point lies on a fixed point, while the second $\mathsf{T}_1$ is free.

The corresponding numbers from Caporaso--Harris
by setting
\[\delta:= 0, \qquad \alpha:= (0,1) \qquad \textnormal{and} \qquad \beta:=(i,1,1),\qquad \forall i=0,1,2,\]
are as follows:
\begin{center}\renewcommand{\arraystretch}{1.2}
\begin{tabular}{|c|c|c|c|c|c|c|}
\hline
$d$ & $7$ & $8$ & $9$ \\
\hline
$N^{d,\delta}(\alpha, \beta)$ & $12$ & $36$ & $72$ \\
\hline
\end{tabular}
\captionof{table}{\textnormal{ }} \label{tab_CH2}
\end{center}
The values in the last row of Tables \ref{IRCAA_tab2} and \ref{tab_CH2}
are in agreement.

\subsection[Counting one nodal curves with tangencies: confirmation with Caporaso-Harris]{Counting one nodal curves with tangencies:\\ confirmation with Caporaso--Harris}

Using the theorems of Section \ref{one_nodal_tang_comp}, the following numbers
\smash{$
\bigl[\mathsf{A}_1^{\mathsf{F}}\mathsf{T}_{k_1} \dots \mathsf{T}_{k_n}\bigr] \cdot y_1^2 y_d^{\delta_d-k-1}$}.
can be computed. A few values are tabulated below:
\begin{center}\renewcommand{\arraystretch}{1.2}
\begin{tabular}{|c|c|c|c|c|c|c|}
\hline
$d$ & $7$ & $8$ \\
\hline
$\mu$ & $y_1^2 y_7^{30}$ & $y_1^2 y_8^{39}$ \\
\hline
$\frac{1}{2}\bigl[\mathsf{A}_1^{\mathsf{F}}\mathsf{T}_{1} \mathsf{T}_{1} \mathsf{T}_{2}\bigr] \cdot \mu$
& $3420$ & $19404$ \\
\hline
\end{tabular}
\captionof{table}{\textnormal{ }} \label{IRCAA_tab3}
\end{center}
The corresponding numbers from Caporaso--Harris
by setting
\begin{gather*}
\delta:=1, \qquad \alpha:= (0),  \qquad \beta:=(0,2,1) \quad \textnormal{and} \\
\delta:=1, \qquad \alpha:= (0),  \qquad \beta:=(1,2,1)
\end{gather*}
are as follows:{\samepage
\begin{center}\renewcommand{\arraystretch}{1.2}
\begin{tabular}{|c|c|c|c|c|c|c|}
\hline
$d$ & $7$ & $8$ \\
\hline
$N^{d,\delta}(\alpha, \beta)$ & $3420$ & $19404$ \\
\hline
\end{tabular}
\captionof{table}{\textnormal{ }} \label{tab_CH3}
\end{center}
The values in the last row of Tables \ref{IRCAA_tab3} and \ref{tab_CH3}
are in agreement.}

Next, the value of the following numbers
\smash{$
\bigl[\mathsf{A}_1^{\mathsf{F}}\mathsf{T}_{k_1} \dots \mathsf{T}_{k_n}\bigr] \cdot y_1^2 y_d^{\delta_d-k-2} a_1$},
will be displayed for certain values. This corresponds to fixing the line (intersecting with $y_1^2$)
and fixing the location of the first tangency point (intersecting with $a_1$). The remaining tangency points
are free.
The values obtained by
using the theorems of Section \ref{smooth_curves_enum} are
\begin{center}\renewcommand{\arraystretch}{1.2}
\begin{tabular}{|c|c|c|c|c|c|c|}
\hline
$d$ & $8$ \\
\hline
$\mu$ & $y_1^2 y_8^{38}a_1$ \\
\hline
$\bigl[\mathsf{A}_1^{\mathsf{F}}\mathsf{T}_{1} \mathsf{T}_{1} \mathsf{T}_{2}\bigr] \cdot \mu$ & $4912$ \\
\hline
\end{tabular}
\captionof{table}{\textnormal{ }} \label{IRCAA_tab4}
\end{center}
The corresponding numbers from Caporaso--Harris
by setting
\[\delta:=1, \qquad \alpha:= (0,1) \qquad \textnormal{and} \qquad \beta:=(1,1,1)\]
are as follows:
\begin{center}\renewcommand{\arraystretch}{1.2}
\begin{tabular}{|c|c|c|c|c|c|c|}
\hline
$d$ & $8$ \\
\hline
$N^{d,\delta}(\alpha, \beta)$
 & $4912$ \\
\hline
\end{tabular}
\captionof{table}{\textnormal{ }} \label{tab_CH4}
\end{center}
The values
in the last row of Tables \ref{IRCAA_tab4} and \ref{tab_CH4}
are in agreement.

\subsection[Counting one cuspidal cubics with tangencies: confirmation with Ernstr\"om-Kennedy]{Counting one cuspidal cubics with tangencies:\\ confirmation with Ernstr\"om--Kennedy}

Using the theorems of Section \ref{Cuspidal_tang}, one obtains that for $d = 3$,
\smash{$
\bigl[\mathsf{A}_2^{\mathsf{F}} \mathsf{T}_1\bigr]\cdot y_1^2 y_3^{6}   =  60$}.
This number is equal to the number of rational cuspidal cubics passing through
$6$ generic points that
is tangent to a given line. This is in agreement with
the answer obtained by Ernstr{\"o}m and Kennedy in \cite{ken}.

\subsection{Counting stable maps with tangencies: confirmation with Gathmann}\label{low_deg_st_map}

Using the results of Section \ref{Rational-curve-with-a-choice-of-a-node}, the
following numbers
$
[(\mathsf{T}_1)_{\textnormal{st}}]\cdot y_1^2 \cdot \mathcal{H}^{3d-2-n} a_1^{n}
$
can be computed. A few numbers are tabulated in the following two tables:
\begin{center}\renewcommand{\arraystretch}{1.2}
\begin{tabular}{|c|c|c|c|c|c|c|}
\hline
$d$ & $3$ & $4$ & $5$ & $6$ & $7$ & $8$ \\
\hline
$\mu$ & $y_1^2 \mathcal{H}^7$ & $y_1^2 \mathcal{H}^{10}$ & $y_1^2 \mathcal{H}^{13}$ & $y_1^2 \mathcal{H}^{16}$ &
$y_1^2 \mathcal{H}^{19}$ & $y_1^2 \mathcal{H}^{22}$ \\
\hline
$\bigl[(\mathsf{T}_1)_{\textnormal{st}}\bigr] \cdot \mu$ & $36$ & $2184$ & $335792$ & $106976160$ & $61739450304$
& $58749399019136$ \\
\hline
\end{tabular}
\captionof{table}{\textnormal{ }} \label{IRCAA_tab5}
\end{center}
\begin{center}\renewcommand{\arraystretch}{1.2}
\begin{tabular}{|c|c|c|c|c|c|c|}
\hline
$d$ & $3$ & $4$ & $5$ & $6$ & $7$ & $8$ \\
\hline
$\mu$ & $y_1^2 \mathcal{H}^6 a_1$ & $y_1^2 \mathcal{H}^{9}a_1$ & $y_1^2 \mathcal{H}^{12}a_1$ & $y_1^2 \mathcal{H}^{15}a_1$ &
$y_1^2 \mathcal{H}^{18}a_1$ & $y_1^2 \mathcal{H}^{21}a_1$ \\
\hline
$\bigl[(\mathsf{T}_1)_{\textnormal{st}}\bigr] \cdot \mu$ & $10$ & $428$ & $51040$ & $13300176$ & $6498076192$
& $5362556317120$ \\
\hline
\end{tabular}
\captionof{table}{\textnormal{ }} \label{IRCAA_tab6}
\end{center}
These are all in agreement with the numbers computed by Gathmann's
program GROWI
(that implements the formulas in \cite{AnGaPhD,AnGa1,AnGa2}).

Finally, Gathmann's approach is extend in Section~\ref{Gathman_generalize} where the line can be varied.
In particular, the following intersection numbers can be computed
\smash{$
 [(\mathsf{T}_1)_{\textnormal{st}} ]\cdot y_1^r \cdot \mathcal{H}^{3d-n-r}\cdot a_1^{n}$}.
A~few numbers are tabulated in the following two tables:
\begin{center}\renewcommand{\arraystretch}{1.2}
\begin{tabular}{|c|c|c|c|c|c|c|}
\hline
$d$ & $3$ & $4$ & $5$ & $6$ & $7$ & $8$ \\
\hline
$\mu$ & $y_1 \mathcal{H}^8$ & $y_1 \mathcal{H}^{11}$ & $y_1 \mathcal{H}^{14}$ & $y_1 \mathcal{H}^{17}$ &
$y_1\mathcal{H}^{20}$ & $y_1 \mathcal{H}^{23}$ \\
\hline
$\bigl[(\mathsf{T}_1)_{\textnormal{st}}\bigr] \cdot \mu$ & $48$ & $3720$ & $698432$ & $263129760$ & $175401698304$
& $189360514383488$ \\
\hline
\end{tabular}
\captionof{table}{\textnormal{ }} \label{IRCAA_tab7}
\end{center}
These values are all in agreement with the alternative approach given in Section~\ref{Rational-curve-with-a-choice-of-a-node}
to compute these intersection numbers.

Finally, for $d=3$ one can compute (using the theorems of Section \ref{one_nodal_tang_comp}) that
\smash{$
\bigl[\mathsf{A}_1^{\mathsf{F}} \mathsf{T}_1\bigr]\cdot y_1 y_3^{8}  = 48$}.
This number is in agreement with the first number tabulated in Table \ref{IRCAA_tab7}.

Finally, we tabulate the values of $\mathsf{N}_d^{\mathsf{T}_1}$ using \eqref{NT1_WDVV_formula}:
\begin{center}\renewcommand{\arraystretch}{1.2}
\begin{tabular}{|c|c|c|c|c|c|c|c|}
\hline
$d$& $3$ & $4$ & $5$ & $6$ & $7$ & $8$\\
\hline
$\mathsf{N}_d^{\mathsf{T}_1}$ & $36$ & $2184$ & $335792$ & $106976160$ & $61739450304$
& $58749399019136$\\
\hline
\end{tabular}
\captionof{table}{\textnormal{ }} \label{tab_NT1_WDVV}
\end{center}
These are all in agreement with the numbers computed by Gathmann's
program GROWI.

\subsection*{Acknowledgements}

We are very grateful to the referees for giving us constructive
and detailed comments on the earlier version of the manuscript.
We are grateful to Chitrabhanu Chaudhuri for several useful discussions related to this paper. We also thank Soumya Pal for
writing a \textsc{Python} program to implement Caporaso--Harris formula for verification.
 The first author is partially supported by a J.C.~Bose Fellowship (JBR/2023/000003). The second author is
funded by the Deutsche Forschungsgemeinschaft (DFG, German Research
Foundation) under Germany's Excellence Strategy – The Berlin Mathematics
Research Center MATH+ (EXC-2046/1, project ID: 390685689).
The fourth author would like to acknowledge the support of the Department of Atomic Energy, Government of India, under project no. RTI4001.

\pdfbookmark[1]{References}{ref}
\LastPageEnding


\begin{thebibliography}{99}
\footnotesize\itemsep=0pt

\bibitem{R.M}
Basu S., Mukherjee R., Enumeration of curves with one singular point,
 \href{https://doi.org/10.1016/j.geomphys.2016.02.008}{\textit{J.~Geom. Phys.}} \textbf{104} (2016), 175--203, \href{https://arxiv.org/abs/1308.2902}{arXiv:1308.2902}.

\bibitem{AiM_m_fold_pt}
Biswas I., Chaudhuri C., Choudhury A., Mukherjee R., Paul A., Counting rational
 curves with an {$m$}-fold point, \href{https://doi.org/10.1016/j.aim.2023.109258}{\textit{Adv. Math.}} \textbf{431} (2023),
 109258, 21~pages, \href{https://arxiv.org/abs/2212.01664}{arXiv:2212.01664}.

\bibitem{IB_RM_AC_ND_arxiv}
Biswas I., Choudhury A., Das N., Mukherjee R., On a fibre bundle version of the
 {C}aporaso--{H}arris formula, \href{https://arxiv.org/abs/2407.14832}{arXiv:2407.14832}.

\bibitem{CH}
Caporaso L., Harris J., Counting plane curves of any genus,  \href{https://doi.org/10.1007/s002220050208}{\textit{Invent.
 Math.}} \textbf{131} (1998), 345--392, \href{https://arxiv.org/abs/alg-geom/9608025}{arXiv:alg-geom/9608025}.

\bibitem{ken}
Ernstr\"om L., Kennedy G., Recursive formulas for the characteristic numbers of
 rational plane curves, \textit{J.~Algebraic Geom.} \textbf{7} (1998),
 141--181, \href{https://arxiv.org/abs/alg-geom/9604019}{arXiv:alg-geom/9604019}.

\bibitem{FuPa}
Fulton W., Pandharipande R., Notes on stable maps and quantum cohomology, in
 Algebraic {G}eometry~-- {S}anta {C}ruz 1995, \textit{Proc. Sympos. Pure
 Math.}, Vol.~62, \href{https://doi.org/10.1090/pspum/062.2/1492534}{American Mathematical Society}, Providence, RI, 1997, 45--96,
 \href{https://arxiv.org/abs/alg-geom/9608011}{arXiv:alg-geom/9608011}.

\bibitem{AnGaPhD}
Gathmann A., {G}romov--{W}itten invariants of hypersurfaces,
 {H}abilitationsschrift, {T}echnischen {U}niversit\"at {K}aiserslautern, 2003.

\bibitem{Gath_blow_up}
Gathmann A., Gromov--{W}itten invariants of blow-ups, \textit{J.~Algebraic
 Geom.} \textbf{10} (2001), 399--432, \href{https://arxiv.org/abs/math.AG/9804043}{arXiv:math.AG/9804043}.

\bibitem{AnGa1}
Gathmann A., Relative {G}romov--{W}itten invariants and the mirror formula,
 \href{https://doi.org/10.1007/s00208-002-0345-1}{\textit{Math. Ann.}} \textbf{325} (2003), 393--412, \href{https://arxiv.org/abs/math.AG/0009190}{arXiv:math.AG/0009190}.

\bibitem{AnGa2}
Gathmann A., The number of plane conics that are five-fold tangent to a given
 curve, \href{https://doi.org/10.1112/S0010437X04001083}{\textit{Compos. Math.}} \textbf{141} (2005), 487--501,
 \href{https://arxiv.org/abs/math.AG/0202002}{arXiv:math.AG/0202002}.

\bibitem{GH3}
Griffiths P., Harris J., Principles of algebraic geometry, \textit{Wiley Classics Lib.},
  \href{https://doi.org/10.1002/9781118032527}{John Wiley \& Sons}, Inc., New York, 1994.

\bibitem{Ionel_genus_one}
Ionel E.-N., Genus one enumerative invariants in {$\mathbb P^2$} with fixed $j$
 invariant, \href{https://doi.org/10.1215/S0012-7094-98-09414-5}{\textit{Duke Math.~J.}} \textbf{94} (1998), 279--324,
 \href{https://arxiv.org/abs/alg-geom/9608030}{arXiv:alg-geom/9608030}.

\bibitem{McSa}
McDuff D., Salamon D., {$J$}-holomorphic curves and symplectic topology, 2nd ed.,
 \textit{Amer. Math. Soc. Colloq. Publ.}, Vol.~52, American
 Mathematical Society, Providence, RI, 2012.

\bibitem{MK_published}
McDuff D., Siegel K., Counting curves with local tangency constraints,
 \href{https://doi.org/10.1112/topo.12204}{\textit{J.~Topol.}} \textbf{14} (2021), 1176--1242, \href{https://arxiv.org/abs/1906.02394}{arXiv:1906.02394}.

\bibitem{Anant-Thesis}
Paul A., Enumeration of singular curves with prescribed tangencies, Ph.D.~Thesis, {N}ational {I}nstitue of Science Education and Research, 2021, available at \url{https://www.proquest.com/docview/3103007951}.

\bibitem{PAUL2024103418}
Paul A., Counting singular curves with prescribed tangency,  \href{https://doi.org/10.1016/j.bulsci.2024.103418}{\textit{Bull. Sci.
 Math.}} \textbf{192} (2024), 103418, 27~pages.

\end{thebibliography}
\end{document}